\documentclass[11pt,a4paper]{article}

\usepackage[a4paper,margin=2.4cm]{geometry}
\usepackage[T1]{fontenc}
\usepackage[utf8]{inputenc}
\usepackage{lmodern}
\usepackage{microtype}
\usepackage{parskip}
\usepackage{amsmath,amssymb,amsthm,amsfonts}
\usepackage{mathtools}
\usepackage{bm}

\usepackage{graphicx}
\usepackage{float}            % [H] placement

\usepackage{underscore}        % allow _ in text safely (e.g., filenames)
\let\origincludegraphics\includegraphics
\renewcommand{\includegraphics}[2][width=0.6\linewidth]{%
  \IfFileExists{#2}%
    {\origincludegraphics[#1]{#2}}%
    {\IfFileExists{#2.pdf}%
       {\origincludegraphics[#1]{#2}}%
       {\IfFileExists{#2.png}%
          {\origincludegraphics[#1]{#2}}%
          {\IfFileExists{#2.jpg}%
             {\origincludegraphics[#1]{#2}}%
             {\IfFileExists{#2.jpeg}%
                {\origincludegraphics[#1]{#2}}%
                {\IfFileExists{#2.eps}%
                   {\origincludegraphics[#1]{#2}}%
                   {\fbox{\begin{minipage}[c][2.2cm][c]{0.85\linewidth}\centering\footnotesize
                      \textsf{[figure placeholder]}\\[2pt]\texttt{\small #2}
                      \end{minipage}}}}}}}}%
}
\usepackage{pifont}
\newcommand{\cmark}{\ding{51}}
\newcommand{\xmark}{\ding{55}}
\usepackage{caption}
\usepackage{subcaption}
\usepackage{array}
\usepackage{booktabs}
\usepackage{multirow}
\usepackage{makecell}

\usepackage[dvipsnames,table]{xcolor}

\definecolor{ddmGray}{HTML}{ECECEC}      % neutral background (svgraybox)
\definecolor{ddmGrayRule}{HTML}{8A8A8A}
\definecolor{ddmBlue}{HTML}{1F6FB2}      % important / key insight
\definecolor{ddmBlueBg}{HTML}{E8F1FA}
\definecolor{ddmAmber}{HTML}{C8881B}     % background information
\definecolor{ddmAmberBg}{HTML}{FCF3E0}
\definecolor{ddmGreen}{HTML}{2E8B57}     % trailer / algebraic insight
\definecolor{ddmGreenBg}{HTML}{E6F4EC}
\definecolor{ddmPurple}{HTML}{7A4FBF}    % tips / where Schwarz enters
\definecolor{ddmPurpleBg}{HTML}{F1EBFB}
\definecolor{ddmRed}{HTML}{B0413E}       % overview / summary
\definecolor{ddmRedBg}{HTML}{FBEDEC}
\definecolor{ddmCode}{HTML}{2E2E2E}

\usepackage{url}
\usepackage[colorlinks=true,
            linkcolor=ddmBlue,
            citecolor=ddmGreen,
            urlcolor=ddmPurple,
            bookmarksnumbered=true,
            breaklinks=true]{hyperref}
\usepackage[capitalise]{cleveref}

\theoremstyle{plain}
\newtheorem{theorem}{Theorem}[section]
\newtheorem{lemma}[theorem]{Lemma}
\newtheorem{proposition}[theorem]{Proposition}

\theoremstyle{definition}
\newtheorem{definition}[theorem]{Definition}

\theoremstyle{remark}
\newtheorem{remark}[theorem]{Remark}

\usepackage{algorithm}
\usepackage{algorithmic}

\usepackage{listings}

\lstdefinestyle{ddmcode}{
  basicstyle=\ttfamily\small,
  keywordstyle=\color{ddmBlue}\bfseries,
  commentstyle=\color{ddmGreen}\itshape,
  stringstyle=\color{ddmRed},
  numbers=left,
  numberstyle=\tiny\color{ddmGrayRule},
  numbersep=8pt,
  frame=single,
  rulecolor=\color{ddmGrayRule},
  backgroundcolor=\color{white},
  showstringspaces=false,
  breaklines=true,
  breakatwhitespace=true,
  columns=fullflexible,
  keepspaces=true,
  tabsize=2,
  xleftmargin=0pt,
  framexleftmargin=0pt,
  morekeywords={macro,mesh,square,include,varf,int2d,int3d,on,dx,dy,dz,
                fespace,solve,problem,real,complex,func,for,if,else,while,
                cout,endl,return},
  literate={->}{{$\rightarrow$}}1
}
\lstdefinelanguage{FreeFEM}[]{C++}{
  morekeywords={macro,mesh,square,include,varf,int2d,int3d,on,dx,dy,dz,
                fespace,solve,problem,real,complex,func}
}
\lstnewenvironment{ddmlisting}[1][]
  {\lstset{style=ddmcode,language=FreeFEM,#1}}
  {}
\lstnewenvironment{ddmshell}[1][]
  {\lstset{style=ddmcode,language=bash,numbers=none,
           keywordstyle=\color{ddmBlue}\bfseries,#1}}
  {}
\lstnewenvironment{ddmoutput}[1][]
  {\lstset{basicstyle=\ttfamily\footnotesize,
           numbers=none,frame=single,rulecolor=\color{ddmGrayRule},
           backgroundcolor=\color{ddmGray!40},
           showstringspaces=false,breaklines=false,columns=fullflexible,
           keepspaces=true,#1}}
  {}

\usepackage[most]{tcolorbox}
\tcbuselibrary{skins,breakable}

\newtcolorbox{svgraybox}{
  enhanced, breakable,
  colback=ddmGray!55, colframe=ddmGrayRule,
  boxrule=0.4pt, arc=2pt,
  left=8pt, right=8pt, top=6pt, bottom=6pt,
  before skip=8pt, after skip=8pt
}

\newtcolorbox{important}[1]{
  enhanced, breakable,
  colback=ddmBlueBg, colframe=ddmBlue,
  boxrule=0.6pt, arc=3pt,
  left=10pt, right=8pt, top=6pt, bottom=6pt,
  before skip=10pt, after skip=10pt,
  title={\textcolor{white}{\textbf{$\rhd$\;\,#1}}},
  coltitle=white,
  attach boxed title to top left={xshift=8pt, yshift*=-2pt},
  boxed title style={colback=ddmBlue, colframe=ddmBlue,
                     boxrule=0pt, arc=2pt, left=4pt, right=4pt, top=2pt, bottom=2pt}
}

\newtcolorbox{backgroundinformation}[1]{
  enhanced, breakable,
  colback=ddmAmberBg, colframe=ddmAmber,
  boxrule=0.6pt, arc=3pt,
  left=10pt, right=8pt, top=6pt, bottom=6pt,
  before skip=10pt, after skip=10pt,
  title={\textcolor{white}{\textbf{\textsf{#1}}}},
  coltitle=white,
  attach boxed title to top left={xshift=8pt, yshift*=-2pt},
  boxed title style={colback=ddmAmber, colframe=ddmAmber,
                     boxrule=0pt, arc=2pt, left=4pt, right=4pt, top=2pt, bottom=2pt}
}

\newtcolorbox{trailer}[1]{
  enhanced, breakable,
  colback=ddmGreenBg, colframe=ddmGreen,
  boxrule=0.6pt, arc=3pt,
  left=10pt, right=8pt, top=6pt, bottom=6pt,
  before skip=10pt, after skip=10pt,
  title={\textcolor{white}{\textbf{$\bigstar$\;\,#1}}},
  coltitle=white,
  attach boxed title to top left={xshift=8pt, yshift*=-2pt},
  boxed title style={colback=ddmGreen, colframe=ddmGreen,
                     boxrule=0pt, arc=2pt, left=4pt, right=4pt, top=2pt, bottom=2pt}
}

\newtcolorbox{tips}[1]{
  enhanced, breakable,
  colback=ddmPurpleBg, colframe=ddmPurple,
  boxrule=0.6pt, arc=3pt,
  left=10pt, right=8pt, top=6pt, bottom=6pt,
  before skip=10pt, after skip=10pt,
  title={\textcolor{white}{\textbf{\faLightbulbO\,\,#1}}},
  coltitle=white,
  attach boxed title to top left={xshift=8pt, yshift*=-2pt},
  boxed title style={colback=ddmPurple, colframe=ddmPurple,
                     boxrule=0pt, arc=2pt, left=4pt, right=4pt, top=2pt, bottom=2pt}
}

\newtcolorbox{overview}[1]{
  enhanced, breakable,
  colback=ddmRedBg, colframe=ddmRed,
  boxrule=0.6pt, arc=3pt,
  left=10pt, right=8pt, top=6pt, bottom=6pt,
  before skip=10pt, after skip=10pt,
  title={\textcolor{white}{\textbf{\textsf{#1}}}},
  coltitle=white,
  attach boxed title to top left={xshift=8pt, yshift*=-2pt},
  boxed title style={colback=ddmRed, colframe=ddmRed,
                     boxrule=0pt, arc=2pt, left=4pt, right=4pt, top=2pt, bottom=2pt}
}

\newcommand{\add}[1]{#1}

\newcommand{\cmd}[1]{\textcolor{ddmBlue}{\texttt{#1}}}
\newcommand{\cmdv}[1]{\textcolor{ddmRed}{\texttt{#1}}}
\newcommand{\cmdp}[1]{\textcolor{ddmGreen}{\texttt{#1}}}
\newcommand{\cmdpv}[1]{\textcolor{ddmPurple}{\texttt{#1}}}
\newcommand{\ff}[1]{\textcolor{ddmCode}{\texttt{#1}}}
\newcommand{\sh}[1]{\textcolor{ddmCode}{\texttt{#1}}}

\newcommand{\svhline}{\hline}

\usepackage{tikz}
\usetikzlibrary{arrows,arrows.meta,positioning,calc,decorations.pathreplacing,
                shapes,shapes.geometric,fit,backgrounds}

\providecommand{\faLightbulbO}{$\bullet$}

\usepackage{enumitem}
\setlist{leftmargin=1.4em,topsep=0.3em,itemsep=0.15em}

\allowdisplaybreaks
\title{\textbf{A Guided Tour of Modern Domain Decomposition}\\[0.35em]
\large From Schwarz Iterations to Robust Preconditioners\\ and HPC Implementations}

\author{Victorita Dolean \and Pierre Jolivet \and Frédéric Nataf \and Pierre-Henri Tournier}

\date{\today}

\begin{document}
\maketitle

\begin{center}
\begin{minipage}{0.78\textwidth}
\itshape\small
\hfill Divide each of the difficulties under examination into as many
\flushright parts as possible, and as might be necessary for its adequate solution.\\[0.8em]
\hfill \upshape\textsc{--- René Descartes}, \emph{Discourse of the Method}
\end{minipage}
\end{center}

\vspace{1em}

\begin{abstract}
Domain decomposition methods (DDMs) provide a unifying framework for the scalable numerical solution of partial differential equations. Originating from Schwarz’s alternating method, they have evolved into a rich family of algorithms that combine local robustness with global convergence acceleration and natural parallelism. Over the past decades, domain decomposition has played a central role in enabling large-scale simulations in numerous applications.
This chapter presents an overview of modern DDMs, with a particular emphasis on scalable preconditioning techniques for challenging problems, including indefinite and high-frequency regimes. We revisit the fundamental concepts -  overlapping decompositions, partition of unity, additive and restricted Schwarz formulations -  and explain their algebraic interpretations. We then clarify their role as preconditioners in Krylov subspace solvers and discuss the necessity of coarse space corrections for scalability.
Beyond a the survey aspect, the chapter distills key theoretical insights and practical design principles that have emerged over the past twenty years. Special attention is given to robust coarse spaces (GenEO, DtN-based approaches) and high-performance implementations. The goal is to provide both a coherent overview of the field and a concise, practice-oriented guide for readers seeking to understand and apply domain decomposition methods without navigating the entire literature.
\end{abstract}

\tableofcontents
\vspace{1em}
\hrule
\vspace{1em}

\section{Basic elements of domain decomposition}
\label{sec:intro_motivation_ddm}

Efficient and reliable solvers for partial differential equations (PDEs) are at the core of modern computational science and engineering. Applications such as medical imaging, electronics, seismic exploration, and aeronautics all require the numerical solution of PDE models at increasingly fine resolutions, leading to a higher physical fidelity. After discretization, these models typically produce very large sparse linear systems whose solution cost dominates the overall simulation or inversion pipeline. This chapter motivates domain decomposition methods (DDM) \cite{Dolean2015, ToselliWidlund2005} as a route to scalable algorithms: they are inherently parallel by splitting a global problem into many coupled local ones, while preserving global coherence through information exchange.

A representative setting arises in electromagnetics, where reconstructing the electric permittivity $\varepsilon$ from field measurements requires repeatedly solving Maxwell-type equations. In the frequency domain, one seeks the electric field $\mathbf{E}$ satisfying
\begin{equation}
\nabla \times (\mu^{-1}\nabla \times \mathbf{E}) - \omega^2 \varepsilon \mathbf{E} = \mathbf{J},
\end{equation}
where $\mu>0$ is the magnetic permeability, $\omega$ the angular frequency, $\varepsilon>0$ the electric permittivity and $\mathbf{J}$ a current source. In inverse problems, $\varepsilon$ is unknown and must be inferred from observations of $\mathbf{E}$, leading to repeated forward solutions in complex geometries and often in high-frequency regimes.

\begin{svgraybox}
The reconstruction of $\varepsilon$ from measurements of $\mathbf{E}$ is typically ill-posed and computationally demanding: small data errors can induce large parameter errors, and high-frequency discretizations yield very large systems that must be solved many times inside an iterative optimisation algorithm. A very performant forward solve is therefore necessary.
\end{svgraybox}

\begin{figure}[H]
    \centering
    \includegraphics[width=0.28\textwidth]{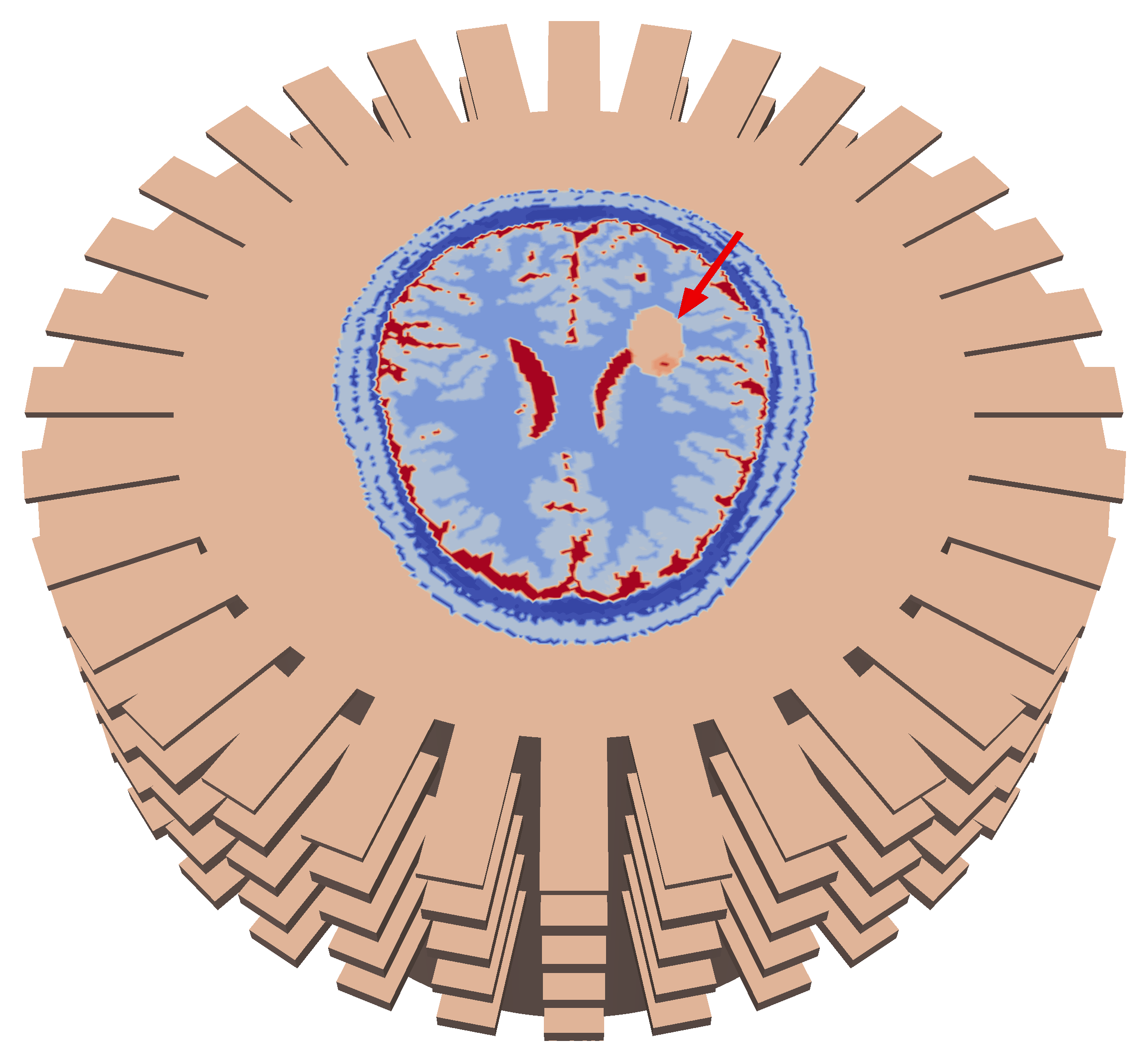}
    \includegraphics[width=0.28\textwidth]{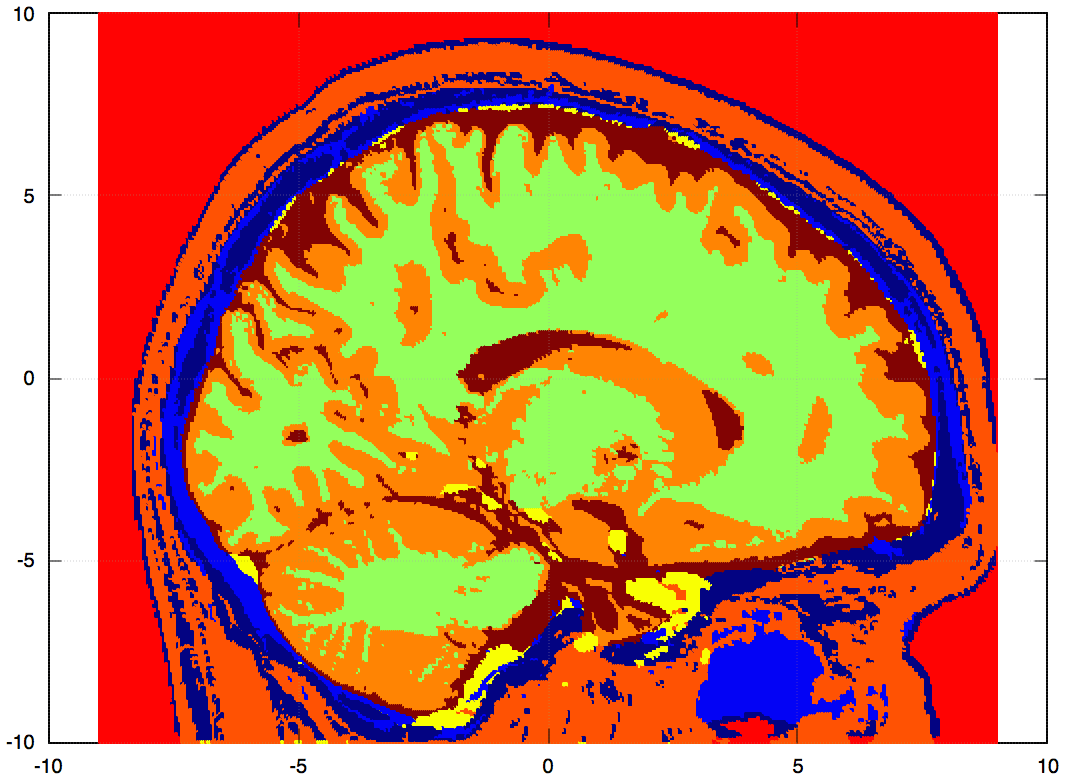}
    \includegraphics[width=0.35\textwidth]{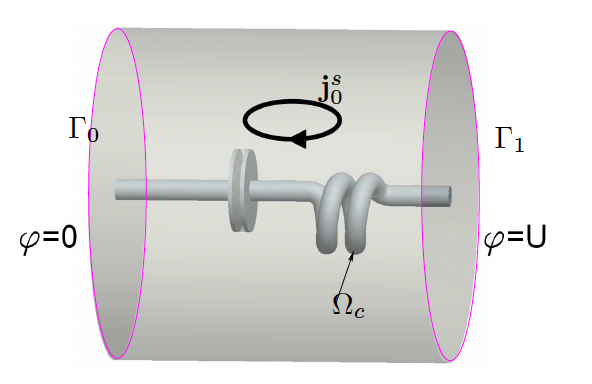}
    \caption{Examples of wave propagation in complex domains: brain imaging for stroke localization \cite{Tournier2017} (left, center), and current injection in an electric component (right).}
\end{figure}
In a nutshell, most PDE discretizations lead to linear systems of the form
\begin{equation}
A\mathbf{u}=\mathbf{b}.
\end{equation}
The solver choice determines not only runtime but also robustness across parameter regimes (heterogeneity, high contrast, high frequency) and across hardware platforms (multicore CPUs, GPUs, clusters). At a high level, one distinguishes among the following classes:
\begin{itemize}
\item \textit{Direct solvers}, which are robust but can become prohibitively expensive in memory and difficult to scale at very large sizes;
\item \textit{Iterative solvers}, which can be scalable and amenable to parallel computing but may converge slowly when $A$ is ill-conditioned or strongly indefinite;
\item \textit{Hybrid strategies}, where some algorithmic structure (e.g., multigrid or domain decomposition) is used to combine local robustness with global scalability.
\end{itemize}
Domain decomposition belongs to this third category: it leverages fast local solves (often direct) while orchestrating global convergence through carefully designed coupling.

Even for sparse matrices arising from PDE discretizations, factorization costs grow rapidly with dimension. Typical asymptotic complexities for sparse Gaussian elimination are summarized in Table~\ref{tab:direct_complexity}. The practical implication is that, while large 2D problems can often still be handled with sophisticated sparse direct methods, 3D problems quickly encounter fill-in and memory limitations. However, direct solvers remain invaluable as local solvers inside hybrid approaches, including DDM.

\begin{table}[!t]
\caption{Asymptotic complexity trends for Gaussian elimination under increasing structural exploitation (illustrative orders).}
\label{tab:direct_complexity}
\begin{tabular}{lccc}
\hline\noalign{\smallskip}
Method & 1D ($d=1$) & 2D ($d=2$) & 3D ($d=3$)\\
\noalign{\smallskip}\svhline\noalign{\smallskip}
Dense matrix & $\mathcal{O}(n^3)$ & $\mathcal{O}(n^3)$ & $\mathcal{O}(n^3)$\\
Band structure exploited & $\mathcal{O}(n)$ & $\mathcal{O}(n^2)$ & $\mathcal{O}(n^{7/3})$\\
Sparse (e.g., nested dissection) & $\mathcal{O}(n)$ & $\mathcal{O}(n^{3/2})$ & $\mathcal{O}(n^{2})$\\
\noalign{\smallskip}\hline\noalign{\smallskip}
\end{tabular}
\end{table}

Moreover, contemporary performance gains come primarily from parallelism, rather than increased clock speed. Most computational clusters now offer many cores, and high performant systems combine massive distributed-memory parallelism with accelerators. Algorithms must therefore provide (i) substantial concurrency, (ii) limited global synchronization, and (iii) locality of data movement.

\begin{figure}[h!]
    \centering
    \add{\includegraphics[width=0.9\textwidth,clip,trim=0.5cm 2cm 3.8cm 2.2cm]{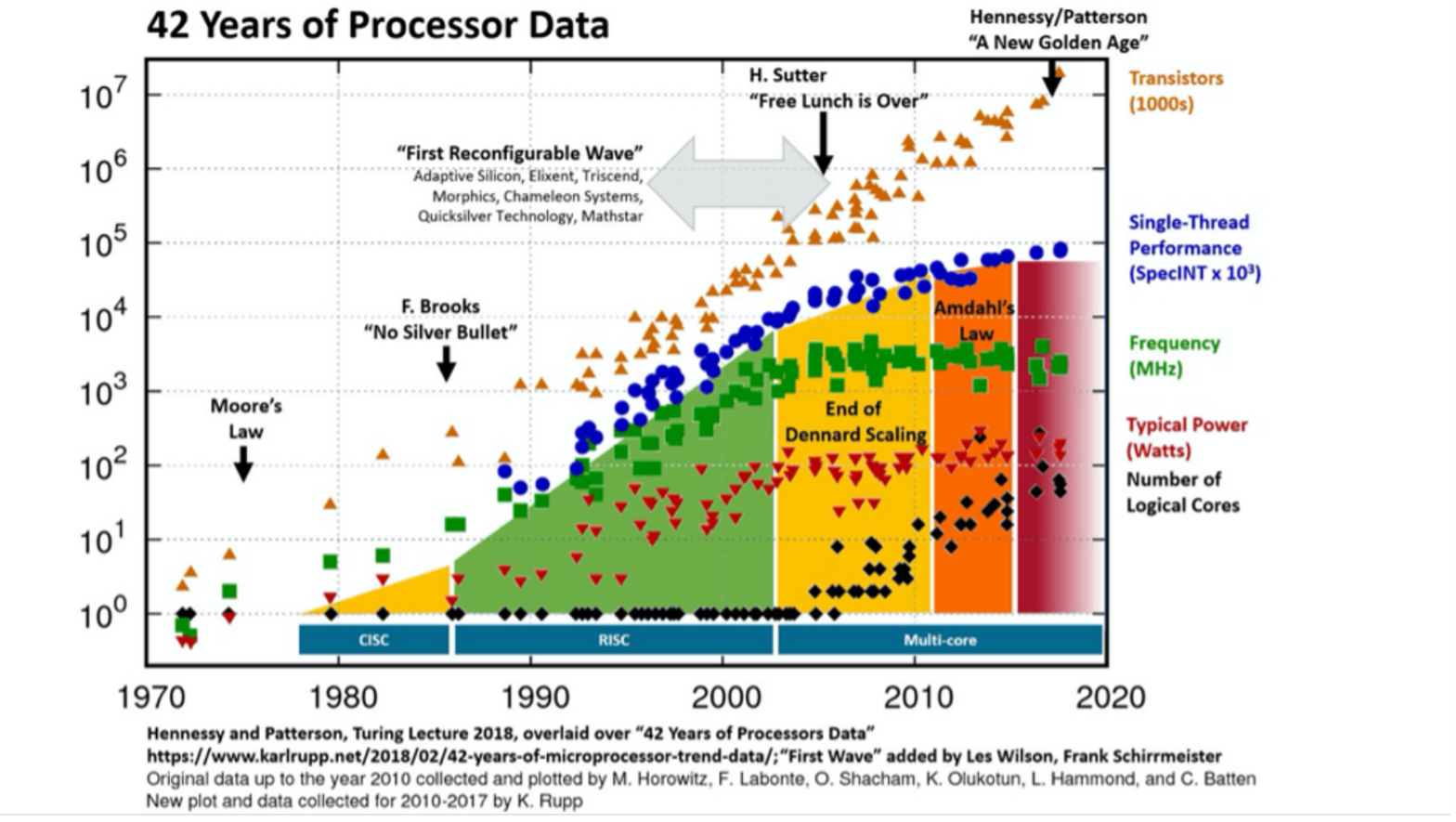}}
    \caption{\add{End of Dennard scaling — frequency gains stall.}}
\end{figure}

\begin{figure}[h!]
    \centering
    \add{\includegraphics[width=0.80\textwidth]{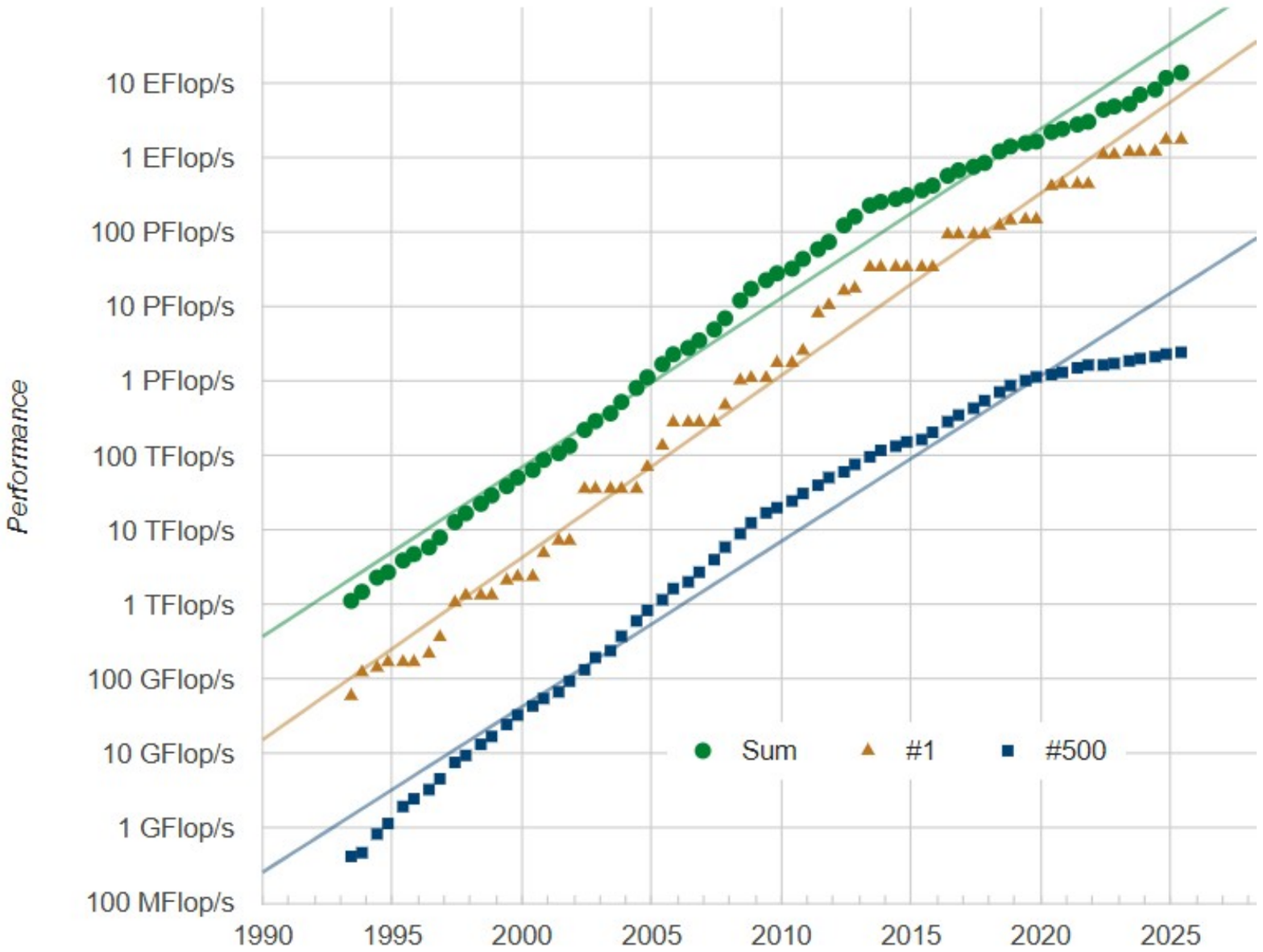}}
    \caption{\add{Performance trend of the Top500 supercomputers (sum,~\#1,~\#500) over the last three decades.}}
\end{figure}

\begin{svgraybox}
Domain decomposition methods align naturally with modern hardware: they replace a monolithic solve by many local solves that can run concurrently, while communicating only through interfaces. This makes DDM an effective abstraction layer between PDE discretizations and parallel architectures. \cite{Jolivet2012}
\end{svgraybox}

Therefore, from computational point of view domain decomposition provides a flexible and modular route to scalable PDE solvers:
\begin{itemize}
\item it exposes parallelism through independent subdomain solves;
\item it supports hybrid strategies that combine local robustness (often via direct solvers) with global convergence acceleration;
\item it matches current hardware trends by emphasizing concurrency and locality.
\end{itemize}
The remainder of this section develops the main domain decomposition concepts and their algebraic versions, preparing their use as standalone iterative methods and as preconditioners for Krylov solvers.

\subsection{The Schwarz Method: Historical Origin of Domain Decomposition}
\label{subsec:schwarz_origin}

Domain decomposition methods trace back to Schwarz' alternating method (1870), originally introduced as an overlapping strategy to solve elliptic boundary-value problems by iteratively coupling subproblems on subdomains. In its simplest form, the method solves
\begin{equation}
-\Delta u = f \quad \text{in } \Omega, \qquad u = 0 \quad \text{on } \partial\Omega,
\label{eq:poisson_model}
\end{equation}
by splitting $\Omega$ into overlapping subdomains $\Omega_1$ and $\Omega_2$ and updating local solutions sequentially.

\begin{backgroundinformation}{The original idea (Schwarz, 1870)}
The key mechanism is \emph{overlap}: information is transmitted between subdomains through Dirichlet data on the artificial interfaces. Without overlap, the classical alternating method loses its contraction property for elliptic problems.
\end{backgroundinformation}

Given iterates $u_1^{n}$ and $u_2^{n}$, compute $u_1^{n+1}$ and $u_2^{n+1}$ by solving
\begin{align}
-\Delta u_1^{n+1} &= f \quad \text{in } \Omega_1,
&
u_1^{n+1} &= 0 \quad \text{on } \partial\Omega_1 \cap \partial\Omega,
\nonumber\\
&&
u_1^{n+1} &= u_2^{n} \quad \text{on } \partial\Omega_1 \cap \overline{\Omega}_2,
\label{eq:schwarz_step1}
\\[0.2em]
-\Delta u_2^{n+1} &= f \quad \text{in } \Omega_2,
&
u_2^{n+1} &= 0 \quad \text{on } \partial\Omega_2 \cap \partial\Omega,
\nonumber\\
&&
u_2^{n+1} &= u_1^{n+1} \quad \text{on } \partial\Omega_2 \cap \overline{\Omega}_1.
\label{eq:schwarz_step2}
\end{align}
The second subproblem uses the most recent interface values from the first solve, which explains the sequential (Gauss--Seidel-like) nature of the classical scheme.

% Requires: \usepackage{tikz}

\begin{figure}[ht]
\centering
\begin{minipage}{0.55\textwidth}
  \centering
  \begin{tikzpicture}[scale=0.8]

    \def\r{2.0}
    \def\rectW{5.2}
    \def\rectH{2.6}
    \def\rectLeft{1.2}

    % Circle (Omega_1)
    \draw[line width=0.9pt] (0,0) circle (\r);
    \node at (-0.55,0) {$\Omega_1$};

    % Rectangle (Omega_2)
    \draw[line width=0.9pt]
      (\rectLeft,-0.5*\rectH) rectangle ++(\rectW,\rectH);
    \node at (\rectLeft+0.5*\rectW,0) {$\Omega_2$};

  \end{tikzpicture}
\end{minipage}\hfill
\begin{minipage}{0.4\textwidth}
  \caption{Two overlapping subdomains $\Omega_1$ and $\Omega_2$ for Schwarz alternating iteration.}
\end{minipage}
\end{figure}
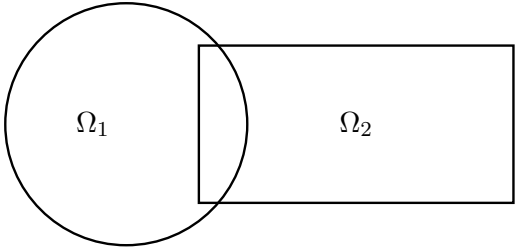

\begin{important}{Key observations}
Schwarz' method already contains the main ingredients of a hybrid solver: local solves on subdomains, communication through interface conditions, and a global iteration built from these local updates. Overlap is essential for contraction in the classical elliptic setting. With vanishing overlap, the basic alternating scheme with Dirichlet transmission conditions stagnate.
\end{important}

A fully parallel variant is obtained by updating all subdomains simultaneously using interface data from the previous iterate. This Jacobi-type variant is often called the \emph{Jacobi--Schwarz} method.

%-----------------------------------------------------------------

\subsection{Abstract Additive and Restricted Additive Schwarz}
\label{subsec:asm_ras_abstract}

Modern domain decomposition methods generalize Schwarz' construction to many subdomains and to arbitrary discretizations. The unifying viewpoint is \emph{local-to-global}: compute local corrections, then assemble a global update by extension (prolongation) and weighting. In this part we follow the presentation from \cite{Dolean2015}.

\begin{definition}[Extension operators]
For each subdomain $\Omega_i$, the extension operator $E_i$ maps a local function $w_i:\Omega_i\to\mathbb{R}$ to a global function $E_i(w_i):\Omega\to\mathbb{R}$ by setting it to zero outside $\Omega_i$.
\end{definition}

\begin{definition}[Partition of unity]
Let $\{\chi_i\}$ be nonnegative weights such that $\chi_i=0$ on $\partial\Omega_i$ and
\begin{equation}
\add{\sum_i E_i(\chi_i)(x)=1} \quad \text{for all } x\in\Omega.
\label{eq:partition_unity_cont}
\end{equation}
Then any sufficiently regular global function $w$ can be reconstructed from its restrictions as
\begin{equation}
w(x)=\sum_i E_i\!\left(\chi_i\, w|_{\Omega_i}\right)(x).
\label{eq:reconstruction_cont}
\end{equation}
\end{definition}

Given a global iterate $u^n$, one solves local subproblems (typically for a correction or residual equation) and assembles the next iterate via
\begin{align}
\text{ASM:}\qquad
u^{n+1} &= \sum_i E_i\!\left(u_i^{n+1}\right),
\label{eq:asm_iter}
\\
\text{RAS:}\qquad
u^{n+1} &= \sum_i E_i\!\left(\chi_i\,u_i^{n+1}\right).
\label{eq:ras_iter}
\end{align}
Restricted Additive Schwarz (RAS) \cite{CaiSarkis1999} introduces partition-of-unity weights to limit redundant overlap updates and reduce communication, while preserving the same conceptual structure.

\begin{tips}{Practical rule of thumb}
ASM is often preferred to preserve symmetry, often useful in theory and analysis, whereas RAS is frequently faster in large-scale implementations because it limits redundancies and communication in the global assembly.
\end{tips}

%-----------------------------------------------------------------
\subsection{Algebraic view: from (block) Jacobi to AS/RAS/ORAS}
\label{subsec:block_jacobi_connection}

Schwarz methods become particularly transparent when written directly at the
linear-algebra level. Consider
\begin{equation}
A\mathbf{U}=\mathbf{F},\qquad
\mathbf{r}^n:=\mathbf{F}-A\mathbf{U}^n.
\label{eq:lin_system}
\end{equation}
The (point) Jacobi iteration reads
\begin{equation}
\mathbf{U}^{n+1}=\mathbf{U}^n + D^{-1}\mathbf{r}^n,
\label{eq:jacobi_point}
\end{equation}
where $D$ is the diagonal of $A$. Replacing scalar unknowns by \emph{subdomain
blocks} yields the standard block-Jacobi viewpoint underlying additive Schwarz
preconditioners.

\begin{backgroundinformation}{Restrictions, extensions, and subdomain operators}
Let $\{\Omega_i\}_{i=1}^N$ be an overlapping decomposition and let $R_i$ be the
Boolean restriction from global unknowns to the unknowns of subdomain $i$.
Define the local (subdomain) operator
\[
A_i := R_i A R_i^T,
\]
possibly modified at the subdomain boundary (e.g.\ Dirichlet for ASM, Robin for
ORAS). The (additive) extension of a local vector $\mathbf{v}_i$ back to the
global space is $R_i^T\mathbf{v}_i$.
\end{backgroundinformation}

%-----------------------------------------------------------------
\subsubsection*{Block-Jacobi as the prototype}

To make the correspondence concrete, split the degrees of freedom
$\mathcal{N}$ into two blocks $\mathcal{N}_1$ and $\mathcal{N}_2$ and write
\begin{equation}
A=
\begin{pmatrix}
A_{11} & A_{12}\\
A_{21} & A_{22}
\end{pmatrix},\qquad
\mathbf{U}=
\begin{pmatrix}
\mathbf{U}_1\\
\mathbf{U}_2
\end{pmatrix},\qquad
\mathbf{F}=
\begin{pmatrix}
\mathbf{F}_1\\
\mathbf{F}_2
\end{pmatrix}.
\label{eq:block_form}
\end{equation}
Block Jacobi computes $\mathbf{U}_1^{n+1}$ and $\mathbf{U}_2^{n+1}$ independently via
\begin{align}
A_{11}\mathbf{U}_1^{n+1} &= \mathbf{F}_1 - A_{12}\mathbf{U}_2^{n},
\label{eq:block_jacobi1}\\
A_{22}\mathbf{U}_2^{n+1} &= \mathbf{F}_2 - A_{21}\mathbf{U}_1^{n}.
\label{eq:block_jacobi2}
\end{align}
Equivalently, the global update can be written in additive form as
\begin{equation}
\mathbf{U}^{n+1}
=
\mathbf{U}^{n}
+
\left(R_1^T A_1^{-1}R_1 + R_2^T A_2^{-1}R_2\right)\mathbf{r}^n =: \mathbf{U}^{n}
+M^{-1}\mathbf{r}^n,
\label{eq:additive_two_block}
\end{equation}
which is the prototypical \emph{additive Schwarz} (ASM) update. We can see that \eqref{eq:additive_two_block} can also be seen as a preconditioned stationary iteration with the preconditioner $M^{-1}$.

%-----------------------------------------------------------------
\subsubsection*{Where RAS enters: restricting overlap contributions}

The distinction between ASM and RAS is purely algebraic:
\emph{ASM injects} each local correction on the full (overlapping) subdomain,
whereas \emph{RAS restricts} the injection either to a non-overlapping core or it sums a fraction of the local contributions of the corrections in a coherent manner.
This is implemented by inserting weights that form a discrete partition of unity.

\begin{definition}[ASM/RAS in a unified algebraic form]
Let $D_i$ be diagonal weight matrices such that
\( \sum_{i=1}^N R_i^T D_i R_i = I.\)
Then the preconditioned stationary iteration  can be written as
\begin{equation}
\mathbf{U}^{n+1}
=
\mathbf{U}^{n}
+ M^{-1}\mathbf{r}^n,\, 
M^{-1}
=
\sum_{i=1}^N R_i^T D_i\, B_i^{-1} R_i,
\label{eq:unified_one_level}
\end{equation}
where
\[
B_i =
\begin{cases}
A_i & \text{(ASM/RAS: Dirichlet-type local operator)},\\
\text{Robin-modified }A_i & \text{(ORAS: optimized local operator)}.
\end{cases}
\]
\end{definition}

\begin{svgraybox}
\textbf{Interpretation.}
\eqref{eq:unified_one_level} is a block-Jacobi like update: each subdomain solves a local problem for its share of the residual, and the
global correction is obtained by \emph{weighted assembly}.
\begin{itemize}
\item \textbf{ASM:} choose $D_i=I$ (full injection on the overlap).
\item \textbf{RAS:} choose $D_i$ that restricts (or downweights) overlap
contributions so that each global degree of freedom is corrected essentially
once.
\item \textbf{ORAS:} keep the same assembly idea but change the local operator
to a Robin/impedance form to better match wave propagation.
\end{itemize}
\end{svgraybox}

\begin{trailer}{Algebraic insight}
Equation~\eqref{eq:unified_one_level} is the prototypical additive Schwarz update: solve independent subproblems, prolong the corrections to the global space, and sum them. The only additional ingredient of RAS is the insertion of weights (a discrete partition of unity) to restrict overlap contributions.
\end{trailer}

%-----------------------------------------------------------------
To make the correspondence concrete, consider the 1D Poisson problem on $\Omega=(0,1)$ with homogeneous Dirichlet boundary conditions. A centered finite difference discretization on a uniform grid with $m$ interior points yields
\begin{equation}
A\mathbf{U}=\mathbf{F},\qquad
A=\frac{1}{h^2}\operatorname{tridiag}(-1,2,-1),
\qquad h=\frac{1}{m+1}.
\label{eq:1d_fd_system}
\end{equation}
We decompose the domain into overlapping subdomains
\begin{equation}
\Omega_1=(0,(m_s+1)h),
\qquad
\Omega_2=(m_s h,1),
\label{eq:1d_overlap_domains}
\end{equation}
so that the overlap consists of one mesh cell. On $\Omega_1$, the local finite difference solve with interface data from $\Omega_2$ reads
\begin{align}
-\frac{u_{1,j-1}^{n+1}-2u_{1,j}^{n+1}+u_{1,j+1}^{n+1}}{h^2} &= f_j,
\qquad 1\le j\le m_s,
\nonumber\\
u_{1,0}^{n+1} &= 0,
\qquad
u_{1,m_s+1}^{n+1} = u_{2,m_s+1}^{n}.
\label{eq:js_local_update}
\end{align}
An analogous update is performed on $\Omega_2$ using the appropriate interface value from $\Omega_1$ (either from the previous iterate in a fully parallel/Jacobi variant, or from the newest iterate in an alternating/Gauss--Seidel variant).

\begin{figure}[ht]
\centering
\resizebox{0.85\linewidth}{!}{%
\begin{tikzpicture}[
  x=\linewidth,     % 0..1 spans exactly the current line width
  y=1cm,
  line cap=round,
  line join=round
]

% -----------------------------
% Parameters (in [0,1])
% -----------------------------
\def\a{0.56}   % left end of overlap
\def\b{0.66}   % right end of overlap
\def\Ychi{2.2} % height of partition-of-unity plot
\def\Ydom{-0.9}% height of domain interval arrows

% Bounding box (prevents cropping surprises with thick lines)
\path[use as bounding box] (-0.02,-1.55) rectangle (1.02,2.75);

% -----------------------------
% 1D mesh line with nodes
% -----------------------------

% --- Draw nodes LAST so they are on top of thick strokes
\foreach \k in {0,...,12}{
  \coordinate (P\k) at ({\k/12},0);
  \fill[white] (P\k) circle (0.075);                % white halo
}

% Mark overlap endpoints (also draw last)
\draw[line width=0.9pt] (\a,0) -- (\a,0.14);
\draw[line width=0.9pt] (\b,0) -- (\b,0.14);
\node[below] at (\a,0) {$x_{m_s}$};
\node[below] at (\b,0) {$x_{m_s+1}$};

% Small overlap indicator arrow under the line
\draw[<->, line width=0.9pt] (\a,-0.45) -- (\b,-0.45);

% -----------------------------
% Partition of unity functions xi_1, xi_2 (continuous, X in [a,b])
% -----------------------------
\draw[line width=3.0pt]
  (0,\Ychi) -- (\a,\Ychi) -- (\b,0) -- (1,0);

\draw[line width=3.0pt]
  (0,0) -- (\a,0) -- (\b,\Ychi) -- (1,\Ychi);

\node at (0.18,1.65) {$\chi_1$};
\node at (0.82,1.65) {$\chi_2$};

% -----------------------------
% Overlapping subdomains Omega_1 and Omega_2 on the 1D line
% Omega_1 = [0,b], Omega_2 = [a,1]
% -----------------------------
\draw[<- , line width=1.1pt] (0,\Ydom) -- (\b,\Ydom);
\draw[-> , line width=1.1pt] (\a,\Ydom-0.15) -- (1,\Ydom-0.15);

\node at ({0.5*\b},\Ydom-0.45) {$\Omega_1$};
\node at ({(\a+1)/2},\Ydom-0.60) {$\Omega_2$};

% Emphasize overlap region on the axis (optional subtle highlight)
\draw[line width=5.0pt, opacity=0.12] (\a,0) -- (\b,0);

\end{tikzpicture}
}
\caption{Overlapping subdomains $\Omega_1=[0,b]$ and $\Omega_2=[a,1]$ with minimum overlap and a partition of unity $(\xi_1,\xi_2)$ for this decomposition.}
\end{figure}
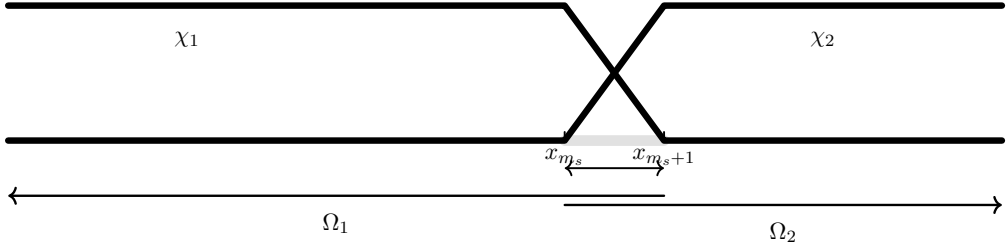

\begin{important}{Key insight}
With the domain splitting \eqref{eq:1d_overlap_domains}, the local solves \eqref{eq:js_local_update} correspond precisely to the block-Jacobi update \eqref{eq:block_jacobi1}--\eqref{eq:block_jacobi2}. Hence for minimal overlap, Additive Schwarz (AS) and Restricted Additive Schwarz (RAS) coincide at the algebraic level: they all reduce to a block-Jacobi scheme. In this minimal overlap configuration, the distinctions between the different popular variants disappear.
\end{important}

%-----------------------------------------------------------------

%-----------------------------------------------------------------
\subsection{From continuous to discrete}
\label{subsec:discrete_setting_matrix}

Domain decomposition algorithms are implemented and analyzed most naturally in discrete, linear-algebraic terms. The goal of this section is to translate the continuous ingredients (restriction, extension, partition of unity) into sparse operators acting on vectors of degrees of freedom, and to show how these operators define local subproblems and global updates.

At the continuous level, a global field $u$ is restricted to subdomains, locally processed, and then extended back to $\Omega$. At the discrete level, the same logic is encoded by sparse restriction matrices, their transposes as prolongations, and diagonal weights implementing a partition of unity.

Let $\Omega=\bigcup_{i=1}^N \Omega_i$ be an overlapping decomposition and let $\mathcal{N}$ be the set of global degrees of freedom (DoFs). The key correspondence is summarized below.

\begin{table}[ht]
\centering
\renewcommand{\arraystretch}{1.4}
\begin{tabular}{p{0.28\textwidth} p{0.32\textwidth} p{0.32\textwidth}}
\toprule
\textbf{Concept}
& \textbf{Continuous level}
& \textbf{Discrete level} \\
\midrule
Domain / index set
& $\Omega=\bigcup_{i=1}^N \Omega_i$
& $\mathcal{N}=\bigcup_{i=1}^N \mathcal{N}_i$ \\

Unknown
& $u:\Omega\to\mathbb{R}$
& $\mathbf{U}\in\mathbb{R}^{\#\mathcal{N}}$ \\

Restriction
& $u_i=u|_{\Omega_i}$
& $R_i$ \\

Extension
& $E_i$
& $R_i^T$ \\

Partition of unity
& $\{\chi_i\}$,
$u=\sum_i E_i(\chi_i u_i)$
& $\{D_i\}$,
$\sum_i R_i^T D_i R_i = I$ \\
\bottomrule
\end{tabular}
\caption{Continuous–discrete correspondence in overlapping Schwarz methods.}
\end{table}

Restriction operators connect global and local vectors. Concretely, if $\mathbf{U}$ is the global vector of DoFs, then $\mathbf{U}_i := R_i\mathbf{U}$ collects the entries living in subdomain $i$, and $R_i^T$ injects a local vector back into the global space.
\begin{backgroundinformation}{Mesh-based viewpoint}
Let $\mathcal{T}_h$ be a global mesh on $\Omega$ and $\mathcal{T}_{h,i}$ the mesh induced on $\Omega_i$. In finite elements, the restriction map $r_i:V_h\to V_{h,i}$ (defined by $r_i(u_h)=u_h|_{\Omega_i}$) is represented algebraically by a Boolean matrix $R_i$. In finite differences, $R_i$ simply selects grid values on the subdomain stencil.
\end{backgroundinformation}
\noindent Given a global linear system $A\mathbf{U}=\mathbf{F}$, the local matrices and residuals are defined by
\begin{equation}
A_i := R_i A R_i^T,
\qquad
\mathbf{r}_i := R_i(\mathbf{F}-A\mathbf{U}),
\label{eq:local_matrix_residual}
\end{equation}
and local corrections $\mathbf{z}_i$ are obtained by solving \(
A_i \mathbf{z}_i = \mathbf{r}_i.\)

\begin{important}{Algebraic ingredients in DDM}
The triplet $(R_i,\,A_i,\,D_i)$ determines the method:
restriction/prolongation ($R_i$, $R_i^T$), the local operator ($A_i$), and the partition-of-unity weights ($D_i$) used to assemble a global update. When subdomains overlap, a global DoF may be present in several local vectors. The role of $D_i$ is to weight these contributions so that the global assembly does not overcount overlap.
\end{important}

%-----------------------------------------------------------------
\begin{definition}[Discrete partition of unity]
Let $R_i$ restrict to the (possibly extended) subdomain DoF set and let $D_i$ be diagonal matrices of matching size. The family $\{D_i\}_{i=1}^N$ is a discrete partition of unity if
\begin{equation}
\sum_{i=1}^N R_i^T D_i R_i = I.
\label{eq:discrete_partition_unity}
\end{equation}
\end{definition}

A common and robust choice is \emph{multiplicity scaling}: each global DoF is weighted by the inverse of the number of subdomains that contain it (see \eqref{eq:multiplicity_scaling} below).

\begin{tips}{Implementation tip}
When $R_i$ are Boolean selection matrices, the diagonal of $\sum_i R_i^T R_i$ stores the overlap multiplicity of each global DoF. Multiplicity scaling amounts to using its reciprocal as a global weight, then restricting it locally to build each $D_i$.

\add{Besides multiplicity scaling, a popular alternative is the \emph{Boolean} partition of unity used in PETSc, in which each global degree of freedom is assigned to a single subdomain (typically the one of lowest rank in $\mathcal{M}_j$) so that $D_i$ has only $0$/$1$ entries. This choice is immune to round-off errors and is in practice more robust than multiplicity scaling, at the price of a slightly less symmetric treatment of the overlap.}
\end{tips}

%-----------------------------------------------------------------
\subsection{Algebraic Schwarz formulation}
\label{subsec:algebraic_partitioning_overlap}

In large-scale applications, subdomains are commonly defined algebraically from the sparsity pattern of $A$. The matrix induces a graph $G=(V,E)$ where nodes represent DoFs and edges represent couplings.

\begin{backgroundinformation}{Graph-based partitioning}
Given a sparse matrix $A$, the associated (symmetrized) graph has an edge $(i,j)$ whenever $A_{ij}\neq 0$. Standard graph partitioners such as \texttt{METIS} \cite{KarypisKumar1998} or \texttt{SCOTCH} \cite{ChevalierPellegrini2008} compute balanced node partitions, thereby achieving the following goals:
\begin{itemize}
\item \textbf{Load balance:} each subdomain carries roughly the same number of unknowns (and local work).
\item \textbf{Reduced communication:} the interface size (edge cut) is small compared with the subdomain volume.
\end{itemize}
\end{backgroundinformation}

\begin{definition}[Overlapping subdomains (graph-based)]
Let $\{\mathcal{N}_i\}$ be a \add{partition of DoFs}. The one-layer overlap extension is
\begin{equation}
\mathcal{N}_i^{\delta=1}
=
\mathcal{N}_i \cup \{\, j \in \mathcal{N} \;:\; \exists\, k\in\mathcal{N}_i \text{ with } A_{kj}\neq 0 \,\}.
\label{eq:overlap_graph}
\end{equation}
More generally, $\delta$ layers are obtained by repeating the neighbor expansion $\delta$ times.
\end{definition}

\begin{figure}[t]
\centering
\begin{minipage}[c]{0.58\textwidth}
  \centering
  \includegraphics[width=\textwidth]{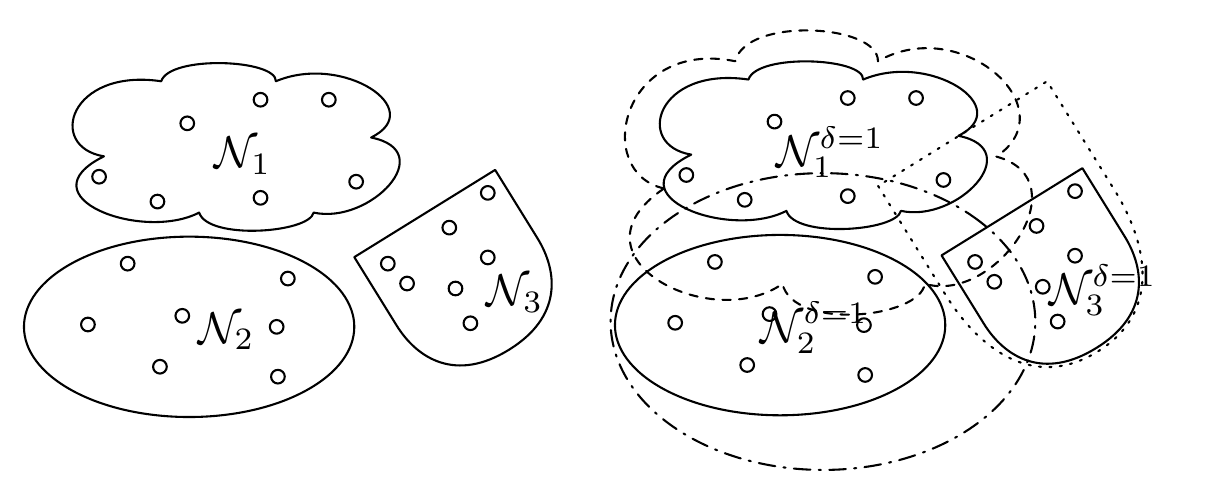}
\end{minipage}\hfill
\begin{minipage}[c]{0.38\textwidth}
  \caption{Extension of subdomains to overlapping neighborhoods in a stencil-based discretization.}
  \label{fig:overlap_extension}
\end{minipage}
\end{figure}

\begin{definition}[Multiplicity scaling and partition of unity]
\label{def:multiplicity_scaling}
Let $\{\mathcal{N}_i^{\delta}\}_{i=1}^N$ denote the (possibly overlapping)
index sets associated with the subdomains.
For a global degree of freedom $j$, define the multiplicity set
\[
\mathcal{M}_j := \{\, i \;:\; j\in\mathcal{N}_i^{\delta} \,\},
\]
and the associated multiplicity \(
m_j := \#\mathcal{M}_j\).
For each subdomain $i$, define a diagonal scaling matrix $D_i$ by assigning
to every local copy of a global degree of freedom $j$ the weight
\begin{equation}
\add{(D_i)_{R_i(j)\,R_i(j)}} = \frac{1}{m_j},
\qquad j\in \mathcal{N}_i^{\delta}.
\label{eq:multiplicity_scaling}
\end{equation}
Then the family $\{D_i\}_{i=1}^N$ satisfies the discrete partition-of-unity identity \eqref{eq:discrete_partition_unity}.

For finite element discretizations \add{and a given domain decomposition $\{\Omega_i\}$}, overlap is naturally defined through basis
function supports. Let $\{\phi_k\}_{k\in\mathcal{N}}$ be a global FE basis and
define the subdomain index set by
\begin{equation}
\mathcal{N}_i
:=
\{\, k\in\mathcal{N} \;:\;
\operatorname{supp}(\phi_k)\cap\Omega_i \neq \emptyset \,\}.
\label{eq:fe_subdomain_dofs}
\end{equation}
The multiplicity of a basis function is then
\begin{equation}
\mu_k
:=
\#\{\, i \;:\; k\in\mathcal{N}_i \,\},
\label{eq:fe_multiplicity}
\end{equation}
i.e.\ the number of subdomains whose overlap contains its support.
In this setting, multiplicity scaling simply reads
\[
\add{(D_i)_{R_i(k)\,R_i(k)}} = \frac{1}{\mu_k},
\]
which distributes the contribution of each basis function evenly across all
subdomains containing it.
\end{definition}

\begin{figure}[t]
\centering
\begin{minipage}[c]{0.55\textwidth}
  \centering
  \includegraphics[width=\textwidth]{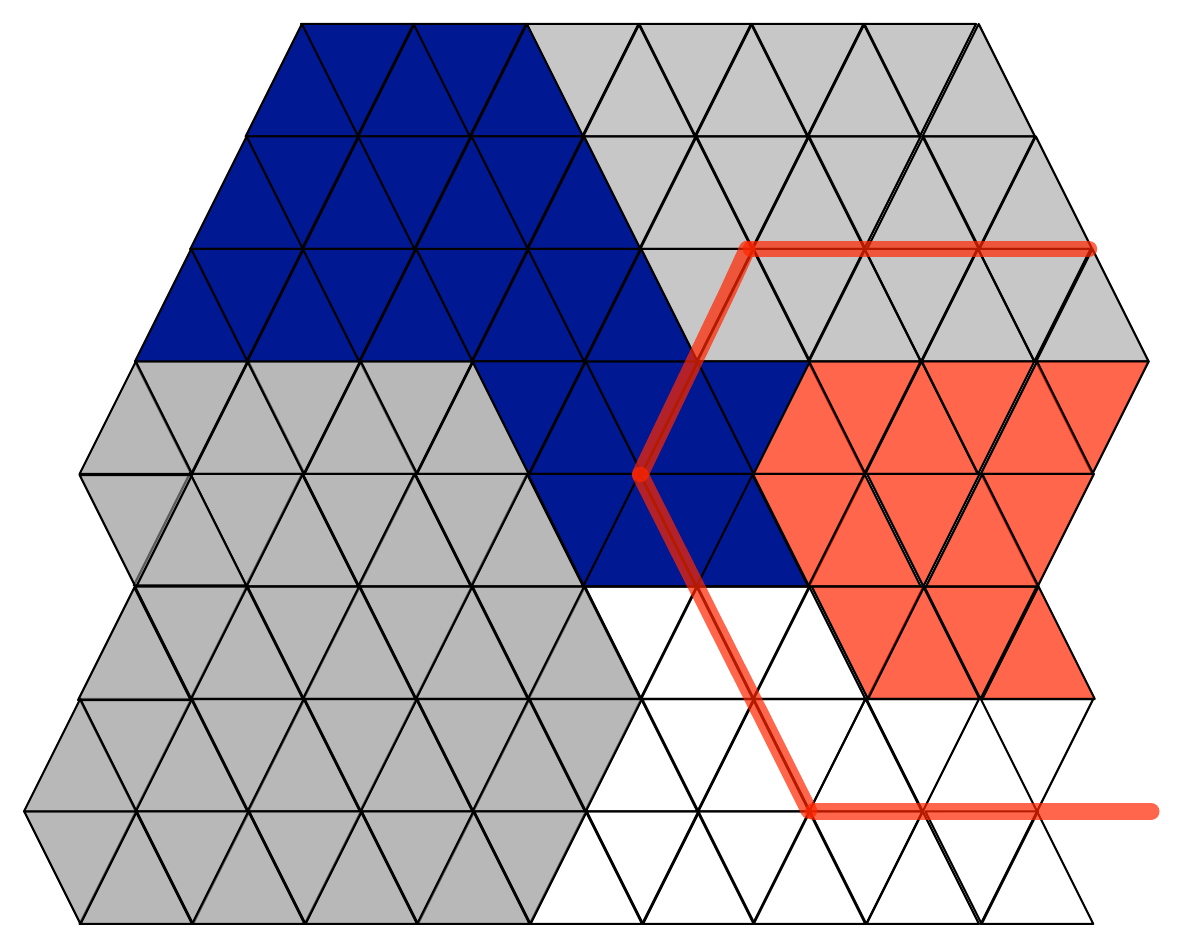}
\end{minipage}\hfill
\begin{minipage}[c]{0.38\textwidth}
  \caption{Overlapping finite element subdomains $\{\Omega_i\}$ covering $\Omega$.}
\label{fig:fe_overlap_2d}
\end{minipage}
\end{figure}
The discrete operators introduced above lead directly to standard additive Schwarz updates \eqref{eq:unified_one_level}.
%=================================================================
%-----------------------------------------------------------------
\subsection{Schwarz preconditioners in Krylov methods}
\label{sec:prec_krylov}

Large-scale PDE discretizations are rarely solved by stationary iterations.
In practice, Schwarz operators are used as \emph{preconditioners}
inside Krylov subspace methods. Consider the linear system
\[
A\mathbf{U}=\mathbf{F}.
\]
Given a linear preconditioner $M^{-1}$, a left-preconditioned Krylov method builds the Krylov space generated by $M^{-1}A$. The quality of the preconditioner is therefore governed most of the time by the spectral
properties of the operator $M^{-1}A$.

\add{Equivalently, one may apply $M^{-1}$ on the right and solve $A M^{-1} \mathbf{y} = \mathbf{F}$ with $\mathbf{U} = M^{-1}\mathbf{y}$. For GMRES, this \emph{right-preconditioned} formulation has the practical advantage that the residual norm computed by the algorithm is the true (unpreconditioned) residual norm, which makes the stopping criterion more reliable. This is the variant used by default in most modern implementations.}

In the unified algebraic form introduced earlier,
one-level Schwarz preconditioners read
\begin{equation}
M^{-1}
=
\sum_{i=1}^N R_i^T D_i\, B_i^{-1} R_i,
\label{eq:unified_schwarz_prec}
\end{equation}
where
\begin{itemize}
\item $R_i$ restricts to subdomain $i$,
\item $B_i$ is a local operator (Dirichlet or Robin type),
\item $D_i$ enforces a discrete partition of unity.
\end{itemize}

Applied to a residual, \eqref{eq:unified_schwarz_prec} performs:
\begin{enumerate}
\item restriction of the residual to each subdomain,
\item independent local solves,
\item weighted prolongation and global assembly.
\end{enumerate}
All steps are naturally parallel.

\begin{important}{Schwarz as an approximate inverse}
The operator $M^{-1}$ approximates $A^{-1}$ by combining many local inverses.
In practice, $B_i^{-1}$ is itself realized approximately
(e.g.\ sparse direct solvers or inner iterations),
yielding an inexact but effective preconditioner.
\end{important}
The most common one-level variants differ only in
\emph{assembly weights} and in the choice of local operator:

\begin{align}
M_{\mathrm{ASM}}^{-1}
&=
\sum_{i=1}^N R_i^T A_i^{-1} R_i, \label{eq:asm}
\\
M_{\mathrm{RAS}}^{-1}
&=
\sum_{i=1}^N R_i^T D_i A_i^{-1} R_i,
\\
\add{M_{\mathrm{ORAS}}^{-1}}
&\add{=
\sum_{i=1}^N R_i^T D_i B_i^{-1} R_i,}
\\
M_{\mathrm{SORAS}}^{-1}
&=
\sum_{i=1}^N R_i^T D_i B_i^{-1} D_i R_i.
\end{align}

\begin{overview}{Typical trade-offs}
\begin{itemize}
\item \textbf{ASM:}
fully additive, structurally symmetric (when $A_i$ are),
convenient for analysis and for SPD solvers.
\item \textbf{RAS:}
reduces redundant overlap contributions,
often improving parallel efficiency and time-to-solution.
\item \add{\textbf{ORAS:}
restricted assembly combined with optimized (Robin-type) local operators; in our experience this is the most popular optimized variant in practice, especially for wave-propagation problems.}
\item \textbf{SORAS:}
combines weighted assembly with optimized (Robin-type)
local operators, aiming to balance practical robustness
and favorable spectral properties.
\end{itemize}
\end{overview}

%-----------------------------------------------------------------
We now clarify the operator viewpoint that connects stationary iterations,
preconditioning, and Krylov methods. Consider the linear system
\(
A\mathbf{x}=\mathbf{b},
\)
and let $B$ be a nonsingular preconditioner approximating $A$.
The left-preconditioned Richardson iteration reads
\begin{equation}
\mathbf{x}^{n+1}
=
\mathbf{x}^n + B^{-1}(\mathbf{b}-A\mathbf{x}^n).
\label{eq:richardson_prec}
\end{equation}
Introducing
\[
C := B^{-1}A,
\qquad
\mathbf{f} := B^{-1}\mathbf{b},
\]
this becomes the fixed-point iteration
\begin{equation}
\mathbf{x}^{n+1}
=
(I-C)\mathbf{x}^n + \mathbf{f}.
\label{eq:fixed_point_form}
\end{equation}

\begin{important}{Stationary convergence criterion}
The iteration \eqref{eq:fixed_point_form} converges if
\[
\rho(I-C) < 1,
\]
where $\rho(\cdot)$ denotes the spectral radius.
Thus convergence depends entirely on the spectrum of the
preconditioned operator $C=B^{-1}A$.
\end{important}

This observation is fundamental:  
\emph{preconditioning modifies the operator whose spectrum controls convergence.}

Stationary iterations apply a fixed linear map at each step.
Krylov methods instead build increasingly rich polynomial approximations
to the inverse of $C$. Given the preconditioned system
\[
C\mathbf{x}=\mathbf{f},
\qquad C=B^{-1}A,
\]
and initial residual $\mathbf{r}^0=\mathbf{f}-C\mathbf{x}^0$,
Krylov methods construct iterates of the form
\begin{equation}
\mathbf{x}^n
=
\mathbf{x}^0 + \mathcal{P}_{n-1}(C)\mathbf{r}^0,
\label{eq:poly_view}
\end{equation}
where $\mathcal{P}_{n-1}$ is a polynomial of degree at most $n-1$.

\begin{backgroundinformation}{Polynomial interpretation}
Instead of repeatedly applying $(I-C)$ as in Richardson,
Krylov methods search for the polynomial
$\mathcal{P}_{n-1}$ that best approximates $C^{-1}$.
\end{backgroundinformation}

For SPD systems, this best approximation is achieved by the Conjugate Gradient (CG) method (Algorithm \ref{alg:cg}) which minimizes the error in the $C$-energy norm. For general non-symmetric systems, GMRES \cite{SaadSchultz1986} minimizes the residual
over the Krylov space. Also for these kind of problems $M_{ASM}^{-1}$ is a natural choice of preconditioner because it preserves symmetry.

\begin{algorithm}[!t]
\caption{Conjugate Gradient (CG) method (SPD case)}
\label{alg:cg}
\begin{algorithmic}[1]
\STATE Choose $\mathbf{x}_0$, set $\mathbf{r}_0=\mathbf{b}-A\mathbf{x}_0$, $\mathbf{p}_0=\mathbf{r}_0$.
\FOR{$k=0,1,2,\dots$ until convergence}
\STATE $\alpha_k = \dfrac{(\mathbf{r}_k,\mathbf{r}_k)}{(\mathbf{p}_k,A\mathbf{p}_k)}$
\STATE $\mathbf{x}_{k+1}=\mathbf{x}_k+\alpha_k\mathbf{p}_k$
\STATE $\mathbf{r}_{k+1}=\mathbf{r}_k-\alpha_k A\mathbf{p}_k$
\STATE $\beta_k = \dfrac{(\mathbf{r}_{k+1},\mathbf{r}_{k+1})}{(\mathbf{r}_k,\mathbf{r}_k)}$
\STATE $\mathbf{p}_{k+1}=\mathbf{r}_{k+1}+\beta_k\mathbf{p}_k$
\ENDFOR
\end{algorithmic}
\end{algorithm}

\begin{tips}{Where Schwarz enters}
To use Schwarz in Krylov methods, one replaces every residual evaluation by the \emph{preconditioned} residual $M_{ASM}^{-1}\mathbf{r}$ (left preconditioning) or inserts $M_{ASM}^{-1}$ inside the Krylov operator (right preconditioning). In both cases, the dominant cost is applying $M_{ASM}^{-1}$, i.e.\ performing the local subdomain solves. This is illustrated in Algorithm \ref{alg:pcg_asm}.
\end{tips}

\begin{proposition}[CG convergence with ASM preconditioning]
\label{prop:pcg_asm}
Assume $A$ is SPD. Let $\kappa=\operatorname{cond}(M_{\mathrm{ASM}}^{-1}A)$\add{${}=\lambda_{\max}(M_{\mathrm{ASM}}^{-1}A)/\lambda_{\min}(M_{\mathrm{ASM}}^{-1}A)$}. Then the iterates of preconditioned CG satisfy
\begin{equation}
\|\bar{\mathbf{x}}-\mathbf{x}_m\|_{M_{\mathrm{ASM}}^{-1}A}
\le
2\left(\frac{\sqrt{\kappa}-1}{\sqrt{\kappa}+1}\right)^m
\|\bar{\mathbf{x}}-\mathbf{x}_0\|_{M_{\mathrm{ASM}}^{-1}A},
\label{eq:pcg_bound}
\end{equation}
where $\bar{\mathbf{x}}$ denotes the exact solution.
\end{proposition}

\begin{algorithm}[!t]
\caption{\add{Preconditioned CG with the ASM preconditioner}}
\label{alg:pcg_asm}
\begin{algorithmic}[1]
\STATE Choose $\mathbf{x}_0$, set $\mathbf{r}_0=\mathbf{b}-A\mathbf{x}_0$, $\mathbf{z}_0=M_{\mathrm{ASM}}^{-1}\mathbf{r}_0$, $\mathbf{p}_0=\mathbf{z}_0$.
\FOR{$k=0,1,2,\dots$ until convergence}
\STATE $\alpha_k = \dfrac{(\mathbf{r}_k,\mathbf{z}_k)}{(\mathbf{p}_k,A\mathbf{p}_k)}$
\STATE $\mathbf{x}_{k+1}=\mathbf{x}_k+\alpha_k\mathbf{p}_k$
\STATE $\mathbf{r}_{k+1}=\mathbf{r}_k-\alpha_k A\mathbf{p}_k$
\STATE $\mathbf{z}_{k+1}=M_{\mathrm{ASM}}^{-1}\mathbf{r}_{k+1}$
\STATE $\beta_k = \dfrac{(\mathbf{r}_{k+1},\mathbf{z}_{k+1})}{(\mathbf{r}_k,\mathbf{z}_k)}$
\STATE $\mathbf{p}_{k+1}=\mathbf{z}_{k+1}+\beta_k\mathbf{p}_k$
\ENDFOR
\end{algorithmic}
\end{algorithm}

\begin{overview}{What one-level Schwarz preconditioners achieve — and what they do not}
The analysis above establishes that one-level additive Schwarz preconditioners provide a \emph{sound} and \emph{practically effective} framework for preconditioning SPD systems:
\begin{itemize}
\item \add{they reduce the residual through independent local subdomain solves on overlapping pieces of the domain;}
\item they are naturally parallel, with computational work confined to local problems.
\end{itemize}

However, the convergence bound~\eqref{eq:pcg_bound} also makes explicit that the performance of preconditioned CG depends entirely on the condition number
\(
\kappa(M_{\mathrm{ASM}}^{-1}A).
\)
While this quantity is bounded for a fixed domain decomposition, it generally \emph{deteriorates as the number of subdomains increases}.
\end{overview}

\begin{important}{Local efficiency vs.\ global scalability}
One-level Schwarz methods are designed to be \emph{locally efficient}:
they resolve short-range interactions and high-frequency error modes very effectively.
What they lack is an explicit mechanism for treating \emph{global} or \emph{long-range} components of the error.

As the domain is partitioned into more subdomains, global error modes must propagate across an increasing number of interfaces.
With only local overlap-based communication, this propagation becomes progressively slower,
leading to a growing condition number and, consequently, to increasing Krylov iteration counts.
\end{important}
\add{%
\begin{svgraybox}
\textbf{Position in the solver hierarchy.}
One-level Schwarz preconditioners control high-frequency (local)
error components efficiently.
However, as discussed earlier,
they do not guarantee weak scalability.
In large-scale or high-frequency regimes,
they must be combined with a coarse space correction,
leading to two-level methods.
\end{svgraybox}%
}

%=================================================================
\section{Advanced features: Coarse Space Corrections}
\label{sec:coarse_space_corrections}

One-level overlapping Schwarz methods (\add{ASM/RAS/ORAS}) are excellent \emph{smoothers}: each application of the preconditioner solves independent subdomain problems and exchanges information only through the overlap. This locality is an advantage for parallelism, but it becomes a limitation for \emph{scalability}. 

As the number of subdomains increases (typically because we use more processors), the algorithm must communicate \emph{global} information across a longer chain of subdomain interfaces. Purely local coupling through overlap then propagates global error components only slowly, which translates into increasing iteration counts. 

Coarse space corrections restore scalability by injecting an explicit \emph{global} mechanism: they represent and transport low-frequency (global) modes across the entire domain in a small number of iterations. 

We recall two standard notions of scalability.

\begin{backgroundinformation}{Strong vs.\ weak scalability}
\begin{itemize}
\item \textbf{Strong scalability (Amdahl)}: how the time-to-solution varies when the total problem size is fixed and the number of processors increases.
\item \textbf{Weak scalability (Gustafson)}: how the time-to-solution varies when the problem size per processor is kept (approximately) constant and the number of processors increases.
\end{itemize}
\end{backgroundinformation}
In domain decomposition, weak scalability is the most relevant metric: when we increase the number of processors, we typically increase the number of subdomains while keeping the local problem size roughly fixed. In that regime, \emph{a scalable method must preserve a constant amount of global communication per iteration} and should converge in a number of iterations that does not grow with the number of subdomains. For iterative methods, a practical proxy for weak scalability is therefore:
\begin{equation*}
\text{\emph{weak scalability}} \quad \Longleftrightarrow \quad
\text{iteration count remains bounded}
\ \ \text{(the cost per iteration is stable).}
\end{equation*}
Equivalently, the preconditioned condition number (or spectrum) should remain bounded independently of $N$. We now illustrate the previous discussion with a minimal numerical experiment.
As the number of subdomains increases while keeping the local problem size fixed,
the iteration count of one-level methods grows linearly.

\begin{figure}[H]
\centering
\begin{tabular}{|c|c|c|c|c|}
\hline
\# Subdomains & 8 & 16 & 32 & 64 \\
\hline
ASM Iterations & 18 & 35 & 66 & 128 \\
\hline
\end{tabular}
\caption{Iteration growth with subdomain count (1-level ASM)}
\end{figure}

% Requires in preamble:

\begin{figure}[H]
\centering
\begin{tikzpicture}[
  sub/.style={draw, line width=0.9pt},
  ov/.style={fill=black, opacity=0.08, draw=none},
  larr/.style={-{Latex[length=2.6mm]}, line width=0.9pt},
  garr/.style={-{Latex[length=2.6mm]}, line width=1.3pt, draw=blue!70!black},
  lab/.style={font=\small},
  scale=0.95
]

%------------------------------------------------------------
% Parameters (keep these as plain integers where needed)
%------------------------------------------------------------
\def\N{7}        % number of subdomains
\def\s{1.1}      % subdomain size
\def\dx{1.0}     % shift between subdomains (overlap if dx < s)
\def\yA{0.0}     % y-position (one-level row)
\def\yB{-2.5}    % y-position (two-level row)

%------------------------------------------------------------
% One-level row: draw squares
%------------------------------------------------------------
\foreach \i in {1,...,\N}{
  \pgfmathsetmacro{\x}{(\i-1)*\dx}
  \draw[sub] (\x,\yA) rectangle ++(\s,\s);
  \node[lab] at (\x+0.5*\s,\yA+0.5*\s) {$\Omega_{\i}$};
}

% One-level row: overlaps + local neighbor arrows (only between 1..N-1)
\foreach \i in {1,...,\numexpr\N-1\relax}{
  \pgfmathsetmacro{\x}{(\i-1)*\dx}
  \draw[ov] (\x+\dx,\yA) rectangle ++(\s-\dx,\s);

  % local communication arrows between neighbors (bidirectional)
  \draw[larr] (\x+\s,\yA+0.72*\s) -- (\x+\dx,\yA+0.72*\s);
  \draw[larr] (\x+\dx,\yA+0.28*\s) -- (\x+\s,\yA+0.28*\s);
}

\node[lab, anchor=west] at (0,\yA+\s+0.35) {\textbf{One-level Schwarz: local communication only}};

%------------------------------------------------------------
% Two-level row: draw squares
%------------------------------------------------------------
\foreach \i in {1,...,\N}{
  \pgfmathsetmacro{\x}{(\i-1)*\dx}
  \draw[sub] (\x,\yB) rectangle ++(\s,\s);
  \node[lab] at (\x+0.5*\s,\yB+0.5*\s) {$\Omega_{\i}$};
}

% Two-level row: overlaps + local arrows
\foreach \i in {1,...,\numexpr\N-1\relax}{
  \pgfmathsetmacro{\x}{(\i-1)*\dx}
  \draw[ov] (\x+\dx,\yB) rectangle ++(\s-\dx,\s);

  \draw[larr] (\x+\s,\yB+0.72*\s) -- (\x+\dx,\yB+0.72*\s);
  \draw[larr] (\x+\dx,\yB+0.28*\s) -- (\x+\s,\yB+0.28*\s);
}

%------------------------------------------------------------
% Coarse "global" component + global arrows
%------------------------------------------------------------
\pgfmathsetmacro{\xmid}{0.5*(\N-1)*\dx+0.5*\s}

\draw[blue!70!black, line width=1.1pt, rounded corners]
  (\xmid-1.1,\yB-0.85) rectangle ++(2.2,0.55);
\node[lab, text=blue!70!black] at (\xmid,\yB-0.575) {coarse space};

% global arrows to ALL subdomains (Reviewer 1, p. 17 suggestion)
\foreach \i in {1,...,\N}{
  \pgfmathsetmacro{\xc}{(\i-1)*\dx + 0.5*\s}
  \draw[garr, draw=blue!70!black] % red colour to mark Reviewer 1's revision
    (\xmid,\yB-0.30)
    .. controls (\xmid,\yB+0.15) and (\xc,\yB+0.15) ..
    (\xc,\yB+0.5*\s);
}

\node[lab, anchor=west] at (0,\yB+\s+0.35) {\textbf{Two-level Schwarz: local + global communication}};

%------------------------------------------------------------
% Legend
%------------------------------------------------------------
\begin{scope}[shift={(0,-4.2)}]
  \draw[larr] (0,0) -- (0.9,0);
  \node[lab, anchor=west] at (1.05,0) {local communication};

  \draw[garr] (5.0,0) -- (5.9,0);
  \node[lab, anchor=west, text=blue!70!black] at (6.0,0) {global (coarse) communication};
\end{scope}

\end{tikzpicture}
\caption{Overlapping subdomains on a long domain. One-level Schwarz exchanges information only with neighbors through overlaps, so global effects propagate gradually across the chain. A two-level method adds a coarse component (blue) that provides a global communication channel, enabling rapid transport of low-frequency information independently of the number of subdomains.}
\label{fig:overlap_local_global_comm}
\end{figure}

This growth is a direct consequence of the
\emph{locality of information exchange} in one-level Schwarz methods.
Each application of the preconditioner only couples neighboring subdomains,
so information propagates across the computational domain through a sequence of
local interactions.
As the number of subdomains increases, the distance (in subdomain graph sense)
over which global error components must travel also increases.

From a spectral viewpoint, this results in:
\begin{itemize}
\item an increasing number of global, low-frequency error modes that are weakly affected by local solves;
\item decreasing minimal eigenvalues of the preconditioned operator;
\item a growing condition number and, consequently, increasing Krylov iteration counts.
\end{itemize}

In short, increasing the number of subdomains increases the \emph{need for global information transfer},
but one-level methods only provide \emph{local} transfer.

\subsection{Condition numbers and weak scalability}
\label{subsec:cond_number_weak_scalability}

We now formalize the notion of weak scalability using spectral estimates. In the case of the CG method, convergence is governed by the spectrum of the preconditioned operator, and in particular by its condition number.

Consider a family of domain decompositions indexed by the number of subdomains $N$,
with the local problem size kept fixed.
A Schwarz preconditioner is said to be \emph{weakly scalable} if the number of
iterations required to reach a prescribed accuracy remains bounded as $N$ increases.
For SPD problems, this is equivalent to requiring that the condition number of the
preconditioned operator remains bounded independently of $N$.

\begin{lemma}[Spectral equivalence implies condition number control]
\label{lem:cond_bound}
Let $A$ and $M_{\mathrm{ASM}}$ be SPD. If there exist constants $C_1,C_2>0$ such that,
for all $\mathbf{x}\in\mathbb{R}^n$,
\[
C_1 \,(M_{\mathrm{ASM}}\mathbf{x},\mathbf{x})
\;\le\;
(A\mathbf{x},\mathbf{x})
\;\le\;
C_2 \,(M_{\mathrm{ASM}}\mathbf{x},\mathbf{x}),
\]
then the eigenvalues of $M_{\mathrm{ASM}}^{-1}A$ satisfy
\[
\lambda_{\min}(M_{\mathrm{ASM}}^{-1}A)\ge C_1,
\qquad
\lambda_{\max}(M_{\mathrm{ASM}}^{-1}A)\le C_2,
\qquad
\kappa(M_{\mathrm{ASM}}^{-1}A)\le \frac{C_2}{C_1}.
\]
\end{lemma}

\begin{svgraybox}
\textbf{Key implication.}
Weak scalability is achieved if the constants $C_1$ and $C_2$ can be chosen
independently of the number of subdomains $N$.
Equivalently, the spectrum of $M_{\mathrm{ASM}}^{-1}A$ must remain uniformly bounded.
\end{svgraybox}

\noindent\textbf{Upper bound: local interactions are sufficient.}
We first note that the largest eigenvalue can be controlled purely by local
considerations.

\begin{lemma}[Upper bound via coloring]
\label{lem:lmax_coloring}
Let $\mathrm{col}(\cdot)$ be a coloring of the subdomains using $\mathcal{N}_c$ colors
such that
\[
\mathrm{col}(k)=\mathrm{col}(l)
\quad \Longrightarrow \quad
\add{R_l\, A\, R_k^T = 0.}
\]
Then
\[
\lambda_{\max}(M_{\mathrm{ASM}}^{-1}A)\le \mathcal{N}_c.
\]
\end{lemma}

\add{%
\begin{proof}
We follow the projector-based argument from Saad's
\emph{Iterative Methods for Sparse Linear Systems} \cite{saad}.
Since $A$ is SPD, it defines the energy inner product
$\langle x,y\rangle_A := y^T A\, x$ on $\mathbb{R}^n$.
An operator $T$ is called \emph{$A$-self-adjoint} if $AT = T^T A$,
and the $A$-self-adjoint operators have real eigenvalues, so it makes
sense to speak of their ``spectrum in the $A$-inner product''. Write
\[
M_{\mathrm{ASM}}^{-1}A
=
\sum_{i=1}^N R_i^T A_i^{-1} R_i\, A
=
\sum_{i=1}^N P_i,
\qquad
P_i := R_i^T A_i^{-1} R_i\, A.
\]
A direct computation using $A_i^T=A_i$ shows that $A P_i = P_i^T A$,
so each $P_i$ is $A$-self-adjoint. Moreover $P_i^2 = P_i$, so $P_i$ is
the $A$-orthogonal projector onto $\mathrm{range}(R_i^T)$; its
eigenvalues are in $\{0,1\}$.

Group the subdomains by colour class and set, for
$c=1,\dots,\mathcal{N}_c$,
\[
Q_c := \sum_{\mathrm{col}(i)=c} P_i.
\]
The colouring assumption $R_l A R_k^T=0$ for
$\mathrm{col}(k)=\mathrm{col}(l)$ implies that the ranges
$\mathrm{range}(R_i^T)$ associated with subdomains of the same colour
are $A$-orthogonal. Hence each $Q_c$ is itself an $A$-orthogonal
projector (onto the direct sum of these ranges), and in particular
\[
0 \;\le\; \langle Q_c x, x\rangle_A \;\le\; \langle x, x\rangle_A
\qquad \text{for all } x\in\mathbb{R}^n.
\]
Summing over the $\mathcal{N}_c$ colour classes,
\[
M_{\mathrm{ASM}}^{-1}A
=
\sum_{c=1}^{\mathcal{N}_c} Q_c
\qquad\Longrightarrow\qquad
\langle (M_{\mathrm{ASM}}^{-1}A)\,x, x\rangle_A
\;\le\; \mathcal{N}_c\, \langle x, x\rangle_A.
\]
Since $M_{\mathrm{ASM}}^{-1}A$ is itself $A$-self-adjoint (as a sum of
$A$-self-adjoint operators), the Rayleigh-quotient characterisation of
its largest eigenvalue yields
\[
\lambda_{\max}(M_{\mathrm{ASM}}^{-1}A)
\;=\;
\max_{x\neq 0}
\frac{\langle (M_{\mathrm{ASM}}^{-1}A)\,x, x\rangle_A}{\langle x,x\rangle_A}
\;\le\;
\mathcal{N}_c.
\]
\end{proof}
}%

This estimate shows that the \emph{upper} part of the spectrum is benign and does
not obstruct scalability.

\noindent\textbf{Lower bound: the true bottleneck.}
In contrast, one-level methods generally fail to control the smallest eigenvalue
as the number of subdomains increases.
This degradation is already visible for simple Poisson problems.

\begin{figure}[H]
\centering
\begin{minipage}[c]{0.8\textwidth}
\includegraphics[width=0.32\textwidth]{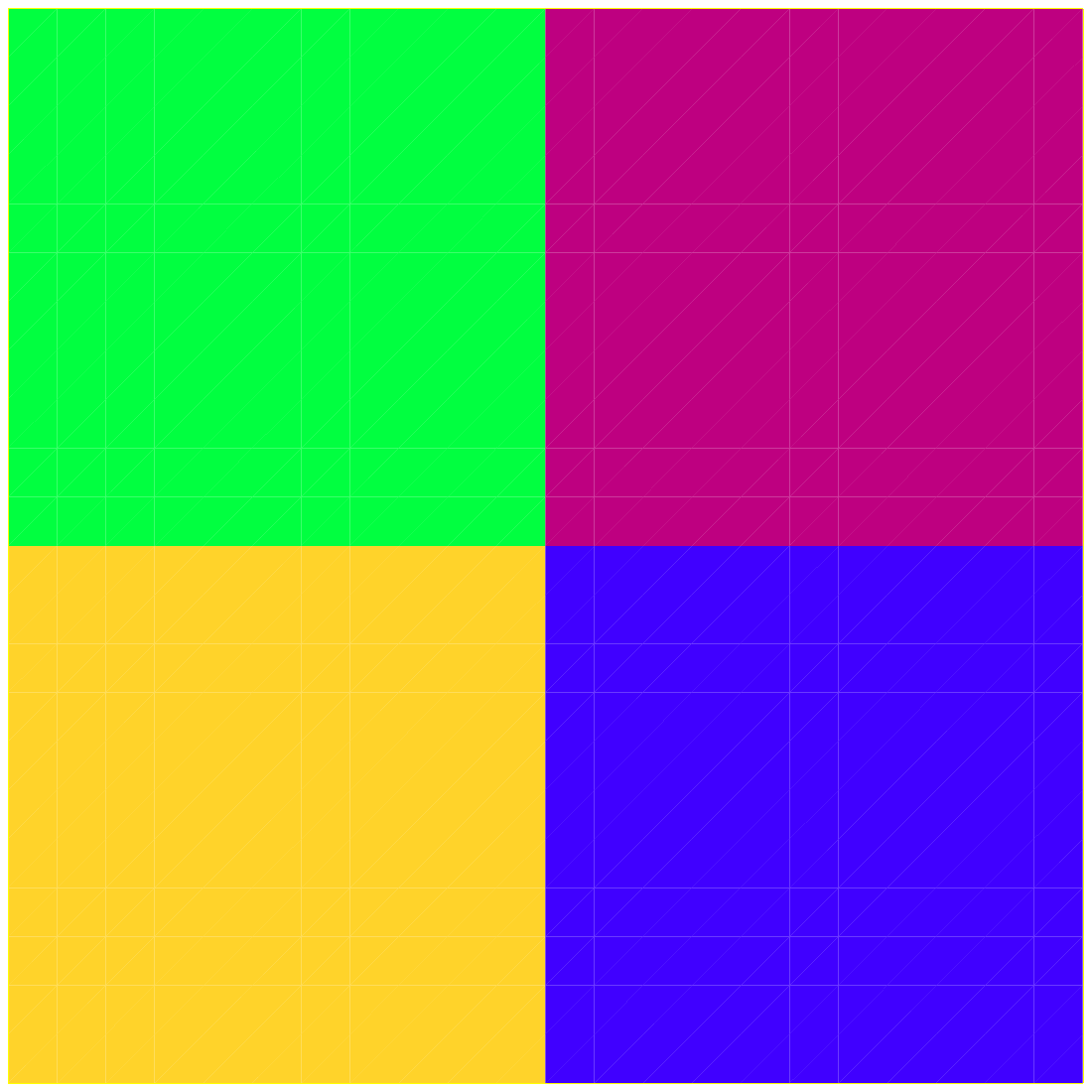}
\includegraphics[width=0.32\textwidth]{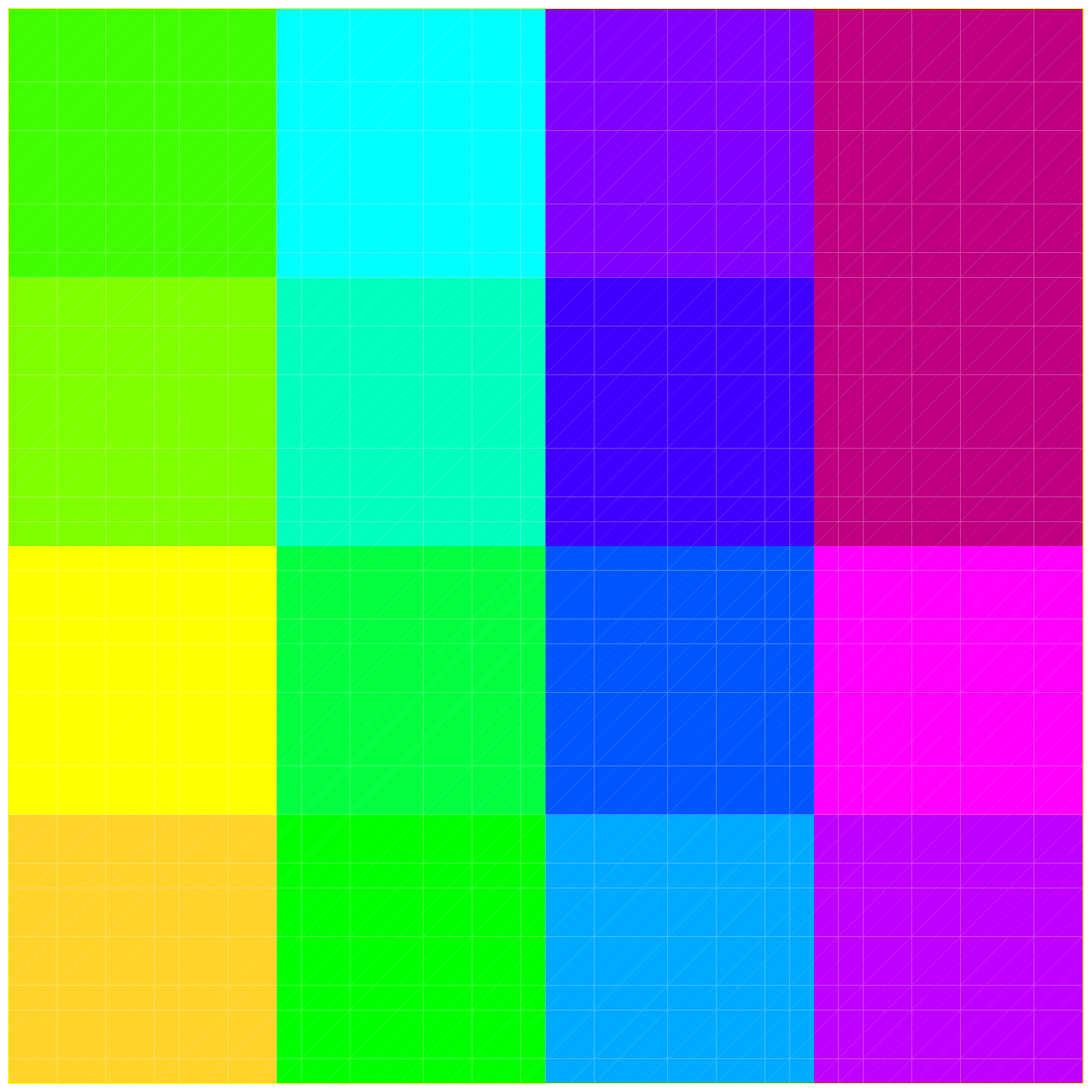}
\includegraphics[width=0.32\textwidth]{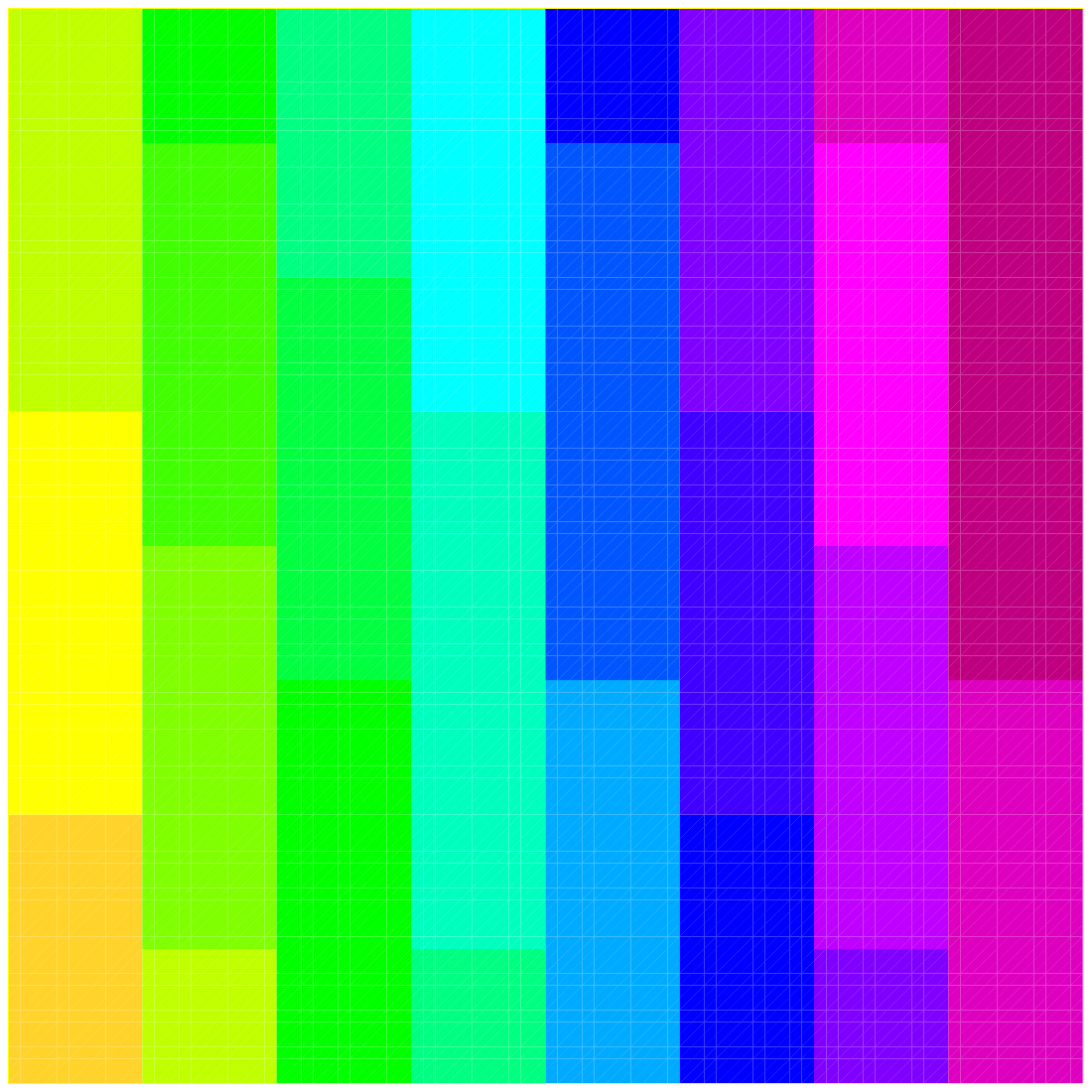}
\end{minipage}\hfill
\begin{minipage}[c]{0.2\textwidth}
\caption{Increasing domain decomposition granularity.}
\label{fig:poisson_granularity}
\end{minipage}
\end{figure}

\begin{table}[!t]
\caption{\add{Poisson problem on a \emph{fixed} $20\times 20$ global grid, two-layer overlap: iteration counts for one-level ASM as the number of subdomains increases (strong scaling regime; the per-subdomain problem size therefore decreases with the number of subdomains).}}
\label{tab:poisson_iter_table}
\centering
\begin{tabular}{lccc}
\hline\noalign{\smallskip}
\# subdomains & $2\times 2$ & $4\times 4$ & $8\times 8$\\
\noalign{\smallskip}\svhline\noalign{\smallskip}
ASM iterations & 20 & 36 & 64\\
\noalign{\smallskip}\hline
\end{tabular}
\end{table}

\begin{important}{Spectral interpretation}
The loss of weak scalability is caused by a small number of
\emph{global, low-frequency modes} that are only weakly affected by local subdomain
solves.
As $N$ increases, these modes lead to smaller values of
$\lambda_{\min}(M_{\mathrm{ASM}}^{-1}A)$,
which in turn drives the growth of the condition number and the iteration count.
\end{important}

\begin{svgraybox}
\textbf{Conclusion.}
One-level Schwarz methods control high-frequency error components efficiently,
but they do not provide a mechanism to uniformly bound the smallest eigenvalues
as the number of subdomains grows.
Restoring weak scalability therefore requires enriching the method with a
\emph{coarse space} that explicitly represents these global modes.
\end{svgraybox}

%-----------------------------------------------------------------
\subsection{Coarse space correction and two-level Schwarz methods}
\label{subsec:coarse_space_correction}

One-level Schwarz methods fail to achieve weak scalability because they do not
efficiently reduce a small number of global error components.
We now introduce a \emph{coarse space correction} that explicitly targets these
slowly converging modes. Consider the linear system \(A\mathbf{x}=\mathbf{b},
\label{eq:Ax=b} \) and let $\mathbf{y}$ be a current approximation with residual \(
\mathbf{r}:=\mathbf{b}-A\mathbf{y}.\)

\begin{backgroundinformation}{Typical slow modes}
The error components that dominate the convergence of one-level methods include:
\begin{itemize}
\item \textbf{Constant functions} for scalar diffusion operators (near-null modes for Neumann problems);
\item \textbf{Rigid body motions} in linear elasticity;
\item \textbf{Near-kernel modes} induced by high-contrast coefficients or weak coupling across subdomains.
\end{itemize}
These modes are global in nature and cannot be efficiently eliminated by purely local subdomain solves.
\end{backgroundinformation}
Let $Z\in\mathbb{R}^{n\times n_c}$ be a matrix whose columns form a basis of a
low-dimensional \emph{space} intended to approximate the slow modes
described above. We seek a correction in the subspace $\mathrm{range}(Z)$ by minimizing the
residual in the energy norm induced by $A^{-1}$.

\begin{lemma}[Galerkin coarse correction]
\label{lem:galerkin_coarse}
The correction $\delta\mathbf{x}=Z\beta$ that minimizes
\[
\|A(\mathbf{y}+Z\beta)-\mathbf{b}\|_{A^{-1}}
\]
is given by
\[
\beta=(Z^TAZ)^{-1}Z^T\mathbf{r}.
\]
\end{lemma}

\begin{proof}
Minimizing the functional
\[
\|A(\mathbf{y}+Z\beta)-\mathbf{b}\|_{A^{-1}}^2
=
(AZ\beta-\mathbf{r},\,AZ\beta-\mathbf{r})_{A^{-1}}
\]
leads to the normal equations
\[
Z^TAZ\,\beta = Z^T\mathbf{r},
\]
which yields the stated expression for $\beta$.
\end{proof}

\begin{definition}[Coarse correction and projection]
The coarse correction is 
\begin{equation}
\delta\mathbf{x}
=
Z(Z^TAZ)^{-1}Z^T\mathbf{r}
=: \add{Q}\,\mathbf{r}.
\label{eq:galerkin_proj}
\end{equation}
\add{The operator $Q$ is the \emph{coarse-solve} operator (mapping a residual to a correction), and is \emph{not} itself a projection. The associated $A$-orthogonal projection onto the coarse space $\mathrm{range}(Z)$ is
\[
\Pi_0 := Q\,A = Z(Z^TAZ)^{-1}Z^T A,
\]
which projects any vector on its $A$-orthogonal component in $\mathrm{range}(Z)$. 
}
\end{definition}
We now combine the coarse correction with the one-level additive Schwarz method.
Define the coarse restriction operator
\[
R_0 := Z^T,
\qquad
A_0 := R_0 A R_0^T = Z^TAZ.
\]

\begin{definition}[Two-level additive Schwarz preconditioner]
\label{def:two_level_asm}
The two-level additive Schwarz preconditioner is defined as
\begin{equation}
M_{\mathrm{ASM},2}^{-1}
=
\underbrace{R_0^T A_0^{-1} R_0}_{\text{coarse solve}}
+
\underbrace{\sum_{i=1}^N R_i^T A_i^{-1} R_i}_{\text{one-level ASM}}.
\label{eq:two_level_asm}
\end{equation}
\end{definition}
The introduction of a suitable coarse space restores weak scalability by
providing an explicit global communication channel that is absent in one-level
methods.
\begin{important}{Interpretation and remarks}
\begin{itemize}
\item The structure of \eqref{eq:two_level_asm} mirrors that of one-level ASM,
augmented by a global correction acting in a low-dimensional space.
\item The coarse problem is small and can be solved redundantly or on a dedicated
coarse communicator with negligible overhead.
\item The key design issue is the choice of the coarse space $Z$:
it must capture the relevant global modes while remaining of modest dimension.
\end{itemize}
\end{important}
 The remainder of this section is devoted to concrete and robust choices
of $Z$.

%-----------------------------------------------------------------
\subsection{The Nicolaides coarse space}
\label{subsec:nicolaides}

The simplest and most classical coarse space for two-level additive Schwarz
methods was introduced by Nicolaides (1987) \cite{Nicolaides1987}. Its guiding principle is that, for scalar diffusion problems, the dominant
slow-to-converge error components are often \emph{nearly constant} over large
portions of the domain.\\

\noindent\textbf{Basic idea.} Each subdomain contributes one coarse basis function that is locally constant
on that subdomain and smoothly blended with its neighbors through a partition
of unity. Together, these functions provide a global mechanism for propagating
constant (or low-frequency) information across the domain.

\begin{definition}[Nicolaides coarse space]
Let $R_i$ be the restriction operator to subdomain $i$ and let $D_i$ be a
diagonal partition-of-unity weight. For each subdomain $i$, define the coarse
basis vector
\begin{equation}
Z_i := R_i^T D_i R_i \mathbf{1},
\qquad i=1,\dots,N,
\label{eq:nico_basis}
\end{equation}
where $\mathbf{1}$ denotes the global all-ones vector.
The global coarse space is spanned by
\[
Z = [\,Z_1,\dots,Z_N\,].
\]
\end{definition}

Each vector $Z_i$ is supported on subdomain $i$ (including its overlap), \add{equals one on the non-overlapping core of $\Omega_i$, and decays smoothly to zero across the overlap region under the weighting by $D_i$.}
The construction relies on the discrete partition-of-unity identity which guarantees that these local constants combine into a globally consistent
correction.

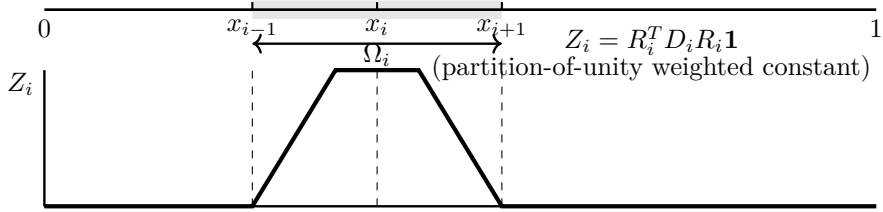
\begin{figure}[H]
\centering
\begin{tikzpicture}[
  x=11cm, y=1.0cm,              % <--- makes [0,1] span 11cm (visible!)
  axis/.style={line width=0.9pt},
  dom/.style={line width=1.0pt},
  basis/.style={line width=1.6pt},
  lab/.style={font=\small},
  tick/.style={line width=0.8pt}
]

% -----------------------------
% Parameters in [0,1]
% -----------------------------
\def\a{0.25}
\def\b{0.40}
\def\c{0.55}
\def\H{1.8}
\def\core{0.10}

% -----------------------------
% Top: domain + subdomains
% -----------------------------
\draw[dom] (0,0) -- (1,0);
\node[lab, below] at (0,0) {$0$};
\node[lab, below] at (1,0) {$1$};

% ticks
\draw[tick] (\a,0) -- (\a,0.10);
\draw[tick] (\b,0) -- (\b,0.10);
\draw[tick] (\c,0) -- (\c,0.10);

\node[lab, below] at (\a,0) {$x_{i-1}$};
\node[lab, below] at (\b,0) {$x_i$};
\node[lab, below] at (\c,0) {$x_{i+1}$};

% overlap highlight
\draw[line width=7pt, opacity=0.10] (\a,0) -- (\b,0);
\draw[line width=7pt, opacity=0.10] (\b,0) -- (\c,0);

% Omega_i support indicator
\draw[<->, line width=0.9pt] (\a,-0.45) -- (\c,-0.45);
\node[lab] at ({0.5*(\a+\c)},-0.60) {$\Omega_i$};

% -----------------------------
% Bottom: basis function Z_i
% -----------------------------
\begin{scope}[shift={(0,-2.6)}]

% axes
\draw[axis] (0,0) -- (1,0);
\draw[axis] (0,0) -- (0,\H);
\node[lab, left] at (0,0.9*\H) {$Z_i$};

% guide lines
\draw[dashed] (\a,0) -- (\a,\H);
\draw[dashed] (\b,0) -- (\b,\H);
\draw[dashed] (\c,0) -- (\c,\H);

% basis: 0 -> ramp up -> plateau -> ramp down -> 0
\pgfmathsetmacro{\bl}{\b-0.5*\core}  % left edge of plateau
\pgfmathsetmacro{\br}{\b+0.5*\core}  % right edge of plateau
\draw[basis]
  (0,0)
  -- (\a,0)
  -- (\bl,\H)
  -- (\br,\H)
  -- (\c,0)
  -- (1,0);

\node[lab, align=center] at (0.73,1.12*\H)
{$Z_i = R_i^T D_i R_i \mathbf{1}$\\[-0.2em]
\small (partition-of-unity weighted constant)};

\end{scope}

\end{tikzpicture}
\caption{Schematic Nicolaides coarse basis function. The vector
$Z_i = R_i^T D_i R_i\mathbf{1}$ \add{equals one on the non-overlapping core of $\Omega_i$ and decays smoothly to zero across the overlap region under the partition-of-unity weights $D_i$.}}
\label{fig:nicolaides_basis}
\end{figure}

\noindent\textbf{Scalability guarantee.}
Although extremely low-dimensional (one basis vector per subdomain), the
Nicolaides coarse space is already sufficient to restore weak scalability for
homogeneous scalar diffusion problems.

\begin{theorem}[Widlund--Dryja bound, \cite{ToselliWidlund2005}]
For scalar diffusion problems with sufficiently regular coefficients, the
two-level additive Schwarz method using the Nicolaides coarse space satisfies
\[
\kappa(M_{\mathrm{ASM},2}^{-1}A)
\;\le\;
C\!\left(1+\frac{H}{\delta}\right),
\]
where $H$ denotes the subdomain diameter and $\delta$ the overlap width.
\end{theorem}

In particular, the condition number is independent of the number of subdomains,
provided that the overlap is chosen proportionally to the subdomain size.The practical impact of the Nicolaides coarse space is illustrated in
Table~\ref{tab:nico_iters}. While the iteration count of the one-level method
grows rapidly with the number of subdomains, the two-level method exhibits
near-constant convergence behavior.

\begin{table}[!t]
\caption{\add{Iteration counts for the two-level additive Schwarz method with and without the Nicolaides coarse correction.}}
\label{tab:nico_iters}
\centering
\begin{tabular}{lcccc}
\hline\noalign{\smallskip}
\# subdomains & 8 & 16 & 32 & 64\\
\noalign{\smallskip}\svhline\noalign{\smallskip}
ASM (one-level) & 18 & 35 & 66 & 128\\
ASM + Nicolaides & 20 & 27 & 28 & 27\\
\noalign{\smallskip}\hline
\end{tabular}
\end{table}

\begin{important}{Limitation of the Nicolaides coarse space}
The Nicolaides space is not robust for problems with strong coefficient
heterogeneity or pronounced multiscale features.
In such cases, constant modes are no longer sufficient to represent the
relevant global error components.
Robustness then requires \emph{adaptive} coarse spaces, typically constructed
from local spectral problems (e.g.\ GenEO).
\end{important}

\subsection{Heterogeneous problems: the GenEO coarse space}
\label{sec:coarse_spaces_heterogeneous}

Many applications involve strong coefficient heterogeneity, which severely
deteriorates the conditioning of the discrete operator and renders simple
coarse spaces ineffective.
A prototypical example is the Darcy pressure equation discretized with
$P^1$ finite elements,
\begin{equation}
A u = f,
\qquad
\operatorname{cond}(A)\ \sim\ \frac{\alpha_{\max}}{\alpha_{\min}}\, h^{-2},
\label{eq:darcy_cond}
\end{equation}
where $\alpha(x)$ denotes the permeability coefficient.
The condition number grows both under mesh refinement ($h\to 0$) and with
increasing coefficient contrast $\alpha_{\max}/\alpha_{\min}$.

Such behavior motivates the design of coarse spaces \cite{Nataf2011} that adapt to the coefficient distribution and remain effective in the presence of arbitrarily
large jumps.

Typical application domains include porous media flow with layered or
stochastic permeability, elasticity with stiff inclusions, and wave propagation
in composite materials.

\begin{figure}[H]
\centering
\begin{minipage}{0.46\textwidth}
  \centering
  \includegraphics[width=\linewidth]{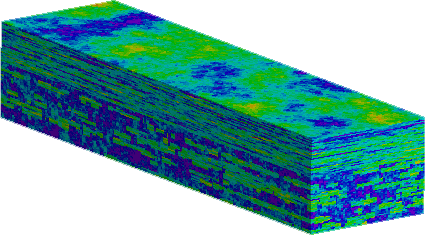}
\end{minipage}\hfill
\begin{minipage}{0.22\textwidth}
  \centering
  \includegraphics[width=\linewidth]{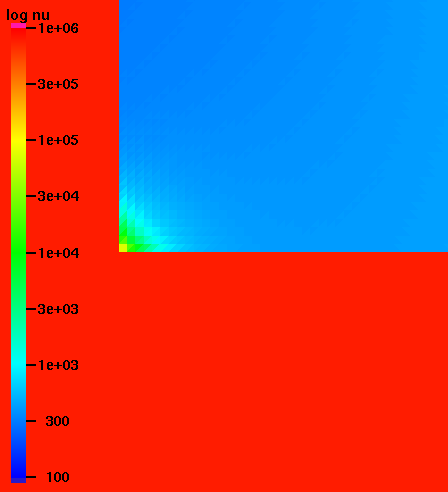}
\end{minipage}\hfill
\begin{minipage}{0.22\textwidth}
  \centering
  \includegraphics[width=\linewidth]{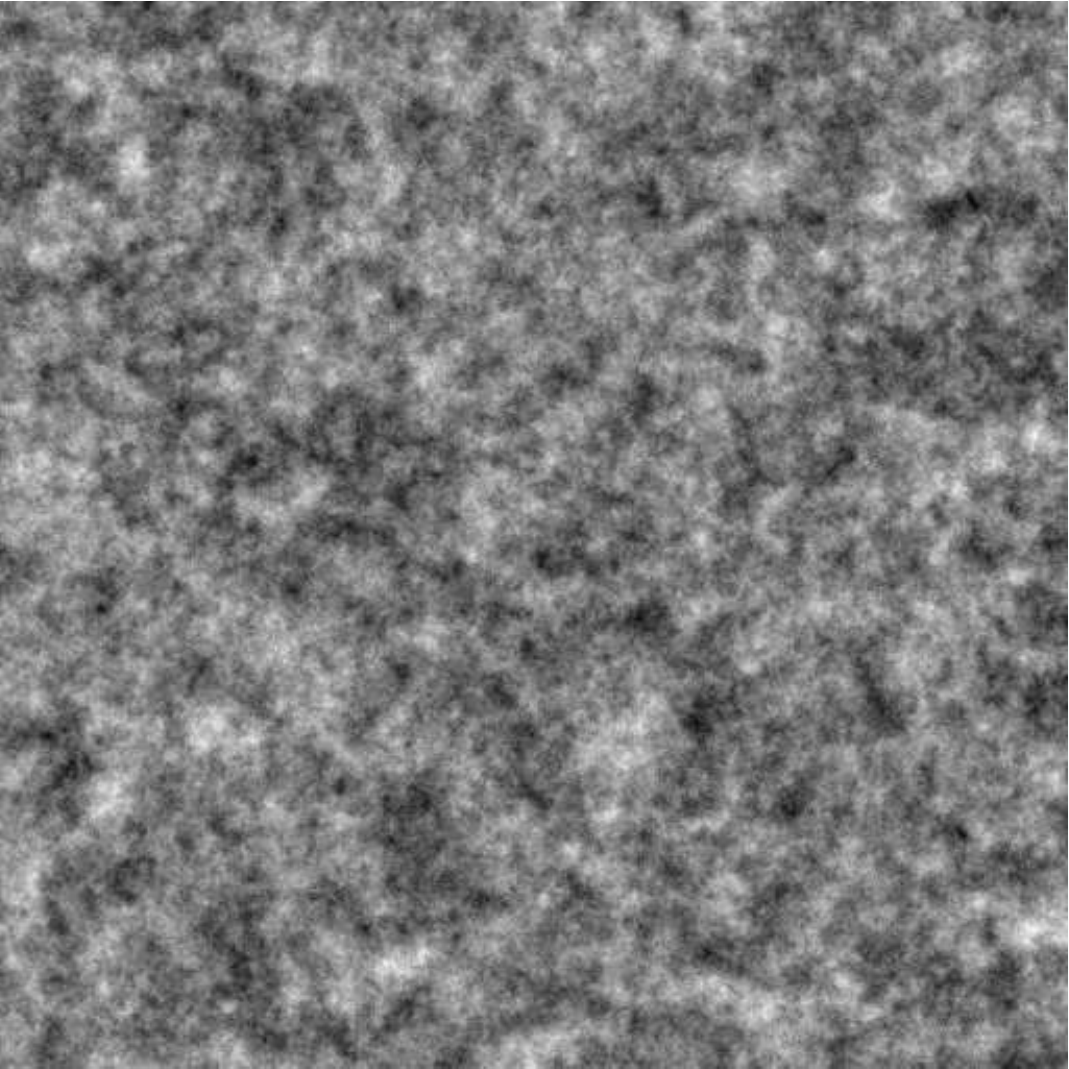}
\end{minipage}
\caption{Examples of heterogeneous coefficients that challenge classical coarse spaces.}
\label{fig:hetero_examples}
\end{figure}

GenEO (Generalized Eigenproblems in the Overlap) \cite{Spillane2014} constructs a coarse space by
identifying, on each subdomain, error components that are poorly reduced by
local solves.
These components are detected by solving local generalized eigenproblems and
selecting the corresponding low-energy modes.\\

\noindent \textbf{Coarse space construction} For each subdomain $j$, GenEO considers the generalized eigenproblem
\begin{equation}
A^{\mathrm{Neu}}_j\,\phi_{j,k}
=
\lambda_{j,k}\,
\bigl(D_j R_j A R_j^T D_j\bigr)\,\phi_{j,k},
\label{eq:geneo_local_eig}
\end{equation}
where $A^{\mathrm{Neu}}_j$ denotes the Neumann operator on subdomain $j$,
$R_j$ is the restriction operator, and $D_j$ is a partition-of-unity weight. The eigenvectors $\phi_{j,k}$ corresponding to small eigenvalues represent
local modes that are weakly controlled by one-level Schwarz methods.\\

\noindent \textbf{Mode selection} Given a threshold $\tau>0$, the GenEO coarse space is defined by selecting all
eigenvectors satisfying
\begin{equation}
\lambda_{j,k}\le \tau
\quad\Longrightarrow\quad
R_j^T D_j \phi_{j,k}\ \in\ Z_{\mathrm{GenEO}}.
\label{eq:geneo_selection}
\end{equation}
The threshold $\tau$ controls the balance between robustness and coarse space
dimension. In particular, the choice $\tau=0$ recovers constant-type modes,
leading to a Nicolaides-like coarse space.\\

The GenEO construction yields a provably robust two-level preconditioner.

\begin{theorem}[Spillane et al., 2014, \cite{Spillane2014}]
Under standard assumptions on the overlap and subdomain partition, the
two-level additive Schwarz method augmented with the GenEO coarse space satisfies
\[
\kappa(M_{\mathrm{ASM},2}^{-1}A)
\ \le\
(1+k_0)\Bigl[\,2+k_0(2k_0+1)\bigl(1+1/\tau\bigr)\Bigr],
\]
where $k_0$ denotes the maximal overlap multiplicity.
\end{theorem}

A practical and robust choice of the spectral threshold is
\[
\tau := \left(\max_j \frac{H_j}{\delta_j}\right)^{-1},
\]
where $H_j$ and $\delta_j$ denote the diameter and overlap width of subdomain $j$,
respectively.
With this choice, the condition number bound becomes independent of coefficient
contrast and mesh size.

\begin{important}{Nicolaides vs. GenEO}
In summary, the Nicolaides coarse space provides a minimal global correction
sufficient for homogeneous scalar problems, while GenEO generalizes this idea by
selecting, through local spectral problems, the coefficient-dependent modes that
govern robustness in heterogeneous and nearly singular settings.
\end{important}

%-----------------------------------------------------------------
We illustrate the performance of GenEO on elliptic problems with strong
heterogeneity and near-null modes, focusing on robustness with respect to
coefficient contrast, domain decomposition, and coarse space size. We consider first a Darcy problem with permeability contrast \(
1 \le \alpha(x) \le 1.5\times 10^6, \)
a regime in which one-level and classical coarse spaces typically fail.
Two domain decompositions are used: a regular Cartesian partition and a
graph-based partition obtained with METIS.

Figure~\ref{fig:darcy_hetero_geneo} shows the heterogeneous coefficient field,
the corresponding pressure solution, and the two decompositions.
The results demonstrate that GenEO remains robust independently of both
coefficient variations and the choice of partitioning.

\begin{figure}[H]
\centering
\begin{minipage}{0.24\textwidth}
  \centering
  \includegraphics[width=\linewidth]{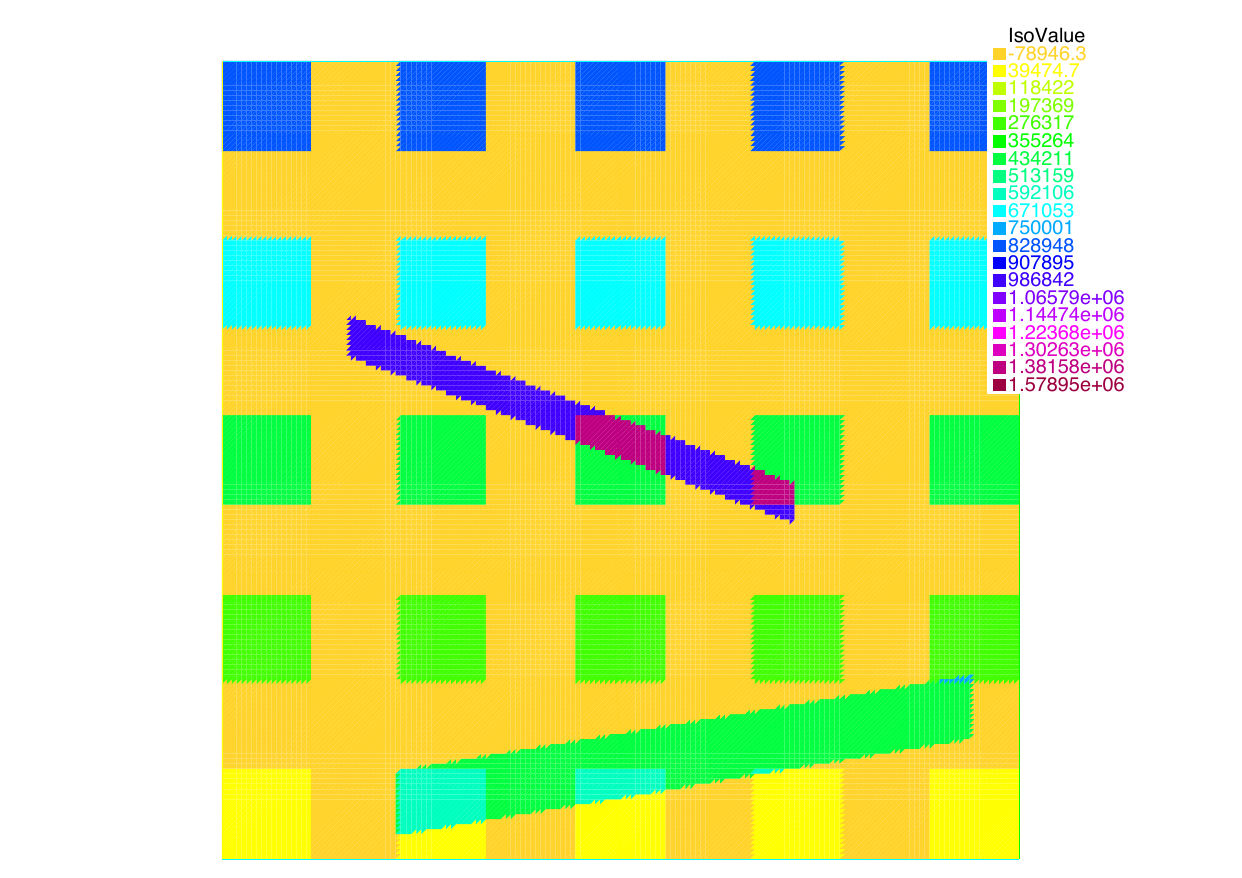}
\end{minipage}\hfill
\begin{minipage}{0.24\textwidth}
  \centering
  \includegraphics[width=\linewidth]{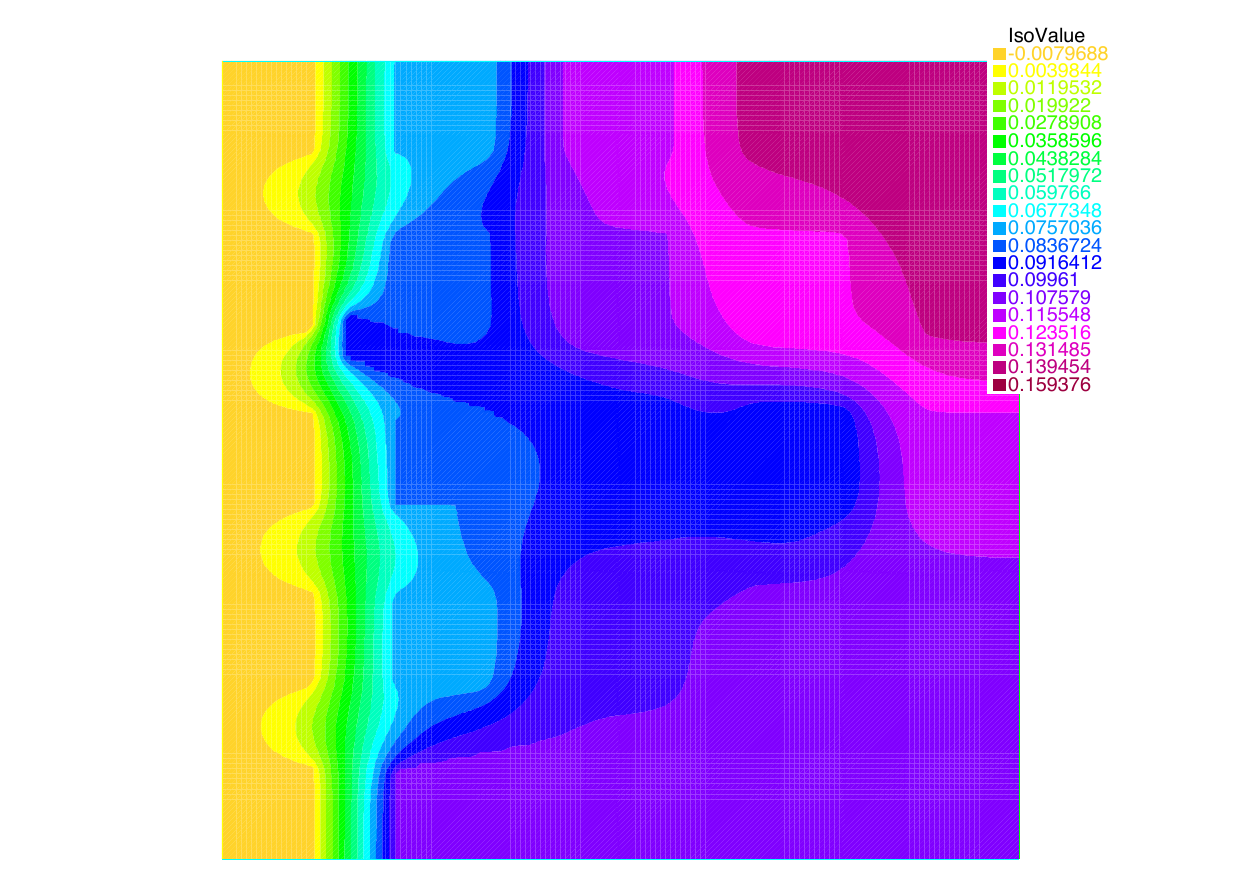}
\end{minipage}\hfill
\begin{minipage}{0.24\textwidth}
  \centering
  \includegraphics[width=\linewidth]{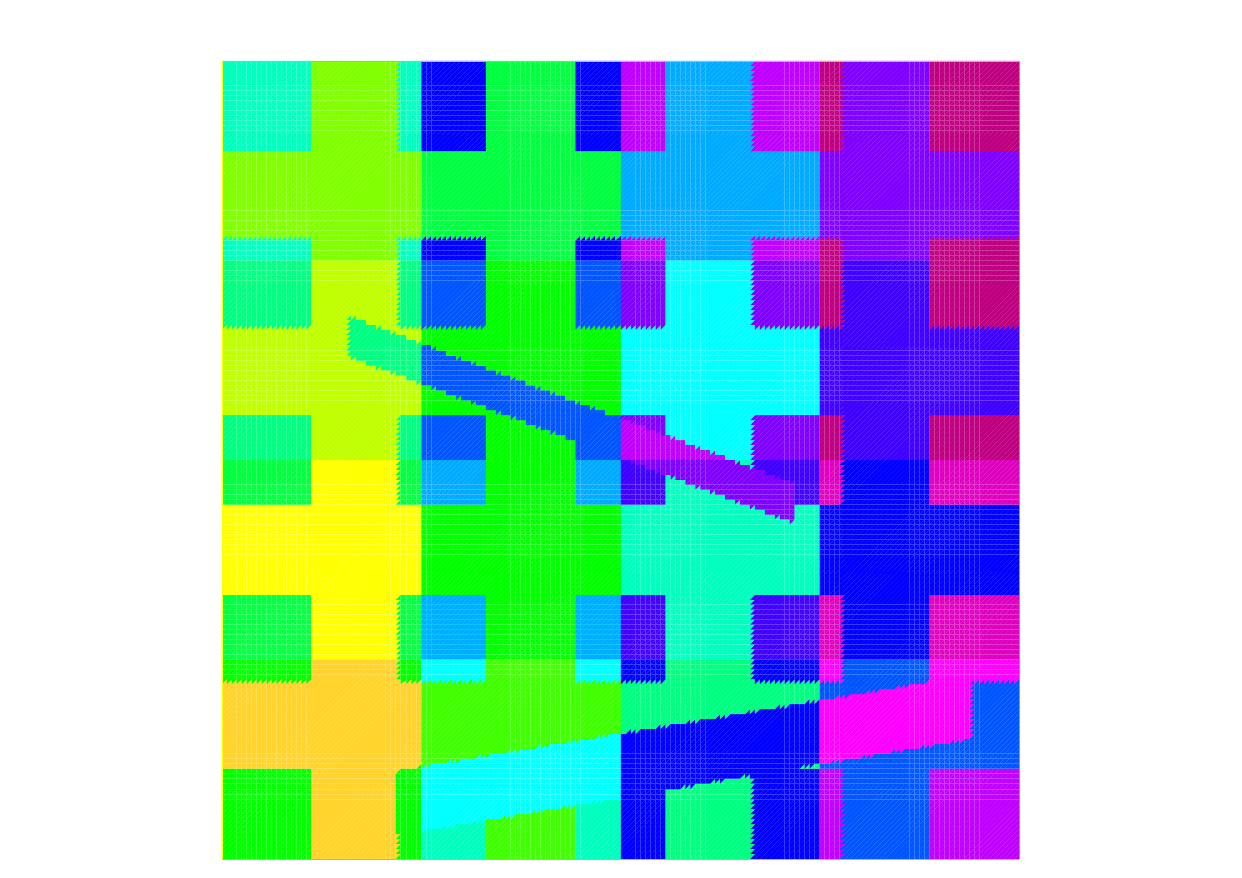}
\end{minipage}\hfill
\begin{minipage}{0.24\textwidth}
  \centering
  \includegraphics[width=\linewidth]{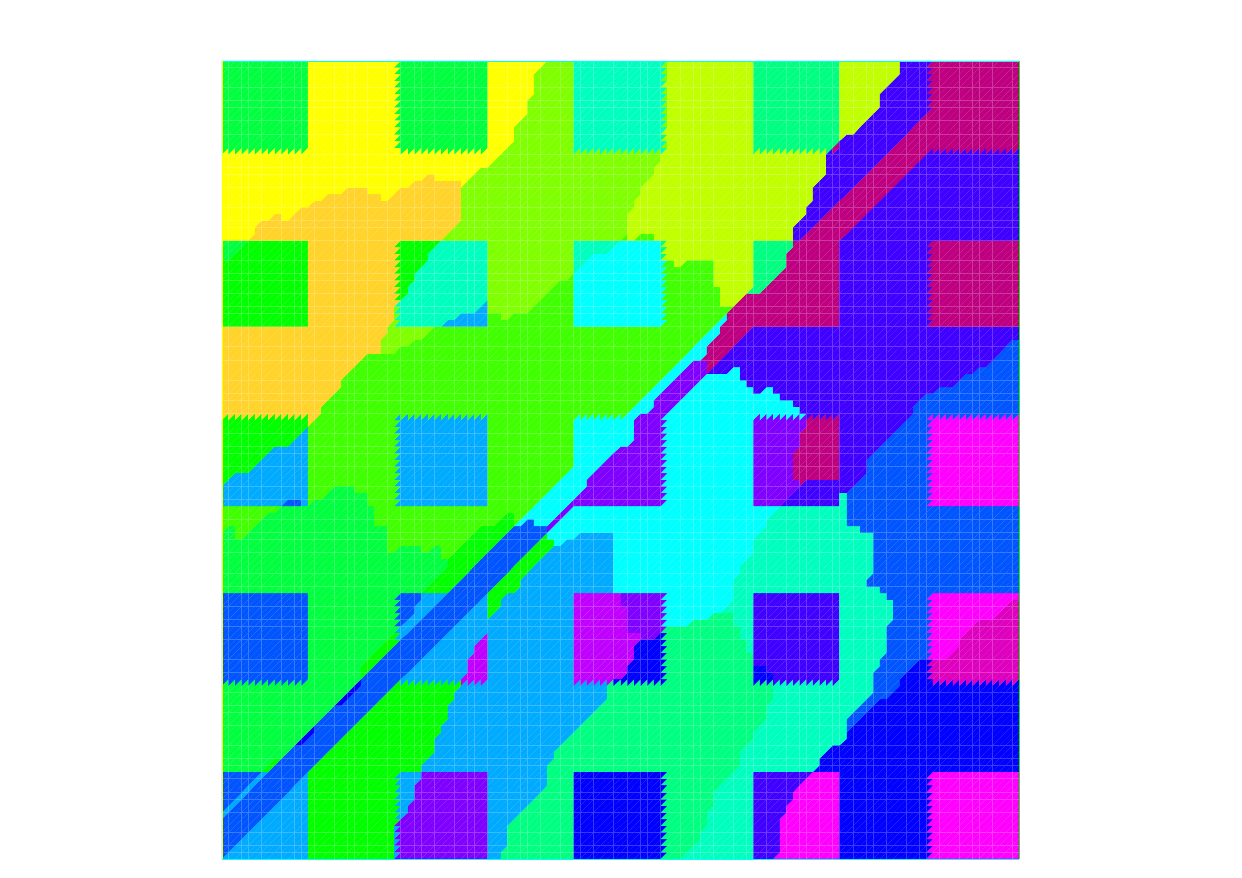}
\end{minipage}
\caption{Darcy problem with strong heterogeneity: permeability field, solution,
and uniform versus METIS-based domain decompositions.}
\label{fig:darcy_hetero_geneo}
\end{figure}

\add{Second, elasticity with strongly heterogeneous coefficients and near-null modes (e.g.\ weakly constrained inclusions, layered media) severely slows down classical Schwarz methods. GenEO, applied to the primal formulation, automatically detects these modes and incorporates them into the coarse space, leading to fast convergence that is essentially independent of the heterogeneity contrast. The genuinely \emph{nearly incompressible} regime ($\nu \to 0.5$) requires in addition a mixed formulation and a dedicated treatment of the pressure block; this saddle-point setting is discussed separately in Section~\ref{sec:elast_NumericalResults}.}

\begin{figure}[h!]
\centering
\begin{minipage}[c]{0.75\textwidth}
\includegraphics[width=0.8\linewidth]{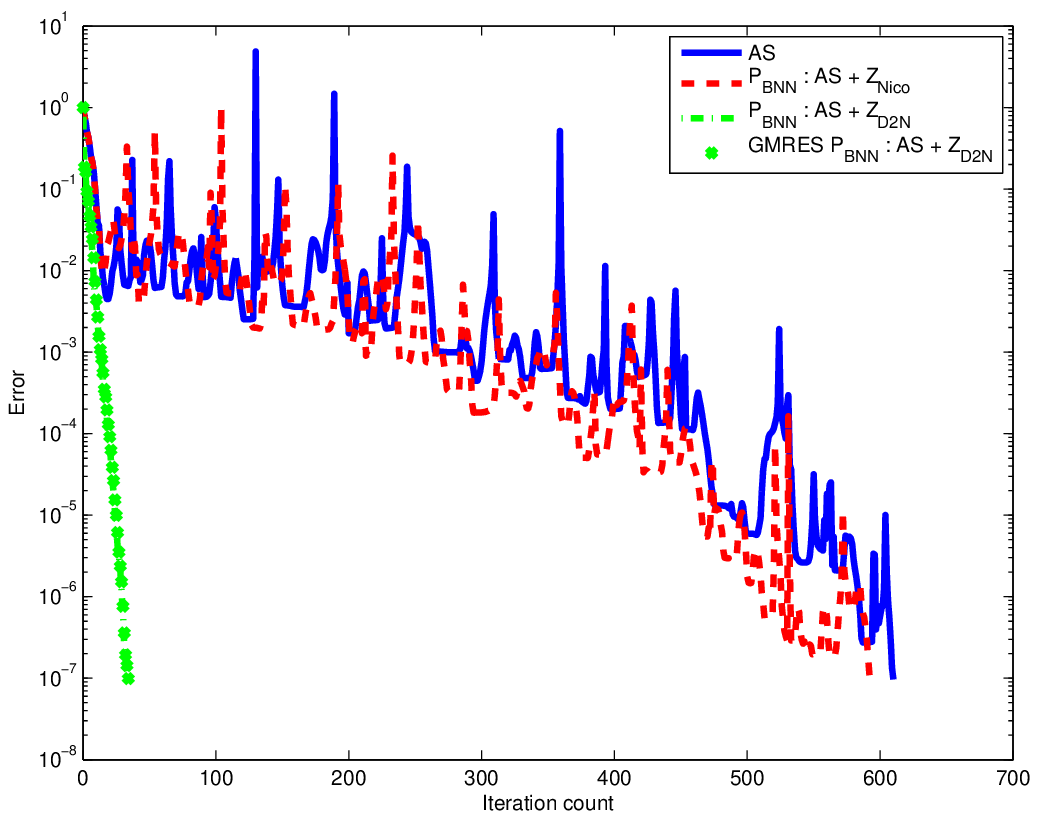}
\end{minipage}
\begin{minipage}[c]{0.2\textwidth}
\caption{\add{GMRES convergence in heterogeneous (not nearly incompressible) elasticity.} GenEO restores rapid convergence compared to ASM and Nicolaides coarse spaces.}
\label{fig:elasticity_convergence_geneo}
\end{minipage}
\end{figure}

\noindent\textbf{Quantitative Comparison.} The following table compares the number of eigenvectors per subdomain selected under three configurations:
\begin{itemize}
  \item No coarse space (pure ASM)
  \item Nicolaides (1 vector per subdomain)
  \item GenEO (adaptive spectral selection)
\end{itemize}
We notice there is an optimal (empirical) value for the number of vectors beyond which there is no further gain in iteration count.
\begin{table}[h!]
\centering
\begin{tabular}{|c|c|c|c|}
\hline
Coarse Space Type & ASM & ASM + $Z_{\text{Nico}}$ & ASM + $Z_{\text{GenEO}}$ \\
\hline
$\max(m_i - 1, 1)$ & – & – & 273 \\
$m_i$ (average)     & 614 & 543 & \textbf{36} \\
$m_i + 1$           & – & – & 32 \\
\hline
\end{tabular}
\caption{Number of local coarse basis functions $m_i$ based on eigenvalue threshold.}
\end{table}

This reduction in iteration count highlights GenEO’s ability to represent difficult modes efficiently thus leading to a better convergence.\\

\noindent\textbf{Spectral Distribution.} The spectrum confirms that only a few modes dominate. These are precisely the ones retained by GenEO. The rest decay rapidly and do not affect convergence significantly.

\begin{figure}[H]
\centering
\begin{minipage}[c]{0.75\textwidth}
\includegraphics[width=0.8\linewidth]{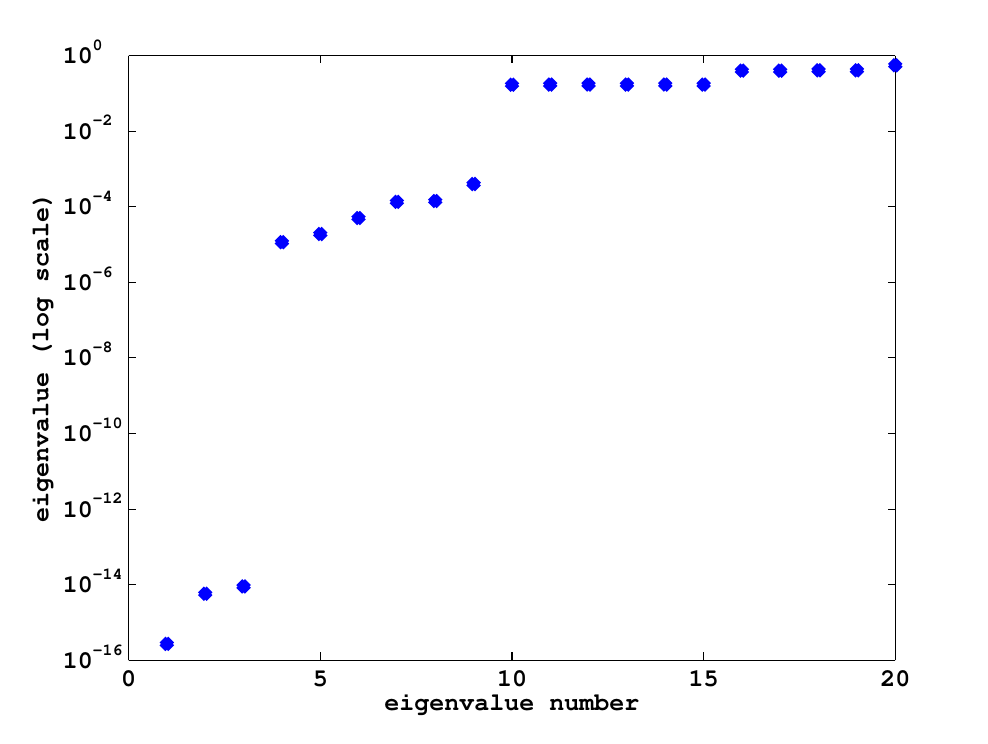}
\end{minipage}
\begin{minipage}[c]{0.2\textwidth}
\caption{Typical local GenEO spectrum (log scale): only a few low-energy modes
are selected for the coarse space.}
\label{fig:geneo_spectrum}
\end{minipage}
\end{figure}

These experiments confirm that GenEO achieves robustness with respect to
heterogeneity, near-null modes, and partitioning, while keeping the coarse
space dimension small.

\subsection{The fictitious space lemma: an abstract framework}
\label{subsec:fsl}

The analysis of Schwarz-type preconditioners can be cast in a unified form by introducing auxiliary (``fictitious'') spaces.
This abstract viewpoint isolates the spectral properties of the
preconditioner from the details of the discretization \add{and provides a template that, in its standard symmetric form, directly covers additive Schwarz (ASM) and its symmetrised variants (e.g.\ SORAS). The non-symmetric variants RAS and ORAS are \emph{not} covered by the symmetric form of the lemma; analyses of these methods typically require additional ingredients (see Remark~\ref{rem:fsl_ras_oras} below).}\\

\noindent\textbf{Global and product spaces.}
We distinguish between the global space associated with the discretized
problem and an auxiliary product space collecting local subdomain variables.
Let
\(
\mathcal{H}_0 := \mathbb{R}^{\#\mathcal{N}}
\)
denote the global finite-dimensional space, equipped with the bilinear form
\[
a(\mathbf{U},\mathbf{V}) := \mathbf{V}^T A \mathbf{U},
\]
where $A$ is the global system matrix.
The fictitious (product) space is defined as
\[
\mathcal{H}_P := \prod_{i=1}^N \mathbb{R}^{\#\mathcal{N}_i},
\]
with one local vector per subdomain.
On $\mathcal{H}_P$ we introduce the block-diagonal bilinear form
\[
b(\mathcal{U},\mathcal{V})
:=
\sum_{i=1}^N \mathbf{V}_i^T A_i \mathbf{U}_i,
\quad
A_i := R_i A R_i^T.
\]

\noindent\textbf{Assembly operator.}
The connection between the fictitious and global spaces is provided by the
(additive) assembly operator
\[
\mathcal{R}_{\mathrm{ASM}} : \mathcal{H}_P \to \mathcal{H}_0,
\qquad
\mathcal{R}_{\mathrm{ASM}}(\mathcal{U}) := \sum_{i=1}^N R_i^T \mathbf{U}_i.
\]

With this notation, the one-level additive Schwarz preconditioner can be written
in the abstract form
\[
M_{\mathrm{ASM}}^{-1}
=
\mathcal{R}_{\mathrm{ASM}}\, B^{-1}\, \mathcal{R}_{\mathrm{ASM}}^{*},
\]
where $B$ is the block-diagonal operator associated with $b(\cdot,\cdot)$. The following result, often referred to as the fictitious space lemma,
provides spectral characterisation of the preconditioners of  the form $M^{-1}=\mathcal{R} B^{-1}\mathcal{R}^*$.

\begin{lemma}[Fictitious space lemma]
\label{lem:fsl}
Assume that the operator $\mathcal{R}:\mathcal{H}_P\to\mathcal{H}_0$ satisfies:
\begin{itemize}
\item[(i)] {Surjectivity:}
for every $u\in\mathcal{H}_0$ there exists $u_P\in\mathcal{H}_P$ such that
$\mathcal{R}u_P=u$;
\item[(ii)] $\,$ {Continuity:}
there exists $c_R>0$ such that
\[
a(\mathcal{R}u_P,\mathcal{R}u_P)
\le c_R\, b(u_P,u_P)
\qquad \forall u_P\in\mathcal{H}_P;
\]
\item[(iii)] $\,$ {Stable decomposition:}
there exists $c_T>0$ such that for every $u\in\mathcal{H}_0$ one can find
$u_P\in\mathcal{H}_P$ with $\mathcal{R}u_P=u$ and
\[
c_T\, b(u_P,u_P) \le a(u,u).
\]
\end{itemize}
Then the spectrum of the preconditioned operator satisfies
\[
\mathrm{spec}(M^{-1}A)\subset [\,c_T,\; c_R\,].
\]
\end{lemma}

This lemma isolates the analysis of Schwarz preconditioners into three
independent ingredients: geometric overlap (surjectivity), local stability
(continuity), and the ability to decompose global functions into local ones
(stable decomposition).

%-----------------------------------------------------------------

\subsubsection*{Application to one-level Schwarz methods}

We now apply Lemma~\ref{lem:fsl} to classical one-level additive Schwarz
preconditioners.

\noindent\textbf{Continuity estimate.}
The continuity constant is controlled by the coloring number
$\mathcal{N}_c$, i.e.,
\[
c_R = \mathcal{N}_c.
\]

\noindent\textbf{Stable decomposition.}
The stability constant is bounded in terms of a local Rayleigh quotient,
\[
\tau_1 :=
\min_i \min_{\mathbf{U}_i\neq 0}
\frac{\mathbf{U}_i^T A_i^{\mathrm{Neu}} \mathbf{U}_i}
     {\mathbf{U}_i^T D_i R_i A R_i^T D_i \mathbf{U}_i},
\]
where $A_i^{\mathrm{Neu}}$ denotes the Neumann operator on subdomain $i$.
If $\mathcal{M}_c$ denotes the maximal overlap multiplicity, then
\[
c_T = \frac{\tau_1}{\mathcal{M}_c}.
\]

Combining these estimates yields the spectral bound
\[
\frac{\tau_1}{\mathcal{M}_c}
\;\le\;
\lambda(M_{\mathrm{ASM}}^{-1}A)
\;\le\;
\mathcal{N}_c.
\]

\noindent\textbf{Remark.}
For problems with strong coefficient heterogeneity, the quantity $\tau_1$
can become very small, explaining the loss of robustness of one-level methods.

%-----------------------------------------------------------------

\subsubsection*{Extension to other Schwarz variants}

\add{The same framework also applies, with minor modifications, to the \emph{symmetric} variants of optimised Schwarz methods (e.g.\ SORAS). For the non-symmetric restricted and optimised methods RAS and ORAS, analogous spectral bounds are available only under additional structural assumptions and are not derived here.}
For instance, the restricted additive Schwarz (RAS) and optimized RAS (ORAS)
preconditioners read
\[
M_{\mathrm{RAS}}^{-1} = \sum_{i=1}^N R_i^T D_i A_i^{-1} R_i,
\qquad
M_{\mathrm{ORAS}}^{-1} = \sum_{i=1}^N R_i^T D_i B_i^{-1} R_i,
\]
where $B_i$ corresponds to a Robin-type subdomain operator.

\add{%
\begin{remark}\label{rem:fsl_ras_oras}
The symmetric form of the fictitious space lemma (Lemma~\ref{lem:fsl}) requires a symmetric preconditioner of the form $M^{-1}=\mathcal{R}B^{-1}\mathcal{R}^{*}$ with a self-adjoint $B$. RAS and ORAS do not fit this template directly; their analysis usually relies on a symmetrisation step (recovering SORAS) or on convergence results for non-stationary iterations. We refer the reader to~\cite{Dolean2015} for a discussion of these subtleties.
\end{remark}%
} 
\subsubsection*{Spectral bounds for SORAS}
Among symmetric variants we can also mention the SORAS preconditioner,
\[
M_{\mathrm{SORAS}}^{-1}
=
\sum_{i=1}^N R_i^T D_i B_i^{-1} D_i R_i.
\]
Define the bilinear forms
\[
a(\mathbf{U},\mathbf{V}) := \mathbf{V}^T A \mathbf{U},
\qquad
b(\mathcal{U},\mathcal{V}) := \sum_{i=1}^N \mathbf{V}_i^T B_i \mathbf{U}_i,
\]
and the assembly operator
\[
\mathcal{R}_{\mathrm{SORAS}}(\mathcal{U})
:=
\sum_{i=1}^N R_i^T D_i \mathbf{U}_i.
\]
Then
\[
M_{\mathrm{SORAS}}^{-1}
=
\mathcal{R}_{\mathrm{SORAS}}\, B^{-1}\, \mathcal{R}_{\mathrm{SORAS}}^{*}.
\]

Applying Lemma~\ref{lem:fsl} yields the bounds
\[
\frac{\tau_1}{\mathcal{M}_c}
\;\le\;
\lambda(M_{\mathrm{SORAS}}^{-1}A)
\;\le\;
\mathcal{N}_c\, \gamma_1,
\]
where
\[
\gamma_1 :=
\max_i \max_{\mathbf{U}_i\neq 0}
\frac{(R_i^T D_i \mathbf{U}_i)^T A (R_i^T D_i \mathbf{U}_i)}
     {\mathbf{U}_i^T B_i \mathbf{U}_i}.
\]

These estimates highlight the role of the local subdomain solvers $B_i$ in
controlling the upper part of the spectrum.

\subsection{Numerical results}
\label{sec:elast_NumericalResults}

A critical test for the robustness of preconditioners is the simulation of nearly incompressible elastic materials. In such settings, standard iterative methods often face severe difficulties due to the ill-conditioning introduced by the incompressibility constraint.

\add{The matrix $A$ arising from the Taylor--Hood discretisation below is symmetric \emph{indefinite}, so the standard GenEO spectral theory (which requires $A$ SPD) does not strictly apply. The results reported in this section should therefore be regarded as numerical evidence of the practical robustness of GenEO-type coarse spaces in the indefinite regime; extensions of the theory to saddle-point and mixed formulations are an active research direction (see, e.g.,~\cite{Bootland2022, DoleanFryLanger2026}).}

We consider a two-material model composed of steel and rubber:
\begin{itemize}
  \item Steel: $(E_1, \nu_1) = (210 \times 10^9, 0.3)$
  \item Rubber: $(E_2, \nu_2) = (0.1 \times 10^9, 0.4999)$
\end{itemize}

The resulting PDE system is discretized using Taylor–Hood mixed finite elements, i.e., $\mathbb{P}_2^d$–$\mathbb{P}_1$ pairs for displacement and pressure. The weak formulation is saddle-point structured, leading to a symmetric but indefinite matrix system:
\[
A =
\begin{bmatrix}
H & B^T \\
B & -C
\end{bmatrix}
\]

\begin{figure}[H]
\centering
\includegraphics[width=0.6\textwidth]{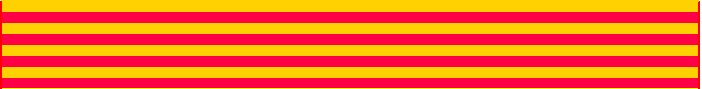}
\caption{2D sandwich: steel core and rubber layers.}
\end{figure}

The domain is partitioned using METIS for optimal subdomain layout:

\begin{center}
\includegraphics[width=0.6\textwidth]{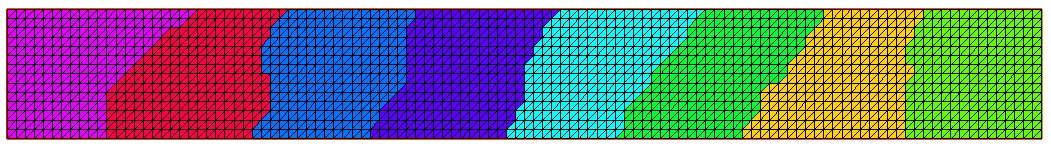}
\end{center}

We test various preconditioners with and without coarse spaces. The results below show the number of GMRES iterations required for convergence, along with the dimension of the coarse space used in each case:

\begin{table}[H]
\caption{GMRES iteration counts for 2D elasticity with increasing subdomains}
\centering
\tiny
\begin{tabular}{|r|c||c||c||c|c||c|c||c|c||c|c|}
\hline
DOFs & Subdomains & AS & SORAS & AS+ZEM & dim & SORAS+ZEM & dim & AS+GenEO & dim & SORAS+GenEO-2 & dim \\
\hline
35841   & 8   & 150   & 184   & 117   & 24   & 79    & 24   & 110   & 184   & 13    & 145   \\
70590   & 16  & 276   & 337   & 170   & 48   & 144   & 48   & 153   & 400   & 17    & 303   \\
141375  & 32  & 497   & >1000 & 261   & 96   & 200   & 96   & 171   & 800   & 22    & 561   \\
279561  & 64  & >1000 & >1000 & 333   & 192  & 335   & 192  & 496   & 1600  & 24    & 855   \\
561531  & 128 & >1000 & >1000 & 329   & 384  & 400   & 384  & >1000 & 2304  & 29    & 1220  \\
1077141 & 256 & >1000 & >1000 & 369   & 768  & >1000 & 768  & >1000 & 3840  & 36    & 1971  \\
\hline
\end{tabular}
\end{table}

The results clearly demonstrate:
\begin{itemize}
  \item \textbf{Standard ASM and SORAS fail} to converge beyond 64 subdomains.
  \item Adding a coarse space (ZEM or GenEO) drastically reduces iteration counts.
  \item GenEO delivers both robustness and scalability even at 256 subdomains.
\end{itemize}

We now test the strong and weak scalability of the proposed methods on large-scale problems.\\

\noindent\textbf{Strong Scalability: Stokes in 3D.} We consider a driven cavity problem for the Stokes system with automatic mesh partitioning.

\begin{center}
\includegraphics[width=0.85\textwidth]{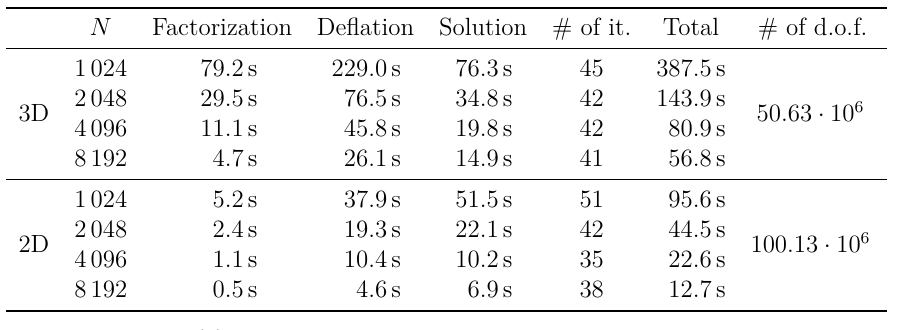}
\end{center}

\noindent\textbf{Weak Scalability: Heterogeneous Elasticity in 3D} We also test weak scalability using a steel–rubber sandwich geometry and automatic mesh partitioning.

\begin{center}
\includegraphics[width=0.63\textwidth]{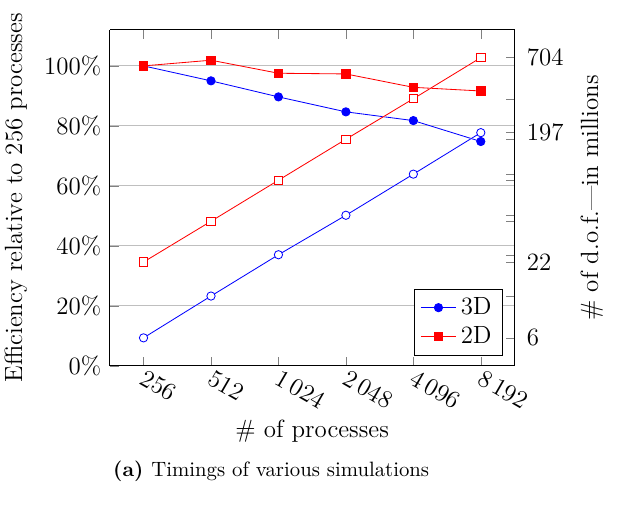}
\end{center}
\textbf{Hardware:} IBM/Blue Gene Q, demonstrating excellent weak scalability with GenEO-based methods.

In this section, we have developed a comprehensive framework for understanding and implementing \textbf{two-level domain decomposition methods} based on \emph{spectral coarse spaces}. These methods address the core challenge of scalability in solving large, heterogeneous PDE systems. We started by identifying the limitations of classical one-level additive Schwarz methods, particularly in terms of weak scalability and sensitivity to coefficient heterogeneity. To overcome these, we introduced \textbf{coarse space corrections}—a global enhancement mechanism designed to efficiently transmit long-range information across the domain.

\tikzset{
  base/.style = {rectangle, draw=black, minimum height=1.2em, minimum width=3.5cm, rounded corners, font=\small, align=center},
  advanced/.style = {rectangle, draw=blue!70!black, thick, fill=blue!10, minimum height=1.2em, minimum width=3.8cm, rounded corners, font=\small\bfseries, align=center},
  coarse/.style = {rectangle, draw=green!50!black, thick, fill=green!10, minimum height=1.2em, minimum width=3.8cm, rounded corners, font=\small\bfseries, align=center},
  line/.style = {draw, -{Latex[scale=1.2]}, thick}
}

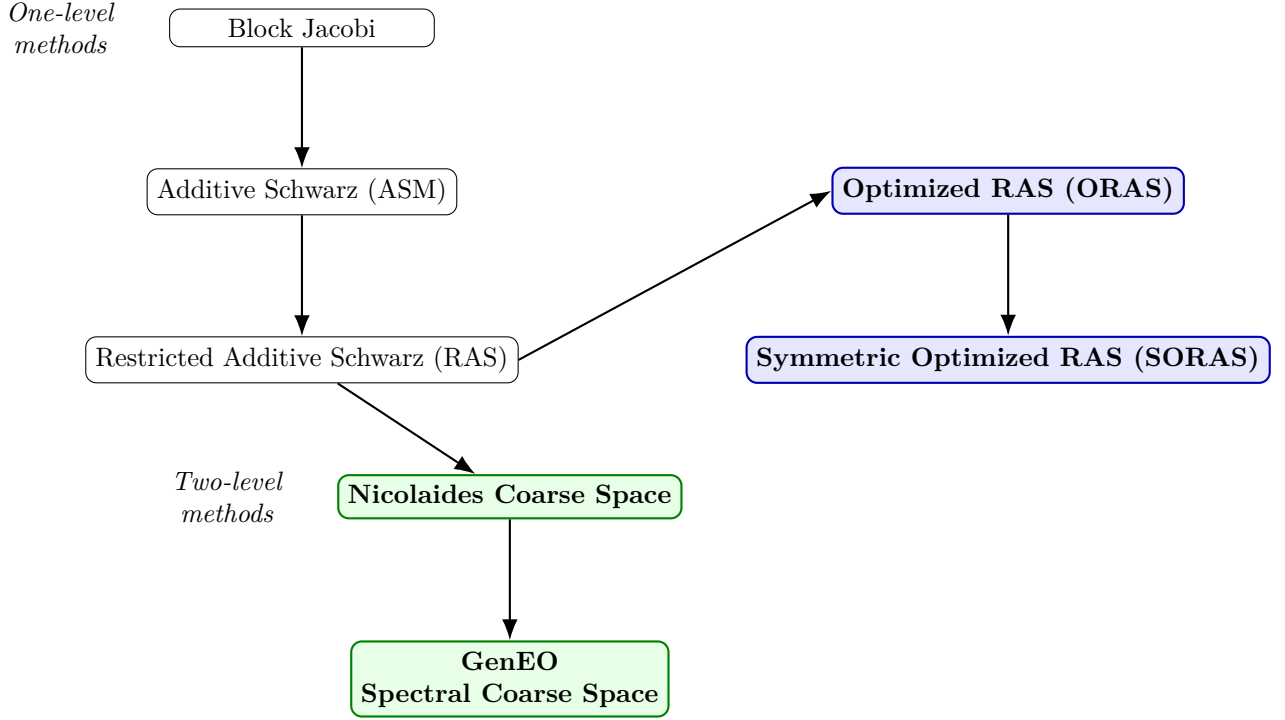
\begin{figure}[h!]
\centering
\begin{tikzpicture}[node distance=1.6cm and 1.5cm]

% One-level methods
\node[base] (jacobi) {Block Jacobi};
\node[base, below=of jacobi] (asm) {Additive Schwarz (ASM)};
\node[base, below=of asm] (ras) {Restricted Additive Schwarz (RAS)};
\node[advanced, right=3cm of ras] (soras) {Symmetric Optimized RAS (SORAS)};
\node[advanced, above=of soras] (oras) {Optimized RAS (ORAS)};

% Two-level enhancements
\node[coarse, below right=1.2cm and -2.4cm of ras] (nic) {Nicolaides Coarse Space};
\node[coarse, below=of nic] (geneo) {GenEO \\ Spectral Coarse Space};

% Arrows: vertical
\draw[line] (jacobi) -- (asm);
\draw[line] (asm) -- (ras);

% Arrows: horizontal to optimized
\draw[line] (ras.east) -- (oras.west);
\draw[line] (oras) -- (soras);

% Arrows: two-level improvements
\draw[line] (ras) -- (nic);
\draw[line] (nic) -- (geneo);

% Optional labels
\node[align=center, font=\small\itshape, left=0.6cm of jacobi] {One-level\\methods};
\node[align=center, font=\small\itshape, left=0.6cm of nic] {Two-level\\methods};

\end{tikzpicture}
\caption{Hierarchy of Schwarz-based methods and their two-level extensions.}
\end{figure}

\subsubsection*{Key Takeaways}
\begin{itemize}
  \item \textbf{Spectral coarse spaces are essential} to ensure robustness and scalability in the presence of multiscale and high-contrast features.
  \item \textbf{GenEO adapts automatically} to material heterogeneity, reducing iteration counts and enabling convergence even at massive scales.
  \item \textbf{Convergence guarantees} are available for a wide class of methods using the Fictitious Space Lemma.
\end{itemize}

%=================================================================
\section{Robust solvers for high-frequency problems}
\label{sec:helmholtz_equation}

Many wave-propagation applications are naturally posed in the frequency domain,
either because the excitation is (approximately) monochromatic or because one
solves for multiple right-hand sides at a fixed frequency.
Two representative examples are electromagnetics and acoustics.

\begin{backgroundinformation}{Two motivating applications}
\begin{itemize}
\item \textbf{Electromagnetic imaging.}
Recovering the electric permittivity $\varepsilon$ from boundary measurements
leads, under a time-harmonic assumption, to the (curl--curl) Maxwell system
\[
\nabla \times (\mu^{-1}\nabla \times \mathbf{E}) - \omega^2 \varepsilon\,\mathbf{E} = \mathbf{J},
\]
where $\mu$ is the magnetic permeability and $\mathbf{E}$ is the electric field.

\item \textbf{Seismic (acoustic) imaging.}
In isotropic acoustics, a time-harmonic ansatz reduces the wave equation to
a Helmholtz problem
\[
-\Delta u - \left(\frac{\omega^2}{c(x)^2}\right) u = f,
\]
where $c(x)$ denotes the local wave speed.
\end{itemize}
\end{backgroundinformation}

These settings motivate the development of \emph{accurate and scalable} solvers
for large, indefinite, frequency-domain systems. For the sake of simplicity, we will limit ourselves to the Helmholtz problem in the remaining part of the section.

%-----------------------------------------------------------------
Assume a monochromatic forcing, the solution $v$ of a wave equation will also be mono-chromatic:
\(
f(x,t)= f(x)\,e^{-i\omega t},\,
v(x,t)= u(x)\,e^{-i\omega t}.
\)
Substitution into the wave equation yields a stationary PDE for $u$:
\begin{equation}
-\Delta u - n(x)^2 \omega^2 u = f,
\qquad n(x)=\frac{1}{c(x)},
\qquad k(x) := n(x)\,\omega .
\label{eq:helmholtz_k}
\end{equation}
We refer to $k(x)$ as the \emph{wavenumber}. When $k$ is large, solutions exhibit
strong oscillations, a feature that drives both approximation and solver
difficulty.

\begin{important}{Oscillations, wavelength, and numerical difficulty}
The typical wavelength scales as $\lambda \sim 1/k$.
As $k$ increases, the solution oscillates more rapidly, the discrete operator
becomes increasingly indefinite, and iterative solvers become more sensitive to
preconditioning and coarse-space design.
\end{important}

\begin{figure}[h!]
\centering
\begin{minipage}[c]{0.75\textwidth}
\includegraphics[width=0.8\linewidth]{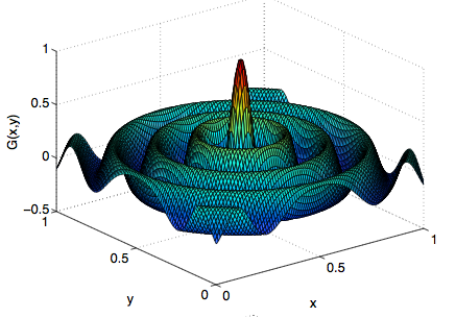}
\end{minipage}
\begin{minipage}[c]{0.2\textwidth}
\caption{Typical oscillatory Helmholtz solution at moderate-to-high frequency. Discretisation should capture oscillations}
\label{fig:helm_solution}
\end{minipage}
\end{figure}

%-----------------------------------------------------------------
A first requirement is \emph{resolution}: the mesh must capture oscillations.
A second, more subtle, requirement is \emph{dispersion control} (the pollution effect),
which forces additional refinement beyond points-per-wavelength heuristics.

\begin{important}{Pollution effect (FEM resolution condition)}
Keeping $h\omega$ constant is generally insufficient to control the error:
for a $p$-th order finite element method, quasi-optimal accuracy typically
requires the more stringent condition
\begin{equation}
h^p\,\omega^{p+1}\ \lesssim\ 1.
\label{eq:pollution_condition}
\end{equation}
In practice, this means that increasing $p$ can reduce the needed degrees of
freedom per wavelength, but does not eliminate the frequency-driven growth of
the global problem size.
\end{important}

\begin{svgraybox}
\textbf{Consequences.}
To \emph{resolve} oscillations one expects $h \sim \omega^{-1}$,
while to \emph{mitigate pollution} one needs $h \ll \omega^{-1}$,
often consistent with $h \sim \omega^{-1-1/p}$.
This trade-off (mesh resolution versus approximation order) directly impacts
memory, time-to-solution, and solver choice.
\end{svgraybox}

As far as the discretisation is concerned, two main strategies are widely used in large-scale Helmholtz computations.

\begin{backgroundinformation}{FD vs.\ FE in a nutshell}
\begin{itemize}
\item \textbf{Finite differences (FD).}
Structured grids, compact stencils, very sparse matrices, and the possibility
of \emph{optimized} schemes that reduce dispersion through coefficient tuning.

\item \textbf{Finite elements (FE).}
Variational formulation, excellent geometric flexibility, and local refinement
(\emph{$h$-adaptivity}) for heterogeneous coefficients and complex interfaces,
while keeping global sparsity.
\end{itemize}
\end{backgroundinformation}
A simple comparison on a velocity-gradient test case (fixed points per wavelength) from \cite{TournierJ022}
illustrates the practical trade-off.

\begin{table}[H]
\centering
\caption{\add{FD vs.\ FE on a velocity-gradient test case (4 ppwl). The parameter $\alpha$ controls the slope of the imposed linear velocity profile $c(x)=c_0+\alpha\, x_3$; larger $\alpha$ produces a wider spread of local wavelengths, as reported in the columns $\lambda_{\min}$ and $\lambda_{\max}$.}}
\label{tab:fd_fe_comp}
\begin{tabular}{|c|c|c|c|c|c|c|}
\hline
$\alpha$ & $\lambda_{\min}$ & $\lambda_{\max}$ & \#DoF FD & \#DoF FE & Error FD & Error FE \\
\hline
0.8 & 125 & 1200 & 13M & 28M & 0.0079 & 0.034 \\
2.0 & 125 & 3125 & 13M & 16M & 0.044  & 0.034 \\
\hline
\end{tabular}
\end{table}
\begin{figure}[h!]
\centering
\includegraphics[width=0.60\textwidth]{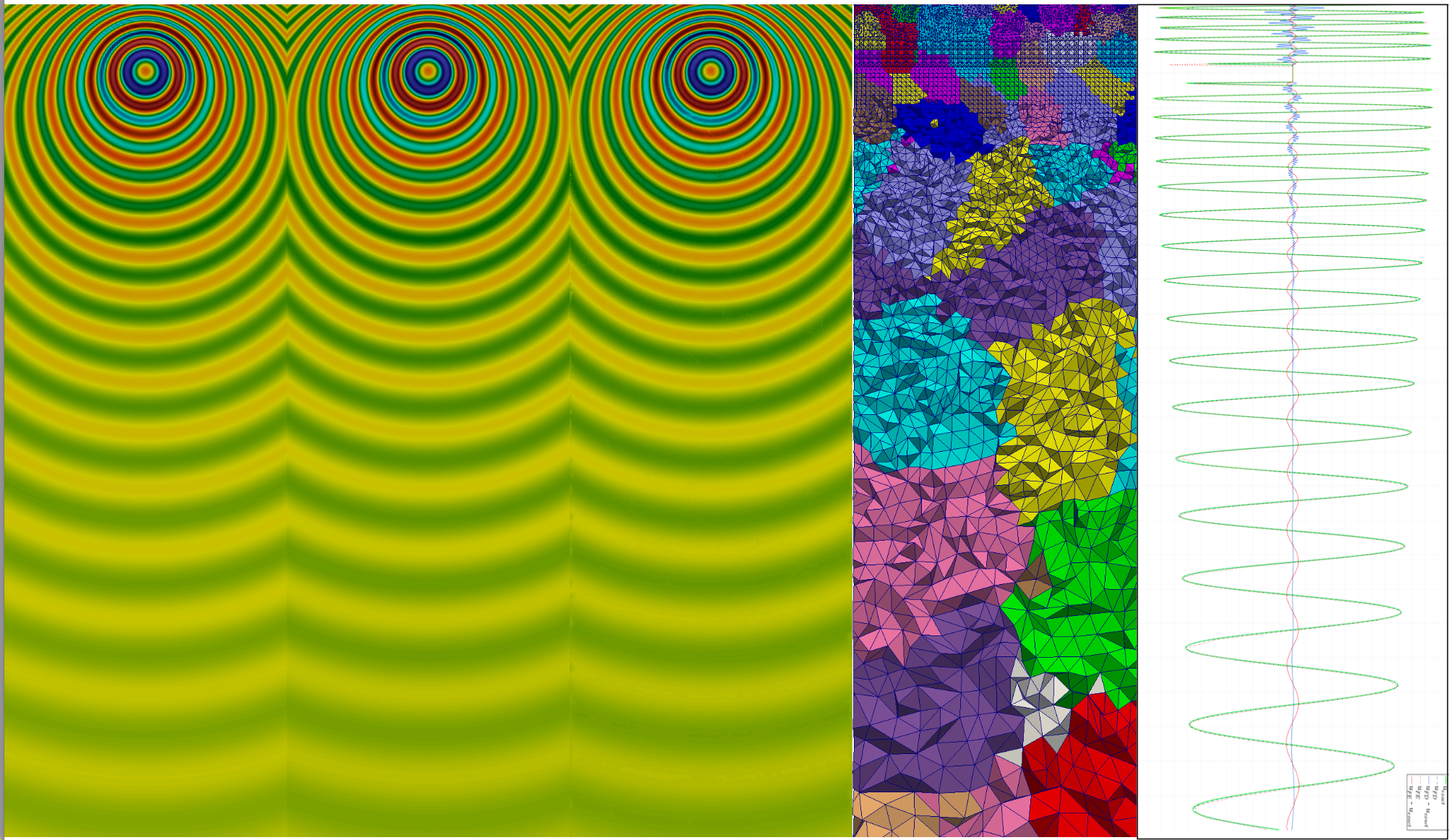}
\caption{Left: FD and FE solutions. Right: FE mesh and vertical profile comparison.}
\label{fig:fd_fe_solutions}
\end{figure}

\begin{important}{Interpretation}
For comparable accuracy, FD may achieve the target error with fewer degrees of
freedom in regimes where a structured discretization is admissible, leading to
lower memory usage and time per iteration. FE becomes indispensable when complex
geometry, interfaces, and local refinement dominate the modeling needs.
\end{important}

Both FD and FE lead to large sparse linear systems \(
A\mathbf{u}=\mathbf{b}\). Moreover, 

\begin{itemize}
\item $A$ is sparse but \textbf{indefinite} (and often highly non-normal after discretization).
\item The global size scales roughly as $n=\#\mathrm{DoF}\sim \omega^{(1+1/p)d}$
(heuristic, combining resolution and pollution effects).
\item Krylov methods (e.g.\ GMRES, BiCGStab) typically require strong preconditioning.
\item Direct solvers become prohibitively expensive in 3D at high frequency.
\end{itemize}

\begin{svgraybox}
\textbf{Transition to domain decomposition.}
At scale, the practical route is to combine Krylov methods with \emph{domain
decomposition} preconditioners. However, for Helmholtz, one-level methods alone
rarely suffice at high frequency: robust performance requires \emph{global}
mechanisms---in particular, carefully designed \emph{coarse space corrections}.
\end{svgraybox}

%-----------------------------------------------------------------
\subsection{A geometric choice: the grid-based coarse space}
\label{subsec:grid_cs}

A natural and historically first choice of coarse space for two-level methods
is obtained from a \emph{geometric} hierarchy.
Starting from the fine mesh $\mathcal{T}_h$, one introduces a coarse mesh
$\mathcal{T}_H$ with mesh diameter $H_{\mathrm{coarse}} \gg h$.
The corresponding coarse finite element space
$V_H \subset V_h$ is typically chosen as a standard low-order FE space
on $\mathcal{T}_H$.

Let $R_0^T$ denote the interpolation (prolongation) operator from the coarse
space $V_H$ to the fine space $V_h$.
Its transpose $R_0$ acts as restriction.
We define
\[
Z := R_0^T,
\qquad
E := Z^T A Z .
\]
Here:
\begin{itemize}
\item $Z$ embeds coarse functions into the fine space;
\item $E$ is the \emph{Galerkin coarse operator}, i.e.,
the stiffness matrix obtained by projecting $A$ onto $V_H$.
\end{itemize}

The associated coarse correction operator reads
\[
M_0^{-1} = Z E^{-1} Z^T,
\]
which is precisely the standard Galerkin projection on a nested finite
element hierarchy.
When combined additively with a one-level Schwarz preconditioner,
this yields the classical two-level domain decomposition method.

\medskip
\noindent
\textbf{Interpretation.}
The geometric coarse space represents low-frequency (long-wavelength)
components of the error.
In elliptic problems, these components correspond to smooth global modes.
In wave problems, they approximate oscillatory modes whose wavelength is
large relative to the subdomain size.
The coarse space thus provides a global communication channel that
one-level methods lack.

\begin{figure}[H]
\centering
\includegraphics[width=0.46\textwidth]{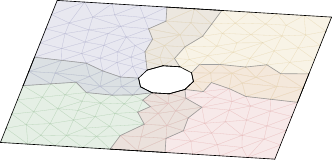}\hfill
\includegraphics[width=0.46\textwidth]{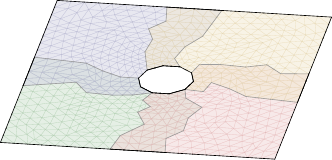}
\caption{Illustration of a geometric coarse space: coarse mesh hierarchy and a
typical two-level decomposition (schematic).}
\label{fig:grid_cs_illustration}
\end{figure}

%--------------------------------------------------------------
For indefinite wave problems, the simple application of the above doesn't work as the situation is quite delicate and requires additional assumptions. Consider the absorptive (shifted Laplacian) Helmholtz operator
\[
-\Delta - (k^2 + i\xi),
\]
where $k$ is the wavenumber and $\xi>0$ introduces absorption.
\add{The absorption term $i\xi$ restores a G\aa rding-type inequality (and in particular makes the sesquilinear form coercive on a shifted problem), which is sufficient to obtain spectral estimates analogous to the elliptic case.} Rigorous analysis shows that two-level domain decomposition
preconditioners with a geometric coarse space can achieve
$k$-independent GMRES convergence,
provided the coarse mesh and absorption are chosen appropriately.

\begin{important}{Robustness regime for absorptive Helmholtz}
Analysis due to Graham, Spence, and Vainikko~\cite{GrahamSpenceVainikko2017}
proves that a grid coarse space is effective if
\[
H_{\mathrm{coarse}} \sim k^{-\alpha} \quad (0<\alpha\le 1),
\qquad
|\xi| \sim k^2,
\]
in which case GMRES convergence is independent of the
wavenumber $k$. In words:
\begin{itemize}
\item the coarse mesh must resolve the wavelength at a $k$-dependent scale;
\item the absorption must be strong enough (comparable to $k^2$)
to control indefiniteness.
\end{itemize}
This robustness result extends to Maxwell's equations; see
\cite{Bonazzoli2019} for details.
\end{important}

The above conditions provide \emph{sufficient} conditions for robustness,
but they are often conservative compared to observed behavior.

\begin{important}{Practical viewpoint}
The shifted-Laplacian theory typically assumes (i) relatively generous overlap between subdomains; (ii) a coarse mesh that becomes increasingly fine as $k$ grows
($H_{\mathrm{coarse}}\to 0$); (iii) absorption levels of the same order as the indefinite term
($|\xi|\sim k^2$).

In many practical computations though (i) acceptable convergence is observed with much coarser grids; (ii) significantly milder absorption ($|\xi|\ll k^2$) suffices; (iii) the coarse-space dimension required in practice is smaller than predicted by worst-case theory.

This discrepancy suggests that the theoretical regime
may overestimate the coarse-space size needed for robustness.
It has motivated the development of alternative coarse spaces,
such as DtN- and GenEO-based constructions,
which aim to capture the relevant global propagative modes
more efficiently and in a problem-adaptive manner.
\end{important}

%--------------------------------------------------------------
\subsection{Spectral two-level preconditioners and coarse-space design choices}
\label{subsec:two_level_summary}

A generic two-level domain decomposition preconditioner for Helmholtz problems
can be written in the form
\[
M_2^{-1}
=
\underbrace{\sum_{j=1}^N R_j^T D_j A_j^{-1} R_j}_{\text{one-level (O)RAS component}}
\;+\;
\underbrace{Z\,(Z^*AZ)^{-1} Z^*}_{\text{coarse correction}} .
\]

The first term represents a one-level overlapping Schwarz preconditioner
(typically RAS/ORAS), which efficiently reduces high-frequency
(local) error components.
The second term is a Galerkin coarse correction,
which targets global and slowly converging modes.
The effectiveness of the two-level method hinges almost entirely
on the choice of the coarse basis $Z$.
It must:
\begin{itemize}
\item represent global error components poorly reduced by local solves;
\item remain of moderate dimension;
\item be robust with respect to heterogeneity and increasing frequency.
\end{itemize}

We summarize the widely used spectral coarse-space constructions in the addition to the grid based one introduced previously.

\textbf{(1) Dirichlet-to-Neumann coarse space (DtN CS) \cite{Conen2014}.}

DtN coarse spaces are constructed from local interface eigenmodes.
For each subdomain $\Omega_j$, one computes eigenpairs
\[
\mathrm{DtN}_{\Omega_j}(u^l_{\Gamma_j}) = \lambda^l u^l_{\Gamma_j},
\]
where $\Gamma_j$ denotes the subdomain interface
and $\mathrm{DtN}_{\Omega_j}$ is the local Dirichlet-to-Neumann map. Selected interface modes are then harmonically extended into the interior:
\[
Z \;\ni\; R_j^T D_j\,\mathcal{H}(u^l_{\Gamma_j}),
\]
where $\mathcal{H}$ denotes harmonic extension.

%-------------------------------------------------------------

\textbf{(2) GenEO-type spectral coarse spaces.}

GenEO (Generalized Eigenproblems in the Overlap)
constructs the coarse space by solving local generalized eigenproblems
that identify modes poorly reduced by the one-level preconditioner. A classical Laplace-based variant reads
\[
L_j \mathbf{u}^l_j
=
\lambda^l\, D_j L_j D_j\,\mathbf{u}^l_j,
\]
where $L_j$ is a local elliptic operator.
Eigenvectors corresponding to small eigenvalues are retained,
and the global coarse space is assembled as
\[
Z \;\ni\; R_j^T D_j \mathbf{u}^l_j.
\]

For Helmholtz problems, frequency-aware variants
(often denoted ``$\mathcal{H}_k$-GenEO'')
replace $L_j$ by Helmholtz-type operators:
\[
\widetilde{B}_j \mathbf{u}^l_j
=
\lambda^l\, D_j B_{j,k} D_j\,\mathbf{u}^l_j.
\]

\medskip
\noindent
GenEO spaces are algebraically adaptive:
they automatically detect problematic local modes
induced by heterogeneity, near-nullspaces,
or indefiniteness.
This makes them particularly attractive
for strongly heterogeneous and high-frequency regimes.

We briefly summarize the current theoretical understanding:

\begin{itemize}
\item \textbf{Grid CS.}
Strong theory exists for absorptive Helmholtz
($k^2+i\xi$) problems.
Robustness is proven when
$H_{\mathrm{coarse}} \sim k^{-\alpha}$ $(0<\alpha\le 1)$
and $|\xi|\sim k^2$.

\item \textbf{$\Delta$-GenEO.}
Provably robust for heterogeneous SPD problems
and effective at low-to-moderate frequencies.
See \cite{Bootland2022, DoleanFryLanger2026}.

\item \textbf{$\mathcal{H}_k$-GenEO.}
Strong numerical evidence supports robustness
in high-frequency indefinite regimes.
Theoretical analysis is more delicate,
with partial results available;
see \cite{Bootland2022b}.
\end{itemize}

\subsection{Numerical experiments and benchmarks}
\label{sec:numerical_benchmarks}

This section evaluates the robustness and scalability of the proposed
two-level domain decomposition preconditioners for high-frequency
wave propagation problems.
All simulations are performed using the open-source
\texttt{FreeFEM}~\cite{Hecht2012} framework with the \texttt{ffddm} plugin and the
\texttt{HPDDM} C++/MPI solver backend.

\medskip
\noindent
\textbf{Overview of benchmark cases.}
We consider two categories of test problems:

\begin{itemize}
\item \textbf{Electromagnetic cavity benchmarks:}
the 2D and 3D COBRA cavity test cases, representative of resonant Maxwell problems.

\item \textbf{Geophysical acoustic models:}
the Marmousi model (2D layered media),
the SEG/EAGE Overthrust model (3D heterogeneous),
and the large-scale crustal model GO\_3D\_OBS.
\end{itemize}

Across all benchmarks, we first compare three coarse-space strategies and then focus on the most performant one:

\begin{enumerate}
\item \textbf{Grid coarse space (Grid CS)} — geometric hierarchy;
\item \textbf{H-GenEO} — spectral coarse space from local eigenproblems;
\item \textbf{DtN} — interface-based Dirichlet-to-Neumann coarse space.
\end{enumerate}

The goals are:
(i) to assess robustness with respect to frequency and discretization density,
(ii) to evaluate weak and strong scalability,
(iii) to quantify time-to-solution in realistic inversion settings.

%-----------------------------------------------------------------
\subsubsection*{Coarse-space comparison on representative wave problems}

We first compare the three coarse-space strategies on
Marmousi, COBRA (2D and 3D), and Overthrust.
Each configuration is tested at low and high frequency,
with discretizations of 5 or 10 points per wavelength (ppwl).

\begin{center}
\scriptsize
\begin{tabular}{c|c|c|c|c|c|c|c|c}
\multicolumn{3}{c|}{} & \multicolumn{2}{c|}{Grid CS} & \multicolumn{2}{c|}{H-GenEO} & \multicolumn{2}{c}{DtN} \\
\hline
Problem & $d$ & freq & 5 ppwl & 10 ppwl & 5 ppwl & 10 ppwl & 5 ppwl & 10 ppwl \\
\hline
\multirow{2}{*}{Marmousi} & \multirow{2}{*}{2D} & low & \cmark & \cmark & \cmark & \cmark & \cmark\cmark & \cmark\cmark \\
\cline{3-9}
 & & high & \cmark\cmark & \cmark & \xmark & \cmark\cmark & \cmark & \cmark \\
\hline
\multirow{2}{*}{COBRA Cavity} & \multirow{2}{*}{2D} & low & \cmark & \cmark & \xmark & \xmark & \cmark\cmark & \cmark\cmark \\
\cline{3-9}
 & & high & \xmark & \cmark & \xmark & \xmark & \cmark\cmark & \cmark\cmark \\
\hline
\multirow{2}{*}{COBRA Cavity} & \multirow{2}{*}{3D} & low & \cmark & \cmark\cmark & \xmark & \xmark & \cmark\cmark & \cmark \\
\cline{3-9}
 & & high & \xmark & \cmark\cmark & \xmark & \xmark & \cmark\cmark & \cmark \\
\hline
\multirow{2}{*}{Overthrust} & \multirow{2}{*}{3D} & low & \cmark & \cmark & \xmark & \cmark & \cmark & \cmark \\
\cline{3-9}
 & & high & \cmark & \cmark & \xmark & \cmark & \cmark & \cmark \\
\end{tabular}
\end{center}

These results, based on the detailed comparison in \cite{Bootland2021},
highlight several trends:

\begin{itemize}
\item Grid CS is generally robust over all regimes and test cases.
\item H-GenEO is effective in heterogeneous SPD-like settings
but less robust in strongly indefinite cavity cases. Also performs better in a high resolution cases (with at least 10 ppwl - points per wavelength)
\item DtN consistently performs well in most of regimes
\end{itemize}

Given its robustness and implementation simplicity (for the spectral coarse spaces we need to pre-compute eigenvalues in a setup phase),
the grid coarse space was retained for the subsequent
large-scale geophysical scalability study in \cite{TournierJ022}.

%-----------------------------------------------------------------
\subsubsection*{Large-scale geophysical benchmarks}

To evaluate scalability in realistic inversion contexts,
we consider isotropic geophysical acoustic frequency-domain models for which the main characteristics are summarized below.

\begin{table}[H]
\centering
\caption{Summary of benchmark properties}
\begin{tabular}{|c|c|c|c|c|c|c|c|c|}
\hline
Model & $c_m$ (m/s) & $c_M$ (m/s) & $f$ (Hz) & $\lambda_{min}$ & $\lambda_{max}$ & $G_m$ & $G_M$ & $N_\lambda$ \\
\hline
Homogeneous & 1500 & 1500 & 7.5 & 200 & 200 & 4 & 4 & 100 \\
Linear      & 1500 & 8500 & 7.5 & 200 & 1133 & 4 & 22.7 & 50 \\
Overthrust  & 2179 & 6000 & 10  & 218 & 600  & 4.4 & 12 & 50 \\
GO\_3D\_OBS & 1500 & 8639 & 3.75 & 400 & 2303 & 4 & 23 & 255 \\
\hline
\end{tabular}
\end{table}
The \texttt{GO\_3D\_OBS} crustal model contains realistic
geological interfaces (bathymetry, sediments).
An adapted tetrahedral mesh is used to match wavelength distribution.
At 3.75~Hz, the adapted mesh contains approximately
132 million elements,
whereas a uniform grid would require 402.4 million cells,
yielding a coarsening factor of about 3.

\begin{figure}[H]
\centering
\includegraphics[width=0.5\linewidth]{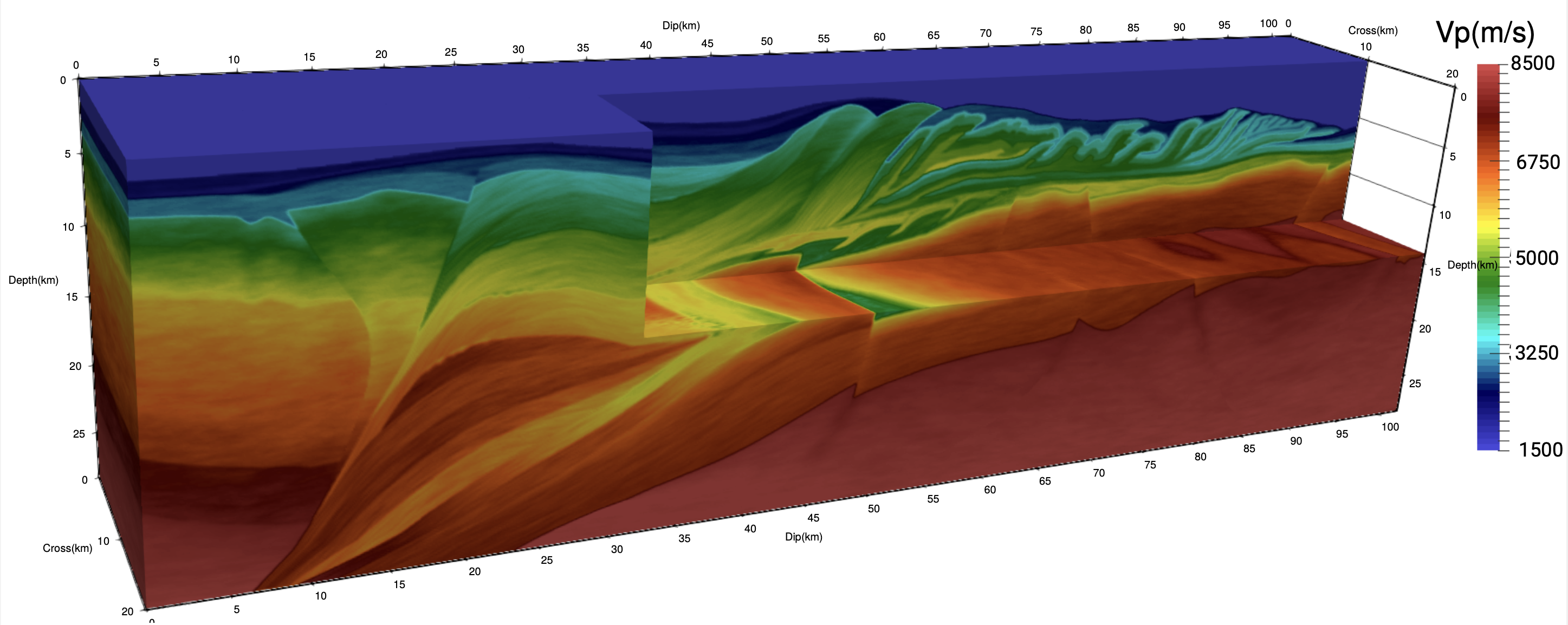}
\includegraphics[width=0.45\linewidth]{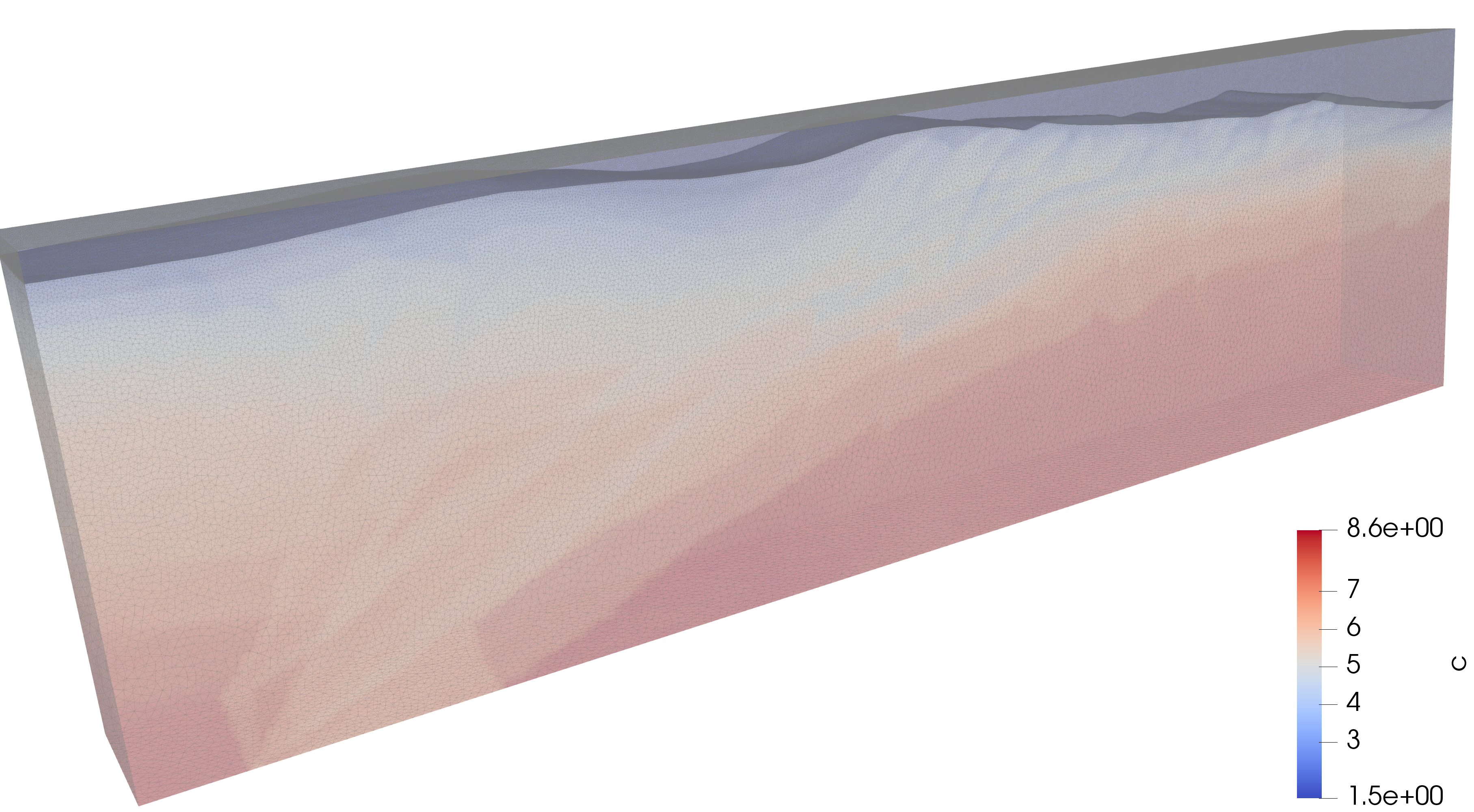}
\caption{Velocity field of the GO\_3D\_OBS crustal model (left),
adapted mesh conforming to bathymetry (right).}
\end{figure}

%-----------------------------------------------------------------
\subsubsection*{Weak scalability: frequency sweep}

We next perform weak scaling experiments on \texttt{GO\_3D\_OBS}
using finite differences.
(At very high frequencies and extreme problem sizes,
finite differences are preferred over finite elements
for memory and assembly efficiency.)

\begin{table}[H]
\centering
\caption{Weak scalability: iteration count and timings with increasing frequency}
\begin{tabular}{|c|c|c|c|c|c|c|c|c|}
\hline
f (Hz) & \#dof ($10^6$) & \#cores & \#it & ovl & $T_f$(s) & $T_s$(s) & $T_{tot}$(s) & $E_w$ \\
\hline
2.5 & 21.4 & 60 & 26 & 3 & 52.5 & 24.8 & 77.3 & 1 \\
3.75 & 67.5 & 360 & 40 & 3 & 21.9 & 18.6 & 40.5 & 1.003 \\
5.0 & 153.5 & 875 & 61 & 3 & 20.9 & 26.0 & 46.9 & 0.811 \\
7.5 & 500.1 & 2450 & 98 & 4 & 26.7 & 60.8 & 87.5 & 0.506 \\
10.0 & 1160.6 & 9856 & 142 & 5 & 15.1 & 70.8 & 85.9 & 0.297 \\
\hline
\end{tabular}
\end{table}

\begin{figure}[H]
\centering
\begin{minipage}[c]{0.7\textwidth}
\includegraphics[width=0.8\linewidth]{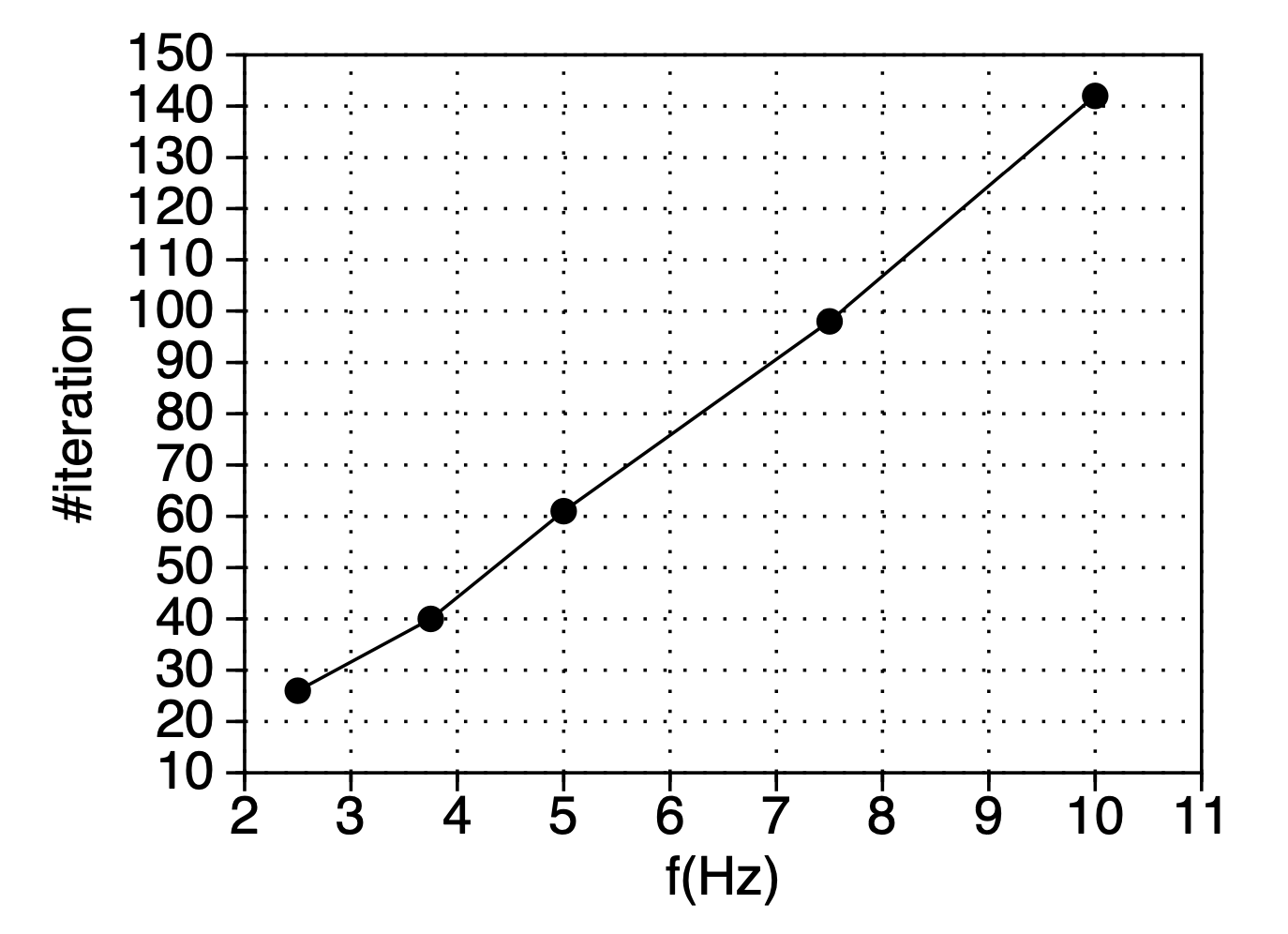}
\end{minipage}
\begin{minipage}[c]{0.28\textwidth}
\caption{Iteration count versus frequency (weak scaling).}
\end{minipage}
\end{figure}

The results show that iteration counts increase with frequency, as expected for indefinite wave problems, while total runtimes remain relatively stable at moderate frequencies. At the largest scales and highest frequencies, a controlled degradation of weak-scaling efficiency is observed, reflecting the combined impact of stronger indefiniteness and increased communication overhead, yet without compromising overall solver robustness.
%-----------------------------------------------------------------
\subsubsection*{Strong scalability: FD vs FE at 3.75 Hz}

At 3.75~Hz, we compare strong scalability
for finite difference (67.5M dofs)
and finite element (597.7M dofs) discretizations.

\begin{figure}[H]
\centering
\includegraphics[width=0.48\linewidth]{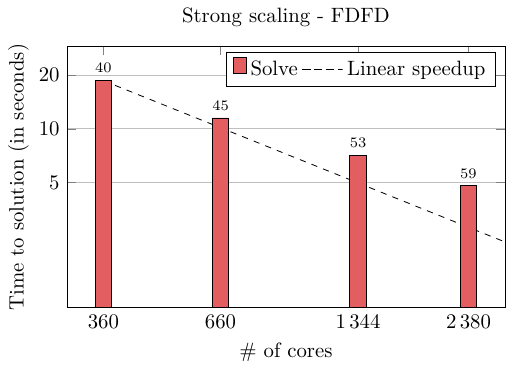}
\includegraphics[width=0.48\linewidth]{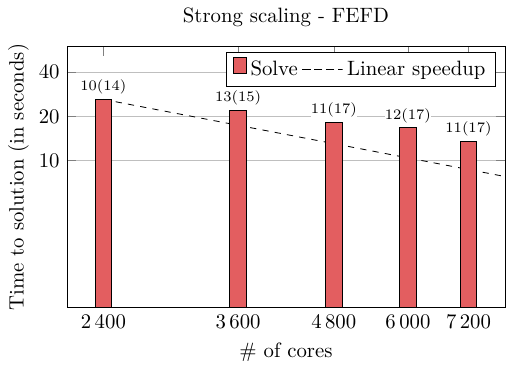}
\caption{Strong scaling at 3.75 Hz:
FD case (left) and FE case (right).}
\end{figure}

In both cases, the solver exhibits nearly linear speedup,
confirming good parallel efficiency
even at very large problem sizes.

%-----------------------------------------------------------------
\subsubsection*{Time-to-solution comparison: time vs frequency domain}

Finally, we compare:
\begin{itemize}
\item a classical time-domain FDTD solver,
\item a direct frequency-domain solver (MUMPS),
\item and an ORAS-preconditioned Krylov solver,
\end{itemize}
for 130 right-hand sides at 3.75~Hz.

\begin{table}[H]
\centering
\caption{Wall-clock hours for solving 130 RHSs at 3.75~Hz}
\begin{tabular}{|c|c|c|c|c|c|c|c|c|c|c|}
\hline
\multicolumn{3}{|c|}{FDTD} & \multicolumn{5}{|c|}{MUMPS} & \multicolumn{3}{|c|}{ORAS} \\
\hline
\#cores & $T_s^{(130)}$ (s) & $T_{hc}$ (hr) & \#c & $T_f$ & $T_s^1$ & $T_s^{130}$ & $T_{hc}$ & \#c & $T_s^{130}$ & $T_{hc}$ \\
\hline
12480 & 264 & 915 & 1920 & 1000 & 3.9 & 28.3 & 548 & 1344 & 256.8 & 96 \\
\hline
\end{tabular}
\end{table}

Despite the larger memory footprint,
the frequency-domain ORAS solver
achieves significantly lower hours-to-solution
than both FDTD and direct factorization.

\medskip
\noindent
\textbf{Overall message.}
These results demonstrate that:
\begin{itemize}
\item two-level preconditioners
remain effective in realistic high-frequency regimes;
\item scalability is maintained up to thousands of cores;
\item frequency-domain solvers equipped with robust Schwarz
preconditioners constitute competitive forward engines
for large-scale full-waveform inversion (FWI).
\end{itemize}

%=================================================================
The numerical solution of the Helmholtz equation at high frequencies remains one of the most challenging tasks in scientific computing due to the rapid growth in complexity with frequency. However, recent developments have enabled the design of efficient, accurate, and scalable forward solvers, particularly in the context of 3D frequency-domain full waveform inversion (FWI). Among these, finite-difference methods (FDFD) have demonstrated superior computational efficiency compared to finite elements (FEFD), provided the mesh structure and geometry allow it. The current framework supports real-world applications such as crustal-scale seismic imaging, where large-scale heterogeneous models must be handled. Two major research directions emerge: (i) further deployment of these solvers as engines for visco-acoustic FWI in 3D, and (ii) extensions to more complex visco-elastic formulations. From a theoretical standpoint, the behavior of some methods remains insufficiently understood, highlighting the need for deeper mathematical insight, possibly through microlocal analysis or operator theory. On the computational side, strategies to reduce the number of forward solves—such as surrogate modeling or learning-based approaches—could significantly improve the viability of inversion workflows.

\section{Using domain decomposition libraries}
\label{sec:implementation}

In this section we present a practical overview of the typical steps needed by a program to use a domain decomposition method in a distributed environment. In a distributed parallel setting, data is distributed among the computing units of the parallel machine, each unit owning a piece of the data. In the domain decomposition framework, a natural way is to assign one subdomain per computing core.

We provide below a condensed description of the interface and key parameters of two domain decomposition libraries: \textit{ffddm} and \textit{HPDDM}.
\begin{itemize}
\item \textit{ffddm} stands for \textit{FreeFEM Domain Decomposition Methods} and is a set of high-level FreeFEM scripts implementing Schwarz domain decomposition techniques in parallel. It aims to simplify the use of parallel solvers for FreeFEM users. Being entirely written in the FreeFEM language also makes it a very good tool for learning and prototyping domain decomposition methods without compromising efficiency.
FreeFEM is an open-source partial differential equation solver based on the finite element method, designed for numerical simulations in engineering and scientific research. It provides a high-level programming language that allows users to easily define variational formulations, mesh geometries, and solve complex multiphysics problems. In what follows, we will focus on a single toy problem, but there are many available examples in the FreeFEM source tree at \url{https://github.com/FreeFem/FreeFem-sources/tree/develop/examples/ffddm}.
\item \textit{HPDDM} (\textit{High-Performance Domain Decomposition Methods}) is a C++ library for solving large-scale sparse linear systems using domain decomposition and iterative solvers, optimized for parallel computing on high-performance architectures. It is available within PETSc (Portable, Extensible Toolkit for Scientific Computation)~\cite{petsc-user-ref}, which provides a flexible framework for managing matrices, vectors, and solvers, allowing \textit{HPDDM} to efficiently handle scalable linear algebra operations within PETSc’s ecosystem. The PETSc interface to \textit{HPDDM} is called \textit{PCHPDDM}~\cite{Jolivet2021KSPHPDDM}. There are also many available examples in the FreeFEM source tree at \url{https://github.com/FreeFem/FreeFem-sources/tree/develop/examples/hpddm}.
\end{itemize}

Once again, in the context of distributed parallel computing, these libraries distribute data and computations across multiple processors and nodes. As memory is not shared between processes, communication is essential to exchange data, assemble global operators, and synchronize iterative solvers. To ensure this communication is performed efficiently and in a portable way across different computing architectures, they rely on the Message Passing Interface (MPI), a standardized protocol that enables multiple processes to communicate and coordinate their work across distributed systems, making it possible to tackle very large simulations that would otherwise be infeasible on a single processor. In order to illustrate the different steps, we go through a typical FreeFEM program solving an elliptic PDE problem using a GenEO domain decomposition method. Below is the full script solving a Poisson problem on the unit square $\Omega$:

\begin{equation*}
-\Delta u = 1 \quad \text{in } \Omega, \qquad u = 0 \quad \text{on } \partial\Omega.
\label{eq:poisson_ff}
\end{equation*}

\begin{ddmlisting}
macro dimension 2// End of macro
include "ffddm.idp"
mesh Th = square(100,100); // global mesh
// Step 1: Decompose the mesh
ffddmbuildDmesh(M, Th, mpiCommWorld)
// Step 2: Define the distributed finite element space
ffddmbuildDfespace(FE, M, real, P2)
// Step 3: Define the problem to solve
macro grad(u) [dx(u), dy(u)] // End of macro
macro Varf(varfName, meshName, null)
    varf varfName(u,v) = int2d(meshName)(grad(u)'*grad(v))
                       + int2d(meshName)(1*v)
                       + on(1,2,3,4, u = 0); // End of macro
ffddmsetupOperator(PB, FE, Varf)
FEVhi ui, bi; // ui and bi are distributed FE functions
ffddmbuildrhs(PB, Varf, bi[])
// Step 4: Define the one level DD preconditioner
ffddmsetupPrecond(PB, Varf)
// Step 5: Define the two-level GenEO Coarse Space
ffddmgeneosetup(PB, Varf)
// Step 6: Solve the linear system with GMRES
FEVhi x0i = 0; // initial guess
ui[] = PBfGMRES(x0i[], bi[], 1e-6, 200, "right");
ffddmplot(FE, ui, "u") // plot the solution
PBwritesummary
\end{ddmlisting}

This parallel FreeFEM program \sh{Laplace.edp} can be run on a laptop, or a high-performance cluster. For example, the program can be launched using 8 cores with the command

\begin{ddmshell}[numbers=none]
ff-mpirun -np 8 Laplace.edp -wg
\end{ddmshell}

We describe the different steps below, emphasizing the range of options accessible to the user via command-line parameters for each step. For each option, we highlight the \textit{ffddm} command-line parameter as well as the corresponding \textit{PCHPDDM} parameter when applicable. The latter set of parameters is not readily usable with the described script, however, they could be used with examples using PETSc instead of \textit{ffddm}, or any other finite element software using PETSc as the linear algebra backend, e.g., code\_aster \url{https://www.code-aster.org/} or Gridap.jl~\cite{BadiaVerdugo2020Gridap}.

\add{Several algorithmic choices (mesh partitioner, overlap width, preconditioner type, coarse correction, GMRES tolerance, \ldots) can be set at runtime through \emph{command-line parameters} passed to the FreeFEM executable, without having to modify the script. A typical invocation combining several of them is}

\begin{ddmshell}[numbers=none]
ff-mpirun -np 8 Laplace.edp -wg            \
    -ffddm_partitioner 1                   \
    -ffddm_overlap 2                       \
    -ffddm_schwarz_method ras              \
    -ffddm_schwarz_coarse_correction ADEF1 \
    -ffddm_geneo_threshold 0.5
\end{ddmshell}

\add{which runs the script on $8$ MPI processes using METIS partitioning, two layers of overlap, RAS as the one-level preconditioner, an ADEF1 coarse correction, and a GenEO threshold $\tau=0.5$. The individual parameters used here are explained in detail in Steps~1--6 below.}

\subsection{Step 1: Decompose the mesh}

The first step is to define an overlapping decomposition of the global mesh \ff{Th} of the domain $\Omega$. This is done with the call to \ff{ffddmbuildDmesh}:
 
\begin{ddmlisting}[firstnumber=5]
ffddmbuildDmesh(M, Th, mpiCommWorld)
\end{ddmlisting}

The first argument of \ff{ffddmbuildDmesh} is an identifier, or \textit{prefix}, chosen by the user as a name for this mesh decomposition -- multiple decompositions can coexist in the same script. Here we chose \ff{M} as the identifier of the mesh decomposition. \ff{ffddmbuildDmesh} will create variables whose names are prefixed by this identifier ; here for example, the local mesh of the subdomain will be called \ff{MThi}.

The third argument is the \textit{MPI communicator} defining the group of processes over which the mesh is decomposed. Typically, all available MPI processes are used, in which case the communicator is \ff{mpiCommWorld}, as in this example.

The way the initial mesh \ff{Th} is partitioned depends on the value of the command-line parameter \cmd{-ffddm_partitioner}. Options are: \cmdv{0}: user-defined partition in the script, \cmdv{1}: use the automatic graph partitioner METIS (default), \cmdv{2}: use the automatic graph partitioner SCOTCH.

The command-line parameter \cmd{-ffddm_overlap} controls the amount of overlap.\\ \cmd{-ffddm_overlap 3} extends each subdomain by 3 layers of mesh elements in a symmetric way, resulting in an overlapping region with a width of 6 mesh elements between subdomains.

\begin{figure}[h!]
    \centering
    \includegraphics[width=0.4\textwidth]{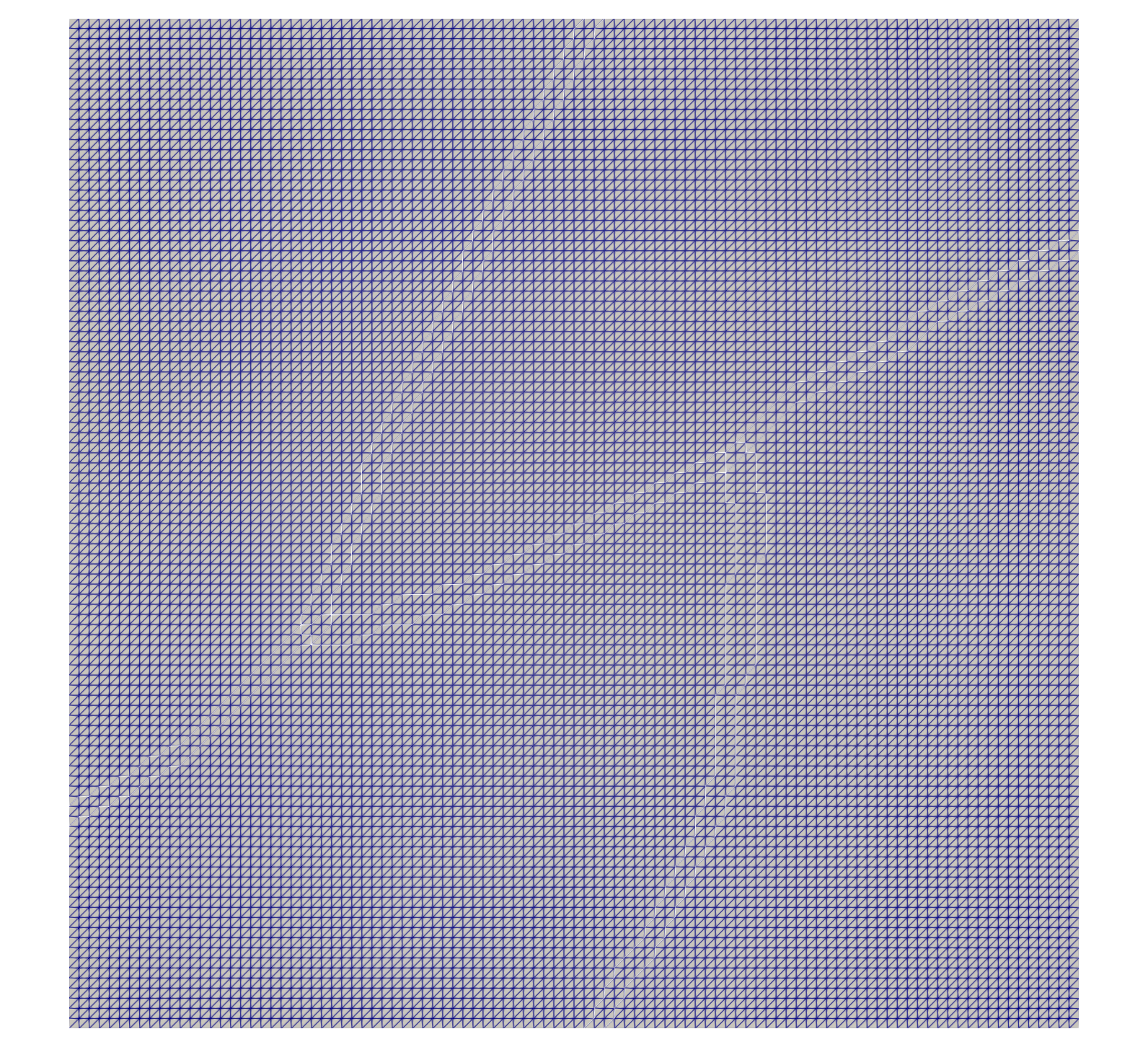}
    \includegraphics[width=0.4\textwidth]{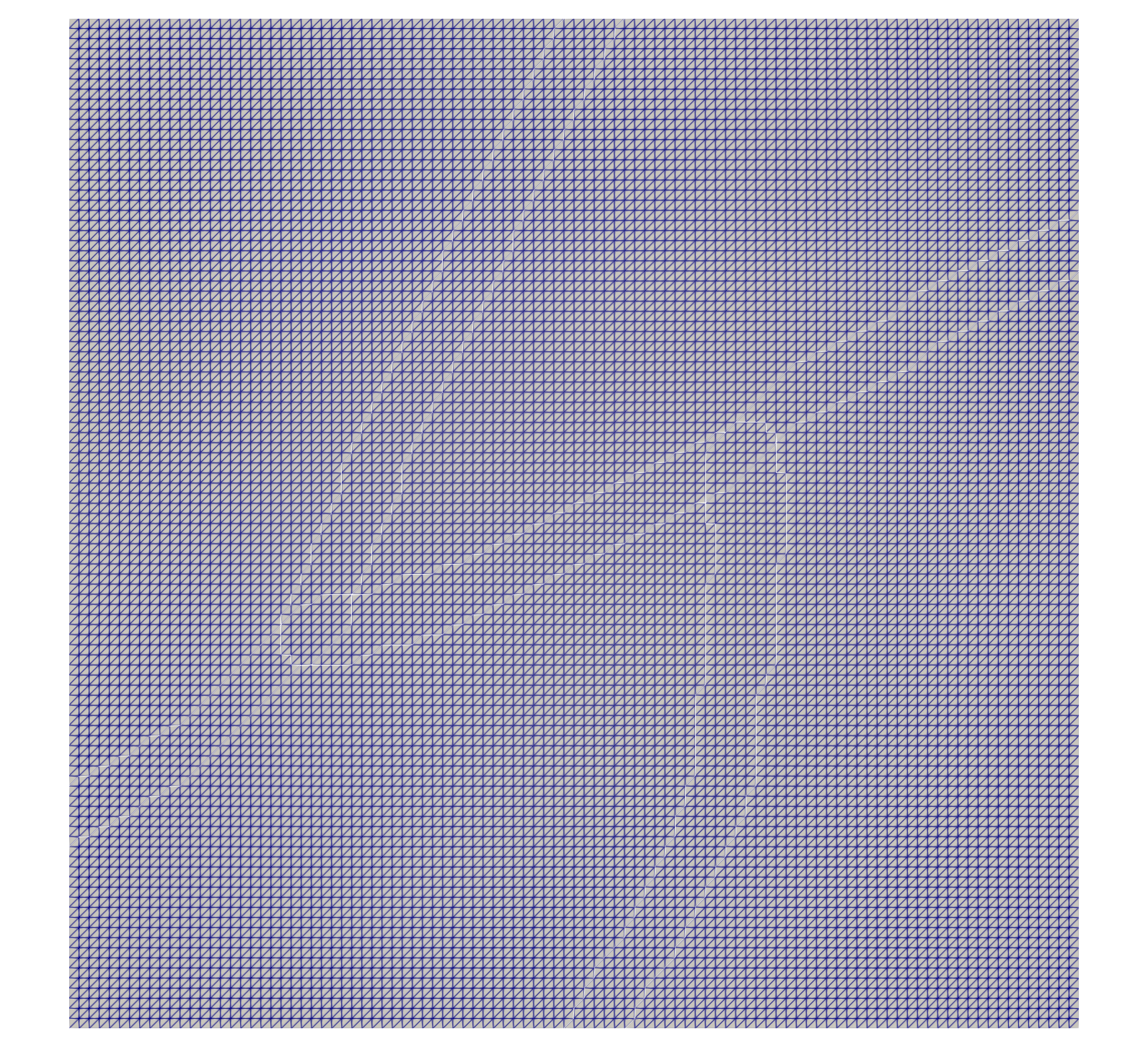}
    \caption{The overlapping mesh decomposition with \cmd{-ffddm_overlap 1} (left) and \cmd{-ffddm_overlap 3} (right).}
\end{figure}

\subsection{Step 2: Define the distributed finite element space}

\ff{ffddmbuildDfespace} builds the distributed finite element space on top of the distributed mesh decomposition. The first parameter is the prefix chosen by the user to identify this distributed finite element space. Here for example, we build the distributed space \ff{FE} on top of the mesh decomposition \ff{M} defined in the previous step.

\ff{ffddmbuildDfespace} constructs in particular the partition of unity matrices $\{D_i\}_{i=1}^N$. Here we build the distributed finite element space for \ff{P2} Lagrange finite elements with \ff{real} coefficients (use \ff{complex} for complex problems):

\begin{ddmlisting}[firstnumber=7]
ffddmbuildDfespace(FE, M, real, P2)
\end{ddmlisting}

\subsection{Step 3: Define the problem to solve}

\ff{ffddmsetupOperator} builds the distributed operator associated to the variational problem that we want to solve:

\begin{ddmlisting}[firstnumber=9]
macro grad(u) [dx(u), dy(u)] // End of macro
macro Varf(varfName, meshName, null)
    varf varfName(u,v) = int2d(meshName)(grad(u)'*grad(v))
                       + int2d(meshName)(1*v)
                       + on(1,2,3,4, u = 0); // End of macro
ffddmsetupOperator(PB, FE, Varf)
\end{ddmlisting}

The first parameter is the prefix chosen by the user to identify this distributed operator. Here for example, we build the distributed operator \ff{PB} on top of the distributed finite element space \ff{FE} defined in the previous step.

The definition of the problem is given by the \ff{Varf} input macro, which in this case corresponds to the variational form of the Poisson problem.

In practice, the global matrix $A$ is constructed in a parallel and distributed way, each process (or core) assembling only the restriction matrix $A_i := R_i A R_i^T$ of its subdomain. The local matrices $A_i$ will then be used to compute the action of the global operator $A$ on distributed vectors in the iterative solver.

In a similar manner, we build the right-hand side of our distributed problem \ff{PB} corresponding to the discretization of the linear form of the \ff{Varf}:

\begin{ddmlisting}[firstnumber=16]
ffddmbuildrhs(PB, Varf, bi[])
\end{ddmlisting}

Each process builds only the restriction \ff{bi} of the right-hand side to its subdomain.

\subsection{Step 4: Define the one level preconditioner}

We enter the setup phase of the domain decomposition preconditioner:

\begin{ddmlisting}[firstnumber=18]
ffddmsetupPrecond(PB, Varf)
\end{ddmlisting}
builds the one-level preconditioner for our distributed operator \ff{PB}. The local matrices $A_i$ are factorized by the direct solver \textit{MUMPS} in each subdomain, each core computing its factorization independently in parallel. These will be used to compute the action of the local inverse $A_i^{-1}$ at each application of the preconditioner.

The choice of the one-level preconditioner is given by the command-line parameter \cmd{-ffddm_schwarz_method}. Options are: \cmdv{asm} (Additive Schwarz), \cmdv{ras} (Restricted Additive Schwarz), \cmdv{oras} (Optimized Restricted Additive Schwarz), \cmdv{soras} (Symmetric Optimized Restricted Additive Schwarz) or \cmdv{none} (no preconditioner). \add{The default is \cmdv{ras} (Restricted Additive Schwarz).} For Optimized methods, local matrices $B_i$ correspond to the discretization of the provided variational form \ff{Varf} on subdomains $\Omega_i$. It may differ from the variational form of the original problem, for example by including Robin-type boundary conditions at the subdomain interface.

\textit{$\rightarrow$ with PCHPDDM:} First, we need to tell PETSc to use \textit{HPDDM} as a preconditioner with \cmdp{-pc_type hpddm}. Passing \cmdp{-pc_hpddm_has_neumann} lets \textit{HPDDM} know that FreeFEM is supplying the local Neumann matrices, while \cmdp{-pc_hpddm_define_subdomains} sets the domain decomposition to be the same as the one used for the Neumann matrices. The default preconditioner is RAS, but we can switch to ASM with \\ \cmdp{-pc_hpddm_levels_1_pc_asm_type basic}. 

We can tell \textit{HPDDM} to use an exact (direct) factorization for the local matrices $A_i$ in the one-level preconditioner with \cmdp{-pc_hpddm_levels_1_sub_pc_type lu}.

\subsection{Step 5: Define the two-level GenEO Coarse Space}

The GenEO coarse space is built by

\begin{ddmlisting}[firstnumber=20]
ffddmgeneosetup(PB, Varf)
\end{ddmlisting}
which solves the GenEO eigenvalue problem~\eqref{eq:geneo_local_eig} concurrently in each subdomain. The Neumann matrices $A^{\mathrm{Neu}}_i$ are built by discretizing \ff{Varf} on subdomain $\Omega_i$.

Eigenvectors with eigenvalues smaller than the user-prescribed threshold $\tau>0$ are automatically selected to enter the coarse space $Z = R_0^T$. $\tau$ can be prescribed on the command-line, for example with \cmd{-ffddm_geneo_threshold 0.5}.

\textit{$\rightarrow$ with PCHPDDM:} \cmdp{-pc_hpddm_levels_1_eps_threshold_absolute 0.5}.

Finally, the coarse space operator $A_0 := R_0 A R_0^T$ is assembled and factorized.

We have seen that the coarse correction can be combined with the one-level preconditioner in an additive way (Equation~\eqref{eq:two_level_asm}). There are also several multiplicative variants, which may further reduce the number of iterations, but can come  with extra cost and are more difficult to analyze. The choice of the coarse correction formula can be specified with the command-line parameter \cmd{-ffddm_schwarz_coarse_correction} which can take the following values:

\begin{align*}
\nonumber
    &&\cmdv{AD}:&&\textit{Additive}, \quad &M^{-1} = \phantom{(I - Q A) }M^{-1}_1\phantom{ (I - A Q)} + Q\\
    &&\cmdv{BNN}:&&\textit{Balancing Neumann-Neumann}, \quad &M^{-1} = (I - Q A) M^{-1}_1 (I - A Q) + Q\\
    &&\cmdv{ADEF1}:&&\textit{Adapted Deflation Variant 1}, \quad &M^{-1} = \phantom{(I - Q A) }M^{-1}_1 (I - A Q) + Q\\
    &&\cmdv{ADEF2}:&&\textit{Adapted Deflation Variant 2}, \quad &M^{-1} = (I - Q A) M^{-1}_1\phantom{ (I - A Q)} + Q\\
    &&\cmdv{RBNN1}:&&\textit{Reduced Balancing Variant 1}, \quad &M^{-1} = (I - Q A) M^{-1}_1 (I - A Q)\\
    &&\cmdv{RBNN2}:&&\textit{Reduced Balancing Variant 2}, \quad &M^{-1} = (I - Q A) M^{-1}_1\phantom{ (I - A Q)}\\
    &&\cmdv{none}:&&\textit{no coarse correction}, \quad &M^{-1} = \phantom{(I - Q A) }M^{-1}_1\phantom{ (I - A Q)}\\
\end{align*}
where $M^{-1}$ is the one-level preconditioner and $Q := R_0^T A_0^{-1} R_0$ is the coarse space correction operator. \add{The default is \cmdv{ADEF1} (Adapted Deflation Variant~1), which in our experience offers a good compromise between robustness and cost.}

\textit{$\rightarrow$ with PCHPDDM:} \cmdp{-pc_hpddm_coarse_correction} allows to specify the coarse correction formula. Options are: \cmdpv{additive}, \cmdpv{balanced} (BNN), \cmdpv{deflated} (ADEF1) or \cmdpv{none}.

\subsection{Step 6: Solve the linear system}

We can now solve the linear system using GMRES, using the \ff{fGMRES} function for our operator \ff{PB}:

\begin{ddmlisting}[firstnumber=22]
FEVhi x0i = 0; // initial guess
ui[] = PBfGMRES(x0i[], bi[], 1e-6, 200, "right");
\end{ddmlisting}
The initial guess \ff{x0i} and the right-hand side \ff{bi} are given as distributed vectors. The solution of the linear system is given as the distributed output vector \ff{ui}. Here the GMRES tolerance is \ff{1e-6}, the maximum number of iterations is \ff{200} and GMRES is preconditioned to the \ff{"right"} (use \ff{"left"} for left-preconditioning). 

\add{Right preconditioning is the default, and is recommended in practice because it provides a reliable residual-based stopping criterion (see also the discussion on page~\pageref{eq:unified_schwarz_prec}).}

\textit{$\rightarrow$ with PCHPDDM:} PETSc command-line parameters related to the Krylov iterative solver are prefixed by \cmdp{-ksp_}. The choice of the Krylov method can be set with \cmdp{-ksp_type}. For example, \cmdp{-ksp_type cg} uses the conjugate gradient. \cmdpv{gmres} is the default Krylov solver. Preconditioning to the right (resp. left) is done with \cmdp{-ksp_pc_side right} (resp. \cmdpv{left}). 

\add{Right preconditioning is also the PCHPDDM default.} The relative convergence tolerance of the iterative solver can be set with \cmdp{-ksp_rtol}, while the maximum number of iterations is prescribed with \cmdp{-ksp_max_it}.\\

Finally, we can plot the global solution with \ff{ffddmplot} and make a call to \ff{PBwritesummary} to print an overview of the timings (in seconds) of the different steps of the method:

\begin{ddmlisting}[firstnumber=24]
ffddmplot(FE, ui, "u") // plot the solution
PBwritesummary
\end{ddmlisting}

\begin{figure}[h!]
    \centering
    \includegraphics[width=0.6\textwidth]{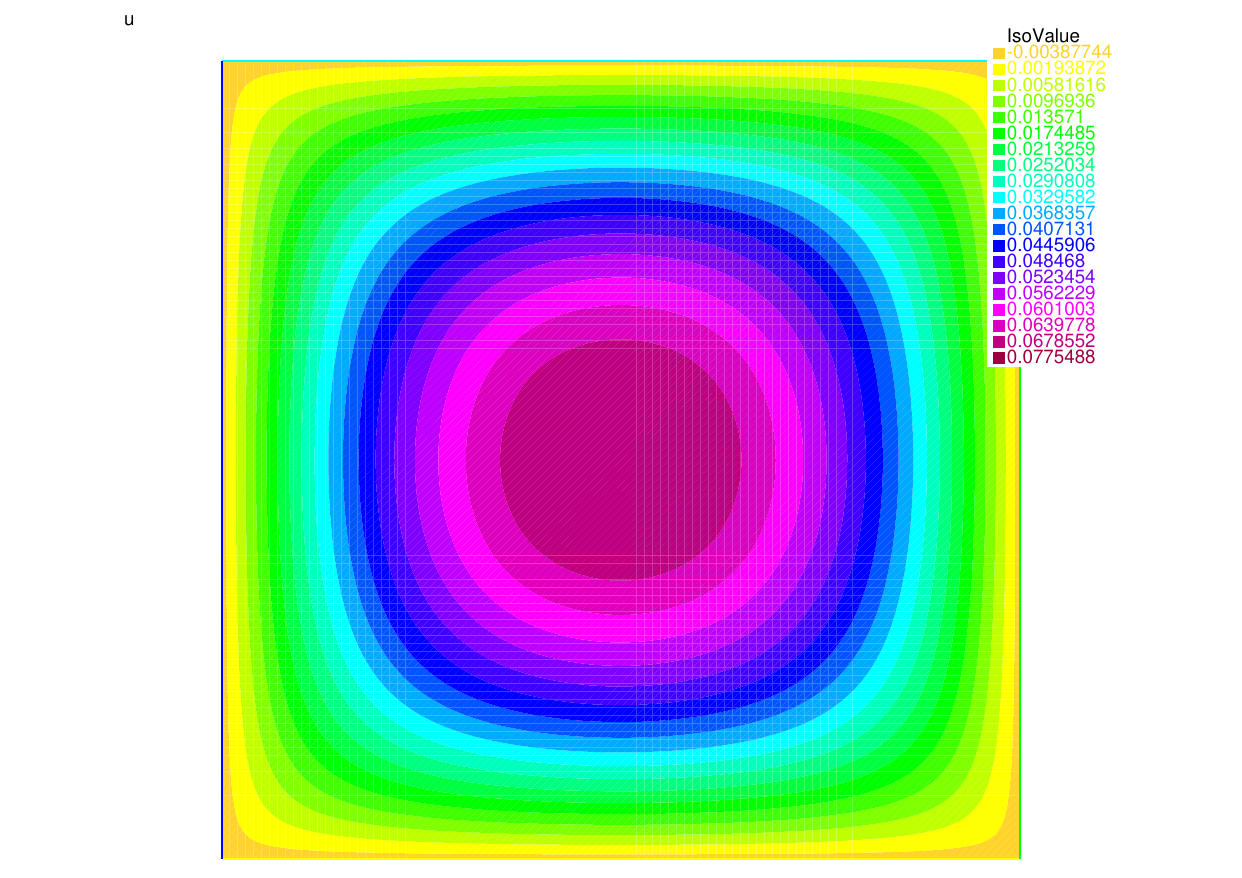}
    \caption{Plot the global solution with \ff{ffddmplot}.}
\end{figure}

The full terminal output of the script is transcribed below:

\begin{ddmoutput}[numbers=none]
[M] Building decomposition from mesh of 20000 elements
[FE] (paraddm) ndof = 40401
[M] timings - local eigenvalue problems : 0.123214 s
[PB] CS size = 114 x 114
[M] timings - building E : 0.028125 s
[PB] It: 1 Residual = 0.0055835 Rel res = 0.970331
[PB] It: 2 Residual = 0.00474896 Rel res = 0.825299
[PB] It: 3 Residual = 0.00290297 Rel res = 0.504493
[PB] It: 4 Residual = 0.00115567 Rel res = 0.200838
[PB] It: 5 Residual = 0.000468384 Rel res = 0.0813982
[PB] It: 6 Residual = 0.000182047 Rel res = 0.031637
[PB] It: 7 Residual = 6.93268e-05 Rel res = 0.012048
[PB] It: 8 Residual = 2.98209e-05 Rel res = 0.00518244
[PB] It: 9 Residual = 9.89744e-06 Rel res = 0.00172003
[PB] It: 10 Residual = 3.4596e-06 Rel res = 0.000601228
[PB] It: 11 Residual = 1.15588e-06 Rel res = 0.000200876
[PB] It: 12 Residual = 4.20447e-07 Rel res = 7.30675e-05
[PB] It: 13 Residual = 1.44859e-07 Rel res = 2.51743e-05
[PB] It: 14 Residual = 4.72674e-08 Rel res = 8.21438e-06
[PB] It: 15 Residual = 1.84267e-08 Rel res = 3.20228e-06
[PB] It: 16 Residual = 7.75534e-09 Rel res = 1.34776e-06
[PB] It: 17 Residual = 3.0426e-09 Rel res = 5.2876e-07
[PB] GMRES has converged in 17 iterations 
[PB] The relative residual is 5.2876e-07
[PB] timings - decomp and partition : 0.075952
[PB] timings - assembly and factorization local matrices : 0.088239
[PB] timings - assembly and factorization coarse problem : 0.028125
[PB] timings - GMRES : 0.117303
[PB] timings - MV : 0.019385
[PB] timings - PREC : 0.044314
[PB] timings - COARSE SOLVE : 0.00196
[PB] timings - total_sum : 0.432833
[PB] timings - total elapsed wall time : 0.386123
\end{ddmoutput}

\newpage
%%%%%%%%%%%%%%%%%%%%%%%% referenc.tex %%%%%%%%%%%%%%%%%%%%%
% sample references
% 
% Use this file as a template for your own input.
%
%%%%%%%%%%%%%%%%%%%%%%%% Springer Nature %%%%%%%%%%%%%%%%%%
%
% BibTeX users please use
% \bibliographystyle{}
% \bibliography{}
%
%\biblstarthook{
\newpage
% \section*{Styling of References}
% References may be \textit{cited} in the text either by number (preferred) or by author/year.\footnote{Make sure that all references from the list are cited in the text. Those not cited should be moved to a separate \textit{Further Reading} section or chapter.} If the citatiion in the text is numbered, the reference list should be arranged in ascending order. If the citation in the text is author/year, the reference list should be \textit{sorted} alphabetically and if there are several works by the same author, the following order should be used:
% \begin{enumerate}
% \item all works by the author alone, ordered chronologically by year of publication
% \item all works by the author with a coauthor, ordered alphabetically by coauthor
% \item all works by the author with several coauthors, ordered chronologically by year of publication.
% \end{enumerate}
% The \textit{styling} of references\footnote{Always use the standard abbreviation of a journal's name according to the ISSN \textit{List of Title Word Abbreviations}, see \url{https://www.issn.org/services/online-services/access-to-the-ltwa/}} depends on the subject of your book:
% \begin{itemize}
% \item The \textit{two} recommended styles for references in books on \textit{mathematical, physical, statistical and computer sciences} are depicted in ~\cite{science-contrib, science-online, science-mono, science-journal, science-DOI}.

% \end{itemize}
%}

\end{document}